\documentclass[a4paper, 11pt]{amsart}

\usepackage[a4paper,top=2cm,bottom=2cm,left=2cm,right=2cm,marginparwidth=1.75cm]{geometry}

\usepackage[english]{babel}

\usepackage{amsmath}
\usepackage{graphicx}
\usepackage[colorlinks=true, allcolors=blue]{hyperref}
\usepackage[T1]{fontenc}
\usepackage[utf8]{inputenc}
\usepackage{amssymb}
\usepackage{enumitem}
\usepackage{stackrel}
\usepackage{amsmath}
\usepackage{amsthm}
\usepackage{enumitem}%
\usepackage{mathrsfs}
\usepackage{blindtext}
\usepackage{graphicx}
\usepackage{url}
\usepackage{wrapfig,lipsum}
\usepackage{mathdots}
\usepackage{tikz-cd}
\usepackage{ amssymb }
\usepackage{wrapfig}
\usepackage{caption}
\usepackage{subcaption}
\usepackage{cite}

\newcommand{\numberset}{\mathbb}
\newcommand{\N}{\numberset{N}}
\newcommand{\Z}{\numberset{Z}}
\newcommand{\R}{\numberset{R}}
\newcommand{\Q}{\numberset{Q}}
\newcommand{\C}{\numberset{C}}

\newcommand{\Ss}{\mathcal{S}}

\newcommand{\Co}{\mathscr{C}}
\newcommand{\bslash}{\backslash}
\newcommand{\Stab}{S^2\times S^2}
\newcommand{\StabTwist}{\C P^2\#\overline{\C P^2}}

\makeatletter\@namedef{subjclassname@2020}{%
\textup{2020} Mathematics Subject Classification}\makeatother

\title{Internal and external stabilization of exotic surfaces in 4-manifolds}
\author{Oliviero Malech}
\address{Scuola Internazionale Superiore di Studi Avanzati (SISSA)\\ Via Bonomea 265\\34136\\Trieste\\Italy}
\email{omalech@sissa.it}
\date{}

\begin{document}
\maketitle
\newtheorem{thm}{Theorem}[subsection]
\newtheorem{mainthm}{Theorem} \renewcommand{\themainthm}{\Alph{mainthm}}
\newtheorem{question}{Question}\renewcommand{\thequestion}{\Alph{question}}
\newtheorem{lem}[thm]{Lemma}
\newtheorem{cor}[thm]{Corollary}
\newtheorem{prop}[thm]{Proposition}
\theoremstyle{definition}
\newtheorem{definition}[thm]{Definition}
\newtheorem{oss}[thm]{Remark}
\newtheorem{constr}[thm]{Construction}

\begin{abstract}
We show that many explicit examples of exotic pairs of surfaces in a smooth 4-manifold become smoothly isotopic after one external stabilization with $\Stab$ or $\StabTwist$. 
  Our results cover surfaces produced by rim-surgery, twist-rim-surgery, annulus rim-surgery, as well as infinite families of nullhomologus surfaces and examples with non-cyclic fundamental group of the complement.  A special attention is given to the identification of the stabilizing manifold and its dependence on the choices in the construction of the surface. The main idea of this note is given by  relating internal and external stabilization, and most of the results, but not all, are proved using this relation. 
  Moreover, we show that the 2-links in the exotic family from the recent work of Bais, Benyahia, Malech and Torres are brunnian, under some additional assumptions about the construction.
\end{abstract}
\tableofcontents

\section{Introduction \label{sec: intro}}
For the sake of simplicity, we use the following definition in this note.
    Given a smooth 4-manifold $X$ (possibly with non-empty boundary) a \emph{surface} $\Sigma\subseteq X$ is a compact and properly embedded smooth submanifold of $X$ of real dimension $2$. Given a pair, or a collection, of smoothly embedded surfaces inside a smooth 4-manifold, it is interesting to ask if they are  isotopic or not. This is the 4-dimensional analogous of knot theory, but in this case the existence of such isotopy can depend on the chosen category, namely topological or smooth.
    We say that a collection of surfaces in a smooth 4-manifold is \emph{exotic} if any distinct pair of elements are topologically ambient isotopic but not smoothly isotopic. If two surfaces are not smoothly isotopic, but they still share similar topological proprieties, it is natural to ask if it is possible to perform any small change to make them smoothly isotopic.  
In \cite{Boyle_1988, InternalStabilizationBaykurSunukian} the notion of \emph{internal stabilization of a surface} is studied, this is the operation of adding a 3-dimensional 1-handle to an embedded surface. If this 1-handle is attached in a small neighborhood of a point of the surface and preserve a local orientation of the surface, then the stabilization is called standard and untwisted.
\begin{definition}Let $X$ be a 4-manifold and $n\in\N$. Two surfaces $\Sigma,\Sigma'\subseteq X$ are \emph{internally $n$-stably isotopic} if they become isotopic after $n$ internal stabilizations each.  They are \emph{trivially} internally $n$-stably isotopic if all the stabilizations are standard and untwisted.\end{definition}
It was proved that any two homologous closed surfaces are internally $n$-stably isotopic for some $n\in\N$ \cite[Theorem 1]{InternalStabilizationBaykurSunukian} and in many cases one standard untwisted stabilization is enough to make exotic surfaces smoothly isotopic \cite[Theorem 2]{InternalStabilizationBaykurSunukian}.
In this note we will study the behavior of embedded surfaces in a smooth 4-manifold $X$ under \emph{external stabilization}, namely the operation of enlarging the ambient space by taking the connected sum with copies of $\Stab$ or $\StabTwist$.
\begin{definition}
Let $B$ be a simply connected 4-manifold. Given two 4-manifolds $X$ and $X'$,
    two surfaces $\Sigma\subseteq X$ and $\Sigma'\subseteq X'$ are (smoothly) \emph{$B$-stably equivalent} if there exists a diffeomorphism of pairs
    \begin{equation}
        \phi:(X\# B,\Sigma)\to (X'\# B,\Sigma'),
    \end{equation}
    where the connected sums are performed away from $\Sigma$ and $\Sigma'$.
    When $X=X'$, we say that $\Sigma,\Sigma'\subseteq X$ are (smoothly) \emph{$B$-stably isotopic} if they are $B$-stably equivalent, the connected sum $X\# B$ is performed away from $\Sigma\cup\Sigma '$ and the map $\phi$ is smoothly isotopic to the identity. In this case we call $B$ \emph{the stabilizing manifold}, or \emph{the stabilizing block}, for the surfaces $\Sigma,\Sigma'\subseteq X$.
    If $B=n(\Stab)$ for a natural number $n\in\N$, then  the notions of $B$-stable equivalence/isotopy coincide with the ones of (not strict) $n$-stable equivalence/isotopy from \cite[Introduction]{Auckly_Kim_Melvin_Ruberman_2015}.
\end{definition}

In \cite{Auckly_Kim_Melvin_Ruberman_2015} and in \cite{akbulut2014isotoping} the authors find  pairs of embedded surfaces in some 4-manifold that are not smoothly isotopic, but they are $(\Stab)$-stably isotopic. In \cite{Auchly1stb} it is shown that any two smoothly embedded closed oriented surfaces with simply connected complements, with the same homology class $\alpha$ and genus, are $B$-stably isotopic, where $B$ is $\StabTwist$ for $\alpha$ characteristic or $\Stab$ for $\alpha$ ordinary (i.e. not characteristic). In this note, we will produce examples of this phenomenon not covered by the mentioned results. In particular, we will focus on explicit examples of exotic surfaces from the literature:  produced via rim-surgery \cite{Fintushel1997SurfacesI4} or twist rim-surgery \cite{Kim_2006,Kim_Ruberman_2008} or annulus rim-surgery \cite{finashin2002knotting}, the nullhomologus 2-tori from \cite{Hoffman_Sunukjian_2020}, the knotted 2-spheres from \cite{torres2023topologically}, the nullhomologus 2-spheres and 2-tori from \cite{torres2020Unknotted} and the 2-links from \cite{(Un)knotted}. We will show that all those constructions are $B$-stable isotopic, see Table \ref{table: surfaces and stabilization}. Moreover, the stabilizing manifold, namely $\Stab$ or $\StabTwist$, depends on the construction, e.g. on an identification of some tubular neighborhood or on the parametrization of a surface. 
\begin{table}[h]
\centering
\begin{tabular}{|l|c|c|c|l}
\cline{1-4}
\begin{tabular}[c]{@{}l@{}}Exotic surfaces\\ \end{tabular}                  & \begin{tabular}[c]{@{}c@{}}Defined \\ in\end{tabular} & \begin{tabular}[c]{@{}c@{}}internally\\ 1-stably isotopic\end{tabular}                                          & \begin{tabular}[c]{@{}c@{}}$B$-stably isotopic,\\ where $B$ is\\ $\Stab$ or $\StabTwist$\end{tabular}                   &  \\ \cline{1-4}
Rim-surgery                                                                &  \cite{Fintushel1997SurfacesI4}                                                     & \begin{tabular}[c]{@{}c@{}}\cite[Section 3.1]{InternalStabilizationBaykurSunukian} 
\end{tabular} & \begin{tabular}[c]{@{}c@{}}\cite{Auchly1stb} \&\\ Theorem \ref{thm: ExtStab RimSurgery}\end{tabular} &  \\ \cline{1-4}
Twist rim-surgery                                                        &     \cite{Kim_2006, Kim_Ruberman_2008}                                               & \cite[Section 3.2]{InternalStabilizationBaykurSunukian}                                                                                & Theorem  \ref{thm: External Twist Rim Surgery}                                                                      &  \\ \cline{1-4}
Annulus rim-surgery                                                        &     \cite{finashin2002knotting}                                               & \cite[Section 3.3]{InternalStabilizationBaykurSunukian}                                                                                & Theorem  \ref{thm: external annulus Surgery}                                                                      &  \\ \cline{1-4}

Nullhomologus 2-tori                                                       &    \cite{Hoffman_Sunukjian_2020}                                                   & \cite[Section 3.6]{InternalStabilizationBaykurSunukian}                                                                                   & Theorem   \ref{thm: ExtStab nullhomologus 2-tori}                                                                       &  \\ \cline{1-4}
Knotted 2-spheres                                                 &   \cite{Auckly_Kim_Melvin_Ruberman_2015}                                                    & \cite[Section 3.7]{InternalStabilizationBaykurSunukian}                                                                    & \cite[Theorem A]{Auckly_Kim_Melvin_Ruberman_2015} 
&  \\ \cline{1-4}

\begin{tabular}[c]{@{}l@{}}Knotted 2-spheres $\Z/k$
\end{tabular} & \cite{torres2023topologically}                                                       & Unknown                                                                                                     & \begin{tabular}[c]{@{}l@{}}Theorem \ref{thm: Stable isotopy for TORRES Z_2}
\end{tabular}                                                                          &  \\ \cline{1-4}
Nullhomologus 2-sphere                                                    &   \cite{fintushel1994fake}                                                   & Unknown                                                  & Theorem   \ref{thm: External fintushel sphere}
 \\ \cline{1-4}
Nullhomologus 2-spheres                                                   &   \cite{torres2020Unknotted}                                                   & Theorem    \ref{thm: intern unknotte 2spheres torres}                                                                   & Theorem   \ref{thm: externa stab 2-sphers (un)knotted torres}
 \\ \cline{1-4}
Brunnianly exotic 2-links                                                    &   \cite{(Un)knotted}                                                    & Theorem      \ref{thm: (un)linked stabilization and bruniannity}                                                                    & Theorem   \ref{thm: Int unknotted 1} 
&  \\ \cline{1-4}

\end{tabular}
   
\caption{\label{table: surfaces and stabilization} Surfaces and the corresponding internal and external stabilization results. }
\end{table}
\subsection{Main results}
The main results of this note appear in Table \ref{table: surfaces and stabilization}, and we refer to the body of this work for more precise statements.
Table \ref{table: surfaces and stabilization} contains the results on internal/external stable isotopy from this note and the best  corresponding internal/external results from the literature. All the theorems in the first four lines in the right column are direct applications of the corresponding results named in the third column in combination with Theorem \ref{mainthm: int to ext}, which we now present.

Notice that internal stabilization only changes the topology of the surface in the original ambient space, while external stabilization only changes the topology of the ambient space.
One of the main contents of this note is given by a relation between internal and external stabilization.
In particular, we will introduce the notion of \emph{ internally $\gamma$-stably isotopic}  (Definition \ref{def: internal stab across}) for a pair of surfaces $\Sigma_1,\Sigma_2\subseteq X$ that coincide on a tubular neighborhood of a loop $\gamma\subseteq X$, surfaces which are called $\nu\gamma$-standards (Definition \ref{def: gamma-standarad}). 
This is  a stronger notion than the usual notions of internal and external stabilization, as asserted in the following theorem.
\begin{mainthm}\label{mainthm: int to ext}
    Let $X$ be a smooth 4-manifold, $\gamma$ a simple loop in $X$ with a chosen framing. Assume that $\Sigma_1,\Sigma_2\subseteq X$ are $ \nu\gamma$-standard surfaces.
If the surfaces $\Sigma_1,\Sigma_2\subseteq X$ are internally $\gamma$-stably isotopic,
then 
\begin{itemize}
\item $\Sigma_1,\Sigma_2\subseteq X$ are trivially internally 1-stably isotopic 
    \item $\Sigma_1,\Sigma_2\subseteq X$ are
$B$-stably isotopic for any $B\in \overline{\Ss}(\Sigma_1,\Sigma_2,\gamma)\subseteq\{\Stab,\StabTwist\}$, where $\overline{\Ss}(\Sigma_1,\Sigma_2,\gamma)$ is the extended stabilization set (Definition \ref{def: Ss}).
\end{itemize}
 Moreover, the statement still holds if we require isotopies to be relative to the boundary.\\
 The set $\overline{\Ss}(\Sigma_1,\Sigma_2,\gamma)$ is not empty if and only if there exists a nullhomotopic push off of $\gamma$ in $X\bslash(\Sigma_1\cup\Sigma_2)$.
\end{mainthm}
The proof of Theorem \ref{mainthm: int to ext} is given in Section \ref{sec: thm A}.
  The results and techniques developed in \cite{InternalStabilizationBaykurSunukian} can be adapted to prove that surfaces in many examples are internally $\gamma$-stably isotopic. So, Theorem \ref{mainthm: int to ext}  will allow us to translate results on internal stabilizations into the external stabilization realm.  
 The study of the extended stabilization set, its dependence on the surfaces and on the curve $\gamma$, will be the main focus of Section \ref{sec: StabilizationSet}.
To the author's knowledge, the only other relation between internal and external stabilization is presented in \cite[Lemma 3.2]{hayden2023stabilization}, while in \cite{InternalStabilizationBaykurSunukian} a relation between internal stabilization of surfaces and external stabilization of the manifolds obtained as branched double cover of the surfaces is described.


    In Section \ref{sec: Rim Surgery}, we will consider  surfaces $\Sigma_{K,\beta}\subseteq X$ obtained by rim-surgery \cite{Fintushel1997SurfacesI4} from a closed, connected, oriented surface $\Sigma\subseteq X$ with simply connected complement and defined by a knot $K\subseteq S^3$ and $\beta$ a simple loop in $\Sigma$ (see Construction \ref{const: Rim surgery} for more details).
    By \cite[Theorem A]{Kim_2019} the surfaces $\Sigma,\Sigma_{K,\beta}\subseteq X$  are $(\StabTwist)$-stably \emph{equivalent}.
    By a direct application of  \cite{Auchly1stb} the surfaces $\Sigma,\Sigma_{K,\beta}\subseteq X$ are $B_\Sigma$-stably isotopic, where the manifold $B_\Sigma$ depends only on the homology class of the surface $\Sigma$. Explicitly 
    \[B_\Sigma=\begin{cases}\Stab, &\text{ if $[\Sigma]$ is ordinary}\\
    \StabTwist, &\text{ if $[\Sigma]$ is characteristic}
    \end{cases}.\] 
    On the other side, our Theorem \ref{thm: ExtStab RimSurgery}
    shows that any manifold $B$ in \emph{the extended stabilization set} $ \overline{\Ss}(\Sigma,\beta)\subseteq\{\Stab,\StabTwist\}$ (defined in Section \ref{sec: StabilizationSet}) is a stabilizing manifold for the surfaces $\Sigma,\Sigma_{K,\beta}\subseteq X$. In this case $\overline{\Ss}(\Sigma,\beta)$ can be expressed as follows:

    \[\overline{\Ss}(\Sigma,\beta)=\begin{cases}\{\Stab,\StabTwist\}, &\text{ if $[\Sigma]$ is ordinary}\\
    \{\Stab\}\text{ or }\{\StabTwist\}, &\text{ if $[\Sigma]$ is characteristic}
    \end{cases}\] 
    Moreover, Proposition \ref{prop: Propriety of S} (points \ref{prop: SS item 4} and \ref{prop: SS item 5}) states that different curves $\beta\subseteq\Sigma$ can give different stabilizing sets $\overline{\Ss}(\Sigma,\beta)$ if the homology class $[\Sigma]$ is characteristic.

In Section \ref{sec: twistRim}, we present a similar result for surfaces obtained by twist rim-surgery \cite{Kim_2006} (Theorem \ref{thm: external Twist Rim Surgery}) and two explicit application are presented in Section \ref{sec: twist rim application}: this includes the study of an exotic collection of surfaces in $B^4$ from \cite{MillerZenke2021transverse} and one with non-cyclic fundamental group of the complement from \cite[Theorem B]{hayden2021exotically}.

In Section \ref{sec: AnnulusRim}, we show that also surfaces produced by annulus rim-surgery \cite{finashin2002knotting} are $\Stab$-stably or $\StabTwist$-stably isotopic (Theorem \ref{thm: external annulus Surgery}).

In Section \ref{sec: nullhom 2tori} we study the nullhomologus 2-tori $T_{K}\subseteq X$ from \cite{Hoffman_Sunukjian_2020}, which   are surfaces topologically unknotted and smoothly knotted, i.e. they bound a topologically flat embedded, but not a smoothly embedded, $S^1\times D^2$. For these surfaces the subscript $K$ denotes a knot in $S^3$ used in their constructions.
One external stabilization  with $\Stab$ or $\StabTwist$ makes the nullhomologus surfaces \emph{smoothly} unknotted (Theorem \ref{thm: ExtStab nullhomologus 2-tori}). It is an open question if the stabilizing manifold can always be taken to be $\Stab$.
\begin{question}\label{question: external nullhomo with S2xS2}
    Is $T_{K}\subseteq X$, defined as in Construction \ref{const: nullhomologus tori}, smoothly unknotted in $X\#(\Stab)$?
\end{question}

In Section \ref{sec: externalPrescribedGroup}, we study collection of surfaces produced by \emph{mixing} (Definition \ref{def: mixing}), motivated by the construction of an exotic set of 2-spheres presented in the proof of \cite[Theorem B]{torres2023topologically}. 
In general, for those collections we are able to prove only that they are $(\Stab)$-stably \emph{equivalent} (see Proposition \ref{prop: stable equivalence}).
Here the problem lies in the fact that very little is known about the internal stabilization behavior of these collections
, and so we are not able to apply Theorem \ref{mainthm: int to ext}.
\begin{question}\label{question: internal stabilization of mixing}
    Given a pair of surfaces $\Sigma,\Sigma'\subseteq X\#(\Stab)$, where $\Sigma'$ is a mixing of $\Sigma$, what is the minimum number of internal or external stabilization needed to make them smoothly isotopic?\end{question}
Instead, we base our work on \cite{Auchly1stb} and we show that surfaces produced by a \emph{satellite operation} (Section \ref{sec: satelliteoperations}) around 2-spheres with simply connected complements are externally $(\Stab)$-stably isotopic or  $(\StabTwist)$-stably isotopic (Theorem \ref{thm: patterns}). 
As an application, some collection from \cite[Theorem B]{torres2023topologically}, the ones with fundamental group $\Z/k\Z$ (constructed from the $\Q$-homology 4-sphere $\Gamma_{k,0}$), are $(\Stab)$-stably \emph{isotopic} (Theorem \ref{thm: Stable isotopy for TORRES Z_2}).
As a consequence, we construct an exotic infinite family of 2-spheres in a simply connected 4-manifold such that the collections produced by satellite operations, for an infinite family of patterns, are also exotic collections (Corollary \ref{cor: torres patterns}). 
\begin{question}\label{question: pattern}
    Does performing non-trivial satellite operations on a given collection of exotic 2-spheres preserve exoticness?
More precisely, given an exotic collection $\Co$ of 2-spheres in a 4-manifold $X$, a 2-sphere pattern $P$ (not unknotted), and define the collection 
\[\Co(P):=\{S(P): S(P)\text{ is the satellite surface of $S\in\Co$ with pattern }P\}.\]
Is $\Co(P)$ an exotic collection?
\end{question}
Corollary \ref{cor: torres patterns} shows that the preceding question has positive answer for a specific exotic collection $\Co$ and patterns $P(k)$ for $k\in\N$.
Notice that we have excluded patterns $P$ with positive genus, since the work of \cite{InternalStabilizationBaykurSunukian} will give a negative answer to our question in many cases.

Section \ref{sec: nullhomologus surfaces} is dedicated to the study of other topologically unknotted surfaces. In Section \ref{sec: stabilization FintStern} we study the 2-sphere implicitly described by \cite{fintushel1994fake} as pointed out in \cite[end of Section 2]{ray2017four} (see Construction \ref{Constr fs: 2-sphere}) and prove that it is smoothly unknotted after one stabilization with $\Stab$ (Theorem \ref{thm: External fintushel sphere}), while it remains unknown if it smoothly unknotted after one internal stabilization. This result can be considered as a refinement and improvement of Matumoto's work \cite[Theorem 1]{OnWeaklyUnknotted2Sphere}.
In Section \ref{sec: Stab Torres unknotted} we show that the 2-spheres and 2-tori from \cite[Theorem A]{torres2020Unknotted} become smoothly unknotted after internal or external stabilization, see Theorems \ref{thm: externa stab 2-sphers (un)knotted torres}, \ref{thm: identification Torres 2-tori}, \ref{thm: intern unknotte 2spheres torres}.
For proving the results of Section \ref{sec: nullhomologus surfaces}, we will use  techniques which have a more general application in the context of 2-links, which are studied in  Section \ref{sec: unlinked}.

In Section \ref{sec: unlinked}, we study the 2-links $\Gamma_K$ from \cite[Theorem A]{(Un)knotted}, the one component case of those 2-links gives topologically unknotted 2-spheres $S_K$, where $K\subseteq S^3$ is a knot.   A consequence of Section \ref{sec: nullhom 2tori} is that the 2-sphere $S_K$ after one internal stabilization is smoothly unknotted in many cases (Theorem \ref{thm: Int unknotted 1}). Indeed, we show that, after the attachment a 3-dimensional 1-handle, the 2-sphere $S_K$ becomes a nullhomologus 2-torus of the form $T_{3_1}\subseteq X\bslash D^4\subseteq X\#(\Stab)$, where $3_1$ denotes the trefoil knot. This is the only instance in this note in which we translate a property of external stabilization to another one of internal stabilization. Section \ref{sec: strategy int unknotted} presents the strategy for the proof for this result based on following and identifying surfaces through a sequence of diffeormorphisms.
\begin{definition}
A $g$-component 2-link $\Gamma$ in a smooth 4-manifold  $X$ \[\Gamma=S_1\sqcup\dots\sqcup S_g\subseteq X\] is the disjoint union of $g$ smoothly embedded 2-spheres. It is \emph{brunnian} if for any $j=1,\dots,g$ the $g-1$ component 2-link $\Gamma\bslash S_j\subseteq X$ is smoothly unlinked, i.e. it bounds  $g-1$ disjoint and smoothly embedded copies of $D^3$ in $X$.     
\end{definition}
This is the generalization of the notion introduced by Brunn  \cite{brunn1892ueber} for links in 3-dimensional space.
In Theorem \ref{thm: (un)linked stabilization and bruniannity} and Corollary \ref{cor: unliked family bruniannity}, we show that the 2-links $\Gamma_K$, under some additional assumptions, are brunnian, and they become smoothly unlinked after one external stabilization with $\Stab$.  This partially answers \cite[Question D]{(Un)knotted} and it should be compared with previous works of \cite[Theorem B]{Auckly_Kim_Melvin_Ruberman_2015} and \cite{hayden2021brunnian}. Both articles produce examples of brunnian exotic families, where different notions of brunnianity are used. In the first one \cite{Auckly_Kim_Melvin_Ruberman_2015} they have a family of \emph{ordered} 2-links, pairwise non-isotopic, becoming isotopic after removing the \emph{same} component, in particular, this is a propriety of the family and not of the surface. In the second one \cite{hayden2021brunnian} they use a notion of brunnianity for surfaces with boundary. Our result, restricted to the case of the 2-sphere $S_K$, has to be compared with \cite[Theorem 1]{OnWeaklyUnknotted2Sphere}.
The strategy for obtaining this result is explained in details in Section \ref{sec: strategy external unlinked}. It is based on a cobordism argument, presented in Section \ref{Cobordism argument} and inspired by \cite{DissolvingKnotSurgery}, which allows us to identify the diffeomorphism type of 4-manifold which is the main obstruction preventing the 2-link $\Gamma_K$ from being smoothly unlinked.

In (\ref{eq: logical dependence of the sections})  we present a logical dependency scheme of the sections of this manuscript, we hope that the reader interested in studying a specific example can benefit from it.
\begin{equation}\label{eq: logical dependence of the sections}    \begin{tikzcd}\S\ref{sec: intro}\arrow{d}\arrow{r} & \S\ref{sec: Internal->external}\arrow{r}\arrow[swap]{d} & \S\ref{sec: application}  \\\S\ref{sec: externalPrescribedGroup} & \S\ref{sec: unlinked}\arrow{r} & \S\ref{sec: nullhomologus surfaces}\end{tikzcd}\end{equation} 

\subsection{Choice of hypotheses in the main theorems}
In any construction of an exotic pair $\Sigma,\Sigma'\subseteq X$ of surfaces in a smooth 4-manifold $X$, we have to take care of the following hypothesis and choices:
\begin{enumerate}
    \item Minimal ingredients needed for the construction of the embedded surface $\Sigma'\subseteq X$; usually these are mainly hypothesis on the ambient space $X$ and an initial surface $\Sigma\subseteq X$.
    \item  Hypothesis needed to make the homology classes $[\Sigma],[\Sigma']\in H_2(X;\Z)$ coincide and  to have an isomorphism of fundamental groups of the complements $X\bslash \Sigma$ and $X\bslash\Sigma'$. 
    \item Hypothesis needed to prove that the surfaces $\Sigma$ and $\Sigma'$ are not smoothly isotopic.
\end{enumerate}
In all the construction that we study in this note, we omit the third set of hypothesis, since those are not needed in our discussion on (not strict) stabilizations. Instead, those represent obstacles that often can shade the results.
So, we do not exclude the case in which the surfaces $\Sigma,\Sigma'\subseteq X$ are already smoothly isotopic, even if our results will be trivial in this case. 
In this way we avoid confusion, and we get a much general statement that works also for non-isotopic surfaces possibly constructed without satisfying the third set of hypothesis. 
On the other hand, we can not omit the first two sets, since the second one contains the minimal hypothesis to not obstruct the existence of an isotopy between $\Sigma $ and $\Sigma'$ in $X\# B$ for a simply connected $B$.
Unfortunately (?) sometimes our result will  need extra assumptions to work, we have done the best of our possibility to reduce them and state them clearly in the corresponding theorems. Notice that there is no known example of a pair of closed connected surfaces that does not become smoothly isotopic after one $\Stab$ stabilization in a simply connected and \emph{closed} 4-manifold. In \cite{hayden2023stabilization}  a counterexample in an open 4-manifold is presented.

\subsection{Open questions about surfaces stabilization}
There are many other constructions of exotic surfaces that this work has not analyzed:
\begin{itemize}
    \item a multitude of exotic surfaces with boundary in $D^4$, some of which are not isotopic after one or many internal stabilizations (see for example \cite{guth2022exotic});
    \item a multitude of exotic non orientable surfaces, including infinite collections of $k\R P^2$ embedded in $S^4$ for different values of $k$ (see for example \cite{FinashinKreckViro} for $k=10$, or \cite{finashin6Rp2} for $k=6$,
    or \cite{matic2024exotic5rp24sphere}
    for $k=5$, or \cite{miyazawa2023gauge} for $k=1$);
    \item exotic closed surfaces with more complicated fundamental group of the complement (see for example \cite[Theorem B]{torres2023topologically} or \cite[Theorem 1.4]{benyahia2024exotically}).
\end{itemize}  
Are these exotic surfaces $\Stab$ or $(\StabTwist)$-stably isotopic?
    In \cite{auckly2023smoothlyknottedsurfacesremain} the author produces pairs of \emph{embeddings} of surfaces that are not smoothly isotopic even after many internal stabilizations, but they become isotopic after one external stabilization with $\Stab$. Is it possible to compute the number of internal stabilizations needed to make them smoothly isotopic?


\subsection{Acknowledgements}
I would like to thank Rafael Torres for suggesting this project and supporting its development with numerous bibliographical suggestions. 
I would like to thank him as well as Valentina Bais and Younes Benyahia for useful discussions and comments.
Remark \ref{rem: sigma is smoothly knotted} and its proof are due to Rafael Torres.
I thank Dave Auckly for pointing out Remark \ref{rem: Rocklin}.
This work was supported by the "National Group for Algebraic and Geometric Structures, and their Applications" (GNSAGA - INdAM).

\subsection{Notation and conventions}
The standard basis of the Euclidean space $\R^n$ will be denoted as $e_1,\dots,e_n$, the n-sphere $S^n$ and the n-disk $D^n$ are subsets of the Euclidean space.
We can write $D^n\subseteq D^{n+k}$ (and similarly for the $S^n\subseteq S^{n+k}$) for all $n,k\in\N$ using the standard identification of $\R^{n}$ with  $\R^{n}\times \{0\}^k\subseteq\R^{n+k}$. In particular, $\pm 1=\pm e_1\in S^n$ for all $n$. The n-torus $T^n$ is the product of n copies of $S^1$.\\
All the 4-manifolds in this note are supposed to be smooth, connected, oriented and they can have boundary.  All submanifolds are smoothly embedded. 
Every homeomorphism or diffeomorphism between oriented manifolds (or submanifolds) is supposed to be orientation preserving if not specified, except for gluing maps. 

\subsubsection{Tubular neighborhoods and framings}
We denote as $\nu Y\subseteq X$ the tubular neighborhood of $Y$ in $X$, where $Y$ is a compact and properly embedded submanifold. Moreover, each submanifold will have a chosen tubular neighborhood with closure defining a disk bundle over $Y$ and, if possible, in the interior of $X$. Sometimes we will abuse of notation denote $\nu Y$ also the closure of the tubular neighborhood. 
Similarly, we define  $\nu_\epsilon Y\subseteq \nu Y$ for  $0<\epsilon<1$, where $\nu_\epsilon Y$ is the restriction to the vectors of norm less than $\epsilon$ (after choosing a disk bundle structure).\\
If $Y$ is a submanifold of codimension $k$ with trivial tubular neighborhood, then we say that a framing for $\nu Y$ is a diffeomorphism 
\[f_Y:\nu Y\to Y\times D^k,\] where  the restriction of $f$ to $Y$ is $f|_Y=id_Y\times \{0\}$.
If the ambient space $X$ is oriented, then we are assuming the $f_Y$ is orientation preserving.
We say that  $Y$ is a \emph{framed submanifold} of $X$, if it comes equipped  with a chosen tubular neighborhood and a framing $(\nu Y,f_Y)$. 
Two framed submanifold $Y\subseteq X$ and $Y'\subseteq X'$ are (smoothly) equivalent if there exists a diffeomorphism $\phi:X\to X'$ such that \[\phi(Y)=Y', \quad\phi(\nu Y)=\nu Y'\quad\text{ and }\quad(\phi|_Y\times id_{D^k})\circ f_Y=f_{Y'}\circ\phi|_{\nu Y}.\]
If $X'=X$ we say that the framed submanifold $Y$ and $Y'$ are smoothly isotopic if $\phi$ is also ambient isotopic to the identity. If $\phi:(X,Y)\to(X',Y')$ is a diffeomorphism of pair and $Y$ is a framed submanifold, then $Y'$ has a \emph{framing induced by the map} $\phi$, namely this is given by $\nu Y':=\phi(\nu Y)$ and \[f_{Y'}:=(\phi|_Y\times id_{D^k})\circ f_Y\circ\phi^{-1}|_{\nu Y'}:\nu Y'\to Y'\times D^k.\]

If $l\subseteq X$ is a simple loop, then usually it comes equipped with an orientation, so a framing for an oriented loop is a diffeomorphism 
\[f_l:\nu l\to l\times D^k\cong S^1\times D^k\]
which respects the orientation of the loop. Notice that a simple loop can be framed if and only if it has an orientable normal bundle, i.e. if it is not an orientation reversing loop of the manifold.

\section{\label{sec: Internal->external}From internal to external stabilizations}
In this section we will uncover a simple relationship between internal and external stabilization of a surface in a 4-manifold. We will be able to find conditions under which an internal stabilization between two surfaces implies that the surfaces become smoothly equivalent after an external stabilization. The main result is Theorem \ref{mainthm: int to ext}.
In section \ref{sec: main idea}  the main idea and the construction are presented. In section \ref{sec: StabilizationSet} we will define the stabilization sets and prove Theorem \ref{mainthm: int to ext}. The following subsections examine the stabilization set in detail.

\subsection{The main idea \label{sec: main idea}}
We begin this section by giving an informal idea of what we are going to do.

Let $X$ be a 4-manifold and $\Sigma\subseteq X$ a surface containing a simple loop $\gamma\subseteq\Sigma$. Suppose that $\gamma$ bounds a \emph{compressing 2-disk} $D$ with interior disjoint from $\Sigma$, see Figure \ref{fig: main idea} in the margin of this page.
\begin{wrapfigure}{r}{5cm}
\centering
\includegraphics[width=5cm]{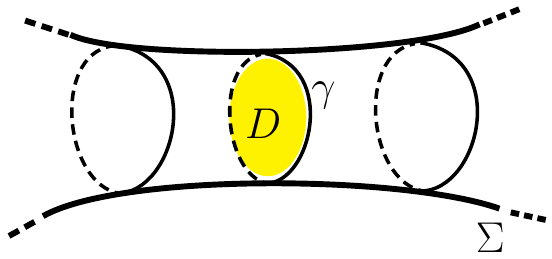}\\
\includegraphics[width=5cm]{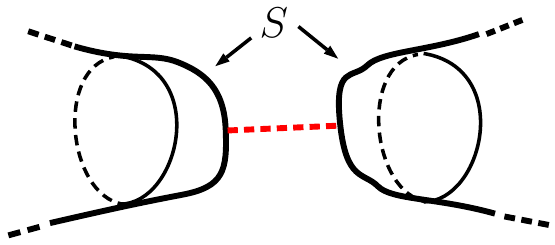}
\caption{\label{fig: main idea}}
\end{wrapfigure} The formal notion of \emph{compressing 2-disk} is given in \cite[Definition 9.1]{4DLightBulb}. We can define a new surface $S$ by removing the annulus $\nu \gamma\cap \Sigma$ and capping off the two new boundary components with two 2-disks, which are parallel copies of $D$. We say that $S$ is obtained from $\Sigma$ by compressing along $D$, see \cite[Definition 9.4]{4DLightBulb} and the bottom figure on the right.
Furthermore, we can see that $\Sigma$ can be obtained from $S$ as result of an internal stabilization along an arc connecting the added two 2-disks; this arc is shown dashed in the figure.  The arc that defines an internal stabilization is sometimes called \emph{cord}. We use \cite{InternalStabilizationBaykurSunukian} and \cite{Boyle_1988} as main references for the notion of \emph{internal stabilization}.
In short, if a surface $\Sigma$ has a compressing 2-disk, then it can be considered as the result of an internal stabilization of another surface $S$. The existence of an embedded compressing 2-disk in $X$ bounding a curve $\gamma\subseteq\Sigma$, in general, is not guaranteed. To solve this issue, we will consider a curve $\bar\gamma$ \emph{a small push off} of the curve $\gamma$ into the complement of $\Sigma$ and we will perform loop surgery on this curve. This operation will produce a new ambient manifold $X^*$ in which $\Sigma$ still defines an embedded surface that we will call $\Sigma^*$ to avoid confusion. The framing of $\bar\gamma$ is chosen in such a way that $\Sigma^*$ has a compressing 2-disk in $X^*$. 
Proposition \ref{Prop: Relation External-Internal} will encode this information.\\ 
Moreover, note that $X^*$ is diffeomorphic to $X\# B$ for $B\in\{\Stab,\StabTwist\}$ if $\bar\gamma$ is nullhomotopic in $X$, and we can control the image of $\Sigma^*$ to be $\Sigma$ under the diffeomorphism $X^*\to X\# B$ if $\bar\gamma$ is nullhomotopic in $X\bslash\Sigma$ (Remark \ref{rem: loop is connected sum}).
We would like to use what discussed until this point to derive properties of a pair  $\Sigma_1,\Sigma_2$ of surfaces in $X$. To do this, we have to suppose that they agree at least on a neighborhood of a loop $\gamma\subseteq\Sigma_1\cap\Sigma_2$. This justifies Definition \ref{def: gamma-standarad}, and it allows us to define both $\Sigma_1^*$ and $\Sigma_2^*$ in $X^*$.
In Proposition \ref{Prop Extern by internal} will be shown that $\Sigma_1^*$ and $\Sigma_2^*$ are smoothly isotopic in $X^*$ under the assumption that $\Sigma_1$ and $\Sigma_2$ become smoothly isotopic in $X$ after one internal stabilization and the natural condition that the isotopy has to preserve a small tubular neighborhood $\nu_{1/2}\gamma$ of $\gamma$. The proof of Proposition \ref{Prop Extern by internal} will be based on the fact that $\Sigma_1^*$ and $\Sigma_2^*$ are internal stabilizations of surfaces $S_1$ and $S_2$, which coincide with $\Sigma_1$ and $\Sigma_2$ in $X\bslash(\nu_{1/2}\gamma)\subseteq X^*$ and coincide with each other on the remaining $(\nu_{1/2}\gamma)^*\subseteq X^*$.

\begin{definition}\label{def: gamma-standarad}
    Let $\gamma\subseteq X$ be a  simple loop in an oriented 4-manifold $X$. Fix for $\gamma$ a tubular neighborhood $\nu\gamma\subseteq int(X)\subseteq X$ and a framing $f:\nu\gamma\to S^1\times D^3$.\\
    We say that a surface $\Sigma\subseteq X$ is $\nu\gamma$\emph{-standard} if $f$ induces a diffeomorphism of couples
\[f:(\nu \gamma,\nu \gamma\cap\Sigma)\to (S^1\times D^3, S^1\times D^1).\]
The curve $\overline{\gamma}=f^{-1}(S^1\times \frac{1}{4}e_2)\subseteq\nu_{1/2}\gamma\subseteq\nu\gamma$ is called \emph{a small push off} of $\gamma$, it comes equipped with a tubular neighborhood 
\[\nu\overline{\gamma}:=f^{-1}\left(S^1\times (\frac{1}{8} D^3+\frac{1}{4}e_2)\right)\subseteq\nu_{1/2}\gamma\subseteq\nu\gamma\subseteq X\]
and a framing $\overline{f}$ induced by a framing $f$ of $\gamma$, i.e. \[\overline{f}:=(id_{S^1}\times h)\circ f|_{\nu\overline{\gamma}}:\nu\overline{\gamma}\to S^1\times D^3,\] where $h:(\frac{1}{8} D^3+\frac{1}{4}e_2)\to D^3$ is the affine transformation $h(x)=8\cdot(x-\frac{1}{4} e_2)$. See Figure \ref{fig: main idea surgery A}.
\end{definition}

\begin{constr}\label{constr: Sigma* S}
 Let $X^*:=(X\bslash\nu\overline{\gamma})\cup_{\overline{f}|_\partial} (D^2\times S^2)$ be the 4-manifold which is the result of loop surgery along $\overline{\gamma}$ and notice that the surface $\Sigma$ still define a surface $\Sigma^*$ in $X^*$. We can isotope the arc $ \frac{1}{2}D^1\subseteq \frac{1}{2}D^3$ fixing the boundary to make an arc $A\subseteq\frac{1}{2} D^3$ with a small part of it in $(\frac{1}{8} D^3+\frac{1}{4} e_2)$, see Figure \ref{fig: main idea surgery A}. 
 We can suppose that $A$ is chosen is such a way that \[((S^1\times A)\cap f(\partial\nu\overline{\gamma}))=S^1\times\{\pm\frac{1}{8}e_1+\frac{1}{4}e_2\}\subseteq f( \partial\nu\overline{\gamma})\subseteq S^1\times D^3 \] Now we define a surface $S\subseteq X^*$ as follows
\[S:=\left((\Sigma\bslash\nu_{1/2}\gamma)\cup (f^{-1}(S^1\times A)\bslash\nu\overline{\gamma})\right)\cup \left(D^2\times\{\pm1\}\right)\subseteq (X\bslash\nu\overline{\gamma})\cup (D^2\times S^2)=X^*\]
We call $S$ \emph{the surface in $X^*$ obtained from $\Sigma^*$ by compressing the loop $\gamma$}.    
\end{constr}

\begin{figure}
\begin{subfigure}[b]{0.45\textwidth}
    \centering
    \includegraphics[width=0.55\textwidth]{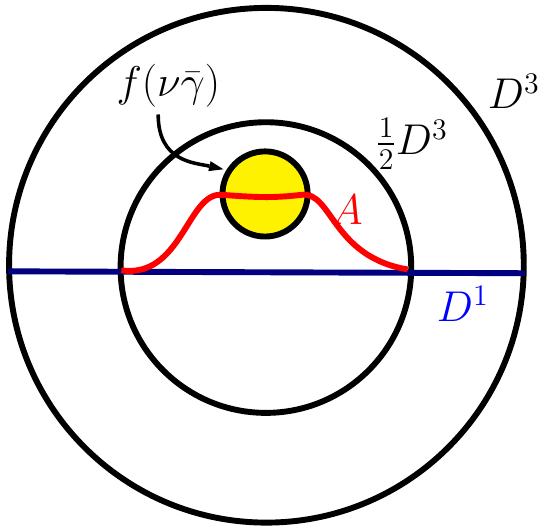}
    \caption{\label{fig: main idea surgery A}}
    
\end{subfigure}
\begin{subfigure}[b]{0.45\textwidth}
    \centering
    \includegraphics[width=0.55\textwidth]{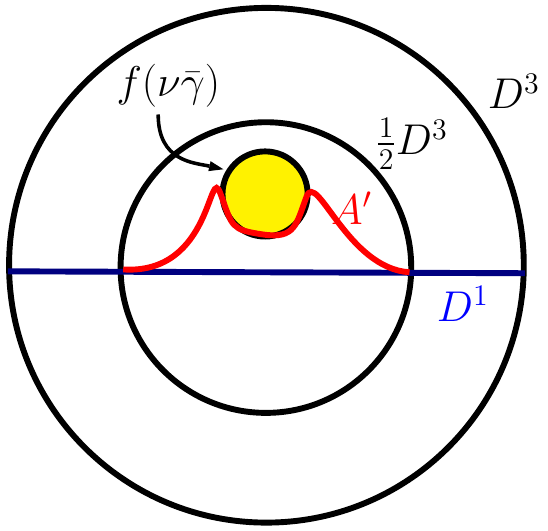}
    \caption{\label{fig: Main idea surgery B}}
\end{subfigure}
\caption{}
    
\end{figure}
\begin{oss}\label{rem: nugamma*}
    Notice that the portion of the surface $S$ lying in $(\nu_{1/2}\gamma)^*:=(\nu_{1/2}\gamma\bslash\nu\overline{\gamma})\cup_{\overline{ f}|_\partial} (D^2\times S^2)\subseteq X^*$ depends only on the framed curve $(\gamma,\nu\gamma, f)$, and does not depend on the surface $\Sigma$.
\end{oss}
\begin{oss}\label{oss: Compressing disk}
    The construction of $S$ can be rephrased as follows. The surface $S$ is   obtained from $\Sigma^*$ by compressing a 2-disk $D_{\gamma}\subseteq(\nu_{1/2}\gamma)^*$ bounding $\gamma\subseteq\Sigma^*$. Indeed, there is an embedded 3-dimensional 2-handle $H_\gamma\subseteq(\nu_{1/2}\gamma)^*$ attached to $\Sigma^*$, such that \[\partial H_\gamma=f^{-1}(S^1\times\frac{1}{2}D^1)\cup \left(\left(f^{-1}(S^1\times A)\bslash\nu\overline{\gamma}\right)\cup (D^2\times\{\pm1\})\right)\subseteq(\nu_{1/2}\gamma)^*\]
    To describe the 2-handle $H_\gamma$, let 
    $Q$ be the 2-dimensional region bounded by the curve $A'\cup\frac{1}{2} D^1\subseteq \frac{1}{2}D^2\subseteq \frac{1}{2}D^3$, where the arc $A'$ is the one depicted in Figure \ref{fig: Main idea surgery B}. Define $H_\gamma$ as
    \[ H_\gamma:=(f^{-1}(S^1\times Q))\cup (D^2\times D^1_-)\subseteq (\nu_{1/2}\gamma\bslash\nu\overline{\gamma})\cup (D^2\times S^2)=(\nu_{1/2}\gamma)^*, \] where $D^1_-$ is the arc in  $S^2\cap (\R\times\R^-\times\{0\})\subseteq S^2$ with endpoints $\pm1\in S^1\subseteq S^2$. The compressing 2-disk $D_\gamma$ is the core 2-disk of $H_\gamma$.
\end{oss}

The next proposition states a useful relation between the surfaces $S$ and $\Sigma^*$.

\begin{prop}\label{Prop: Relation External-Internal}
   Let $X$ be a 4-manifold and $\gamma\subseteq X$ a simple framed loop. Assume that $\Sigma$ is $\nu\gamma$-standard surface in $X$ and that $\overline{\gamma}$ is a small push off of $\gamma$.
Let $S,\Sigma^*\subseteq X^*=(X\bslash \nu\bar\gamma)\cup(D^2\times S^2)$ be as Construction \ref{constr: Sigma* S}.
Consider the arc \[\alpha:=\{1\}\times D^1_-\subset D^2\times S^2\subset X^*,\] where $D^1_-$ is the arc in $S^2\cap (\R\times\R^-\times\{0\})\subseteq S^2$ with endpoints $\pm1\in S^1\subseteq S^2$.\\ The surface $S\subseteq X^*$ becomes smoothly isotopic to $\Sigma^*\subseteq X^*$  after an internal stabilization along the arc $\alpha$.
Moreover, $\pi_1(X^*\bslash S)$ is isomorphic to $\pi_1(X^*\bslash \Sigma^*)$.
In particular, if $\overline{\gamma}$ is nullhomotopic in $X\bslash\Sigma$, then $\pi_1(X\bslash \Sigma)$ is isomorphic to $\pi_1(X^*\bslash S)$.
\end{prop}

\begin{proof}
We can see $\alpha$ as the cocore of the 3-dimensional 2-handle $H_\gamma\subseteq(\nu_{1/2}\gamma)^*$ attached to $\Sigma^*$ described in Remark \ref{oss: Compressing disk}. So, $H_\gamma$ defines a 3-dimensional 1-handle attached to $S$ with core the chord $\alpha$, 
and this concludes the first part of the proof.

Now we will prove the claim on the fundamental groups. 
Since $\Sigma$ is obtained from $S$ via an internal stabilization we have that $\pi_1(X^*\bslash \Sigma)$ is isomorphic to $ \pi_1(X^*\bslash S)$ modulo the normal subgroup generated by $m_+\cdot m_-^{-1}$, where $m_{\pm}$ are the class in $\pi_1(X^*\bslash S_1)$ of loops obtained connecting the meridians of $D^2\times\{\pm1\}\subseteq S$ with paths, near the arc defining the stabilization, to a base point. For more details see \cite[Proof of Lemma 9]{Boyle_1988}. It is possible to see that $m_+=m_-\in\pi_1(D^2\times (S^2\bslash\{\pm1\}))$ and so in $\pi_1(X^*\bslash S)$, therefore $\pi_1(X^*\bslash S)$ and $\pi_1(X^*\bslash \Sigma^*)$ are isomorphic.
\end{proof}

\begin{oss} \label{rem: loop is connected sum}
    If $\overline{\gamma}$ is the boundary of an immersed 2-disk $D$ and $U$ is a neighborhood of $D$, then we have a diffeomorphism  $\phi_{D,U}:X^*\approx X\# B$, where the connected sum is performed inside $U$, $\phi_{D,U}$ restrict to the identity in the complement of $U\cup\nu\gamma$  and $B$ is $\Stab$ or $\StabTwist$.
 In particular, if $\Sigma$ is closed and disjoint from $D$, then $U$ can be chosen to be disjoint from $\Sigma$ and it follows that $\phi_{D,U}(\Sigma^*)=\Sigma\subseteq X\# B$. That is we have a diffeomorphism of pairs
 \[\phi_{D,U}:(X^*, \Sigma^*)\to (X\#B,\Sigma)\]
 Therefore, Proposition \ref{Prop: Relation External-Internal} shows how an external stabilization of a surface $\Sigma$ in $X$ can be seen as an internal stabilization of the surface $\phi_{D,U}(S)$ in $X\# B$.
\end{oss}


\begin{definition}\label{def: internal stab across}
    For a $\nu\gamma$-standard surface $\Sigma$ we can define the 
    surface $\Sigma'$ as a result of an internal stabilization, locally orientation preserving, of $\Sigma$ along a 3-dimensional 1-handle $h$ with core arc $f^{-1}(\{1\}\times a)$, where $a$ is any  arc in $D^3$ connecting the two components of $D^1\bslash (\frac{1}{2} D^1)\subseteq D^3\bslash (\frac{1}{2} D^3)$ and $f(h)$ is contained in $ S^1\times (int(D^3)\bslash (\frac{1}{2} D^3))$.
    We say that $\Sigma'$ is the trivial \emph{internal stabilization across $\gamma$} of $\Sigma$. It is well-defined up to smooth isotopy relative to $\nu_{1/2}\gamma:=f^{-1}(S^1\times(\frac{1}{2}D^3))$ and $X\bslash \nu\gamma$. \\
    A 1-handle $h$ defines an \emph{internal stabilization away} from $\gamma$ if it defines an internal stabilization for a $\nu\gamma$-standard surface and $h$ is disjoint from the closure of ${\nu_{1/2}\gamma}$.\\
    We say that two $\nu\gamma$-standard surfaces $\Sigma_1,\Sigma_2\subseteq X$ are \emph{internally $\gamma$-stably isotopic} if the surfaces $\Sigma_1',\Sigma_2'\subseteq X$, defined by a trivial internal stabilization across $\gamma$, are smoothly isotopic relative to $\nu_{1/2}\gamma$.
\end{definition}
\begin{oss}
The surface $\Sigma'$, the trivial internal stabilization across $\gamma$ of a surface $\Sigma$, does not depend on the chosen framing of $\gamma$ up to smooth isotopy relative to ${\nu_{1/2}\gamma}\cup (X\bslash \nu\gamma)$. 
\end{oss}

\begin{prop}\label{Prop Extern by internal}
Let $X$ be a 4-manifold, $\gamma$ a simple loop in $X$ with a chosen framing. Assume that $\Sigma_1,\Sigma_2\subseteq X$ are $\nu\gamma$-standard surfaces. Let $\Sigma_1^*,\Sigma_2^*\subseteq X^*$ be defined as in Construction \ref{constr: Sigma* S}.
If one of the following two conditions holds:
\begin{enumerate}[label=(\Alph*)]
    \item the surfaces $\Sigma_1,\Sigma_2\subseteq X$ are internally $\gamma$-stably isotopic,
    \item the surfaces $\Sigma_1,\Sigma_2\subseteq X$ are oriented, and they become smoothly ambient isotopic relative to $\nu_{1/2} \gamma$ after one orientation preserving internal stabilization away from $\gamma$, the manifold $X^*$ is simply connected, $\pi_1(X^*\bslash \Sigma_i^*)$ is cyclic  and $\Sigma_i\bslash\gamma$ is connected for $i=1,2$,
\end{enumerate}
then $\Sigma_1^*,\Sigma_2^*\subseteq X^*$ are smoothly isotopic. 
 Moreover, the statement still holds if we require isotopies to be relative to the boundary.
\end{prop}
\begin{proof}
By Proposition \ref{Prop: Relation External-Internal}, the surfaces $\Sigma_1^*,\Sigma_2^*\subseteq X^*$ can be considered as the internal stabilization along an arc $\alpha$ of surfaces $S_1,S_2\subseteq X^*$ which coincide with $\Sigma_1$ and $\Sigma_2$ in $X\bslash\nu_{1/2} \gamma$ and which match a standard piece in $(\nu_{1/2}\gamma)^*$, see Remark \ref{rem: nugamma*}.

 By hypothesis, in both cases, there exists an isotopy $f_t:X\to X$ relative to ${\nu_{1/2} \gamma}$ between the surfaces $\Sigma'_1,\Sigma'_2\subseteq X$ obtained from $\Sigma_1,\Sigma_2\subseteq X$ via an internal stabilization, respectively, along some arcs $\alpha_1,\alpha_2\subseteq X$ away from $\gamma$. Define $S_i'\subseteq X^*$ to be the surface obtained from $S_i\subseteq X^*$ via an internal stabilization along the arc $\alpha_i$.
To finish the proof we will construct the following sequence of smooth isotopies between the indicated surfaces:

    \[\Sigma_1^*\leadsto S_1'\leadsto S_2'\leadsto\Sigma_2^*\]
 
 Since the surfaces $S_i'$ and $\Sigma'_i$ coincide in $X\bslash\nu_{1/2}\gamma=X^*\bslash (\nu_{1/2}\gamma)^*\subseteq X^*$  for $i=1,2$, the surfaces $S_1'$ and $S_2'$ coincide in $(\nu_{1/2}\gamma)^*$,
then the isotopy $f_t$ induces a smooth ambient isotopy of $X^*$, which restrict to the identity on $(\nu_{1/2}\gamma)^*$ and it sends $S_1'$ to $S_2'$.
 
    Under the first set of assumptions we have that the internal stabilizations are trivial and across $\gamma$, this means that we can assume $\alpha_1=\alpha_2=\alpha$. So there exist a smooth ambient isotopy between $\Sigma_i^*$ and the surface $S_i'$ for $i=1,2$. 
    
    Under the second set of assumption by \cite[Lemma 3]{InternalStabilizationBaykurSunukian} the stabilization arc $\alpha$ can be positioned to where is needed, so again there exist a smooth ambient isotopy between $\Sigma_i^*$ and the surface $S_i'$ for $i=1,2$.
    This concludes the proof.
\end{proof}

\subsection{Stabilization sets \label{sec: StabilizationSet}}
In this section we will define \emph{the stabilization sets}, these sets will allowed us to generalize Remark \ref{rem: loop is connected sum} and to state results more clearly.
\begin{definition}\label{def: Ss}
     Let $\gamma\subseteq X$ be a framed loop in the interior of a 4-manifold $X$, with framing $f:\nu\gamma\to S^1\times D^3$.  Let $\Sigma_1,\Sigma_2\subseteq X$ be two $\nu\gamma$-standard surfaces, $\overline{\gamma}$  be the small push off of $\gamma$ and $\overline{ f}:\nu\overline{\gamma}\to S^1\times D^3$ its framing. Identify $X$ as $X\# S^4$, where the connected sum is performed along a 4-disk in $X$ disjoint from $\Sigma_1\cup \Sigma_2\cup\overline{\gamma}$, and let $U$ be the unknotted loop $S^1\subseteq S^4\bslash D^4\subseteq X\# S^4$. 
     Suppose that there exists a smooth automorphism $\phi$ of $ X\# S^4$ which moves the loop $\overline{\gamma}$ to the loop $U$ and we can also suppose that $\phi(\nu\gamma)=\nu U\subseteq S^4\bslash D^4$. The loop $U\subseteq S^4$ has an induced framing, namely $f(\phi):=\overline{ f}\circ \phi^{-1}|_{\nu U}:\nu U\to S^1\times D^3$, so call $B(\phi)$ the 4-manifold that is the result of loop surgery along $U$ in $S^4$, namely \[B(\phi):=(S^4\bslash\nu U)\cup_{f(\phi)|_\partial} (D^2\times S^2).\] Notice that $B(\phi)$ is diffeomorphic to $\Stab$ or $\StabTwist$, see for example \cite[Proposition 5.2.4]{Gompf_Stipsicz_1999}.  \\
     We define \emph{the stabilization set of $\Sigma_1,\Sigma_2\subseteq X$ with respect to $\gamma$ framed with $f$ } to be the set
     \[\Ss(\Sigma_1,\Sigma_2, (\gamma,f)):=\bigcup_{\phi} B(\phi)\subseteq\{\Stab,\StabTwist\},\]
     where the union  is taken over all (possibly empty) automorphism $\phi$ as described above which restricts to the identity on 
     $\Sigma_1\cup\Sigma_2\cup\partial X$. If the framing $f$ of $\gamma$ is understood, we will omit it from the notation and write $\Ss(\Sigma_1,\Sigma_2, \gamma)$ instead of $\Ss(\Sigma_1,\Sigma_2, (\gamma,f))$. 
     We define \emph{the extended stabilization set of $\Sigma_1,\Sigma_2\subseteq X$ with respect to $\gamma$ } 
     to be the set
     \[\overline{\Ss}(\Sigma_1,\Sigma_2, \gamma):=\bigcup_{ f'} {\Ss}(\Sigma_1,\Sigma_2, (\gamma,f'))\subseteq\{\Stab,\StabTwist\},\]
     where the union is taken over all possible framings $ f':\nu\gamma\to S^1\times D^3$ for which $\Sigma_1,\Sigma_2\subseteq X$ are $\nu\gamma$-standard.
     If $\Sigma\subseteq X$ is a $\nu\gamma$-standard surface, then define the \emph{the stabilization set of $\Sigma\subseteq X$ with respect to $\gamma$ with framing $f$ } to be the set
     \[\Ss(\Sigma, (\gamma,f)):=\Ss(\Sigma,\Sigma, (\gamma,f))\subseteq\{\Stab,\StabTwist\}\]
     and similarly for the extended version
     \[\overline\Ss(\Sigma, \gamma):=\overline{\Ss}(\Sigma,\Sigma, \gamma)\subseteq\{\Stab,\StabTwist\}.\]
 \end{definition}
Later in this note is given an alternative definition of the stabilization set, which is stated in terms of immersed 2-disks in the complement of $\Sigma_1\cup\Sigma_2\cup \partial X$ bounding $\overline{\gamma}$, see  Proposition \ref{prop: equivalent definition}.\\
Now we are able to state the main result of this section.

\begin{thm}\label{thm Extern by internal}
Let $X$ be a  4-manifold and  $\gamma\subseteq X$ a simple loop in $X$ with a chosen framing. Assume that $\Sigma_1,\Sigma_2\subseteq X$ are $ \nu\gamma$-standard surfaces. 
If the surfaces $\Sigma_1,\Sigma_2\subseteq X$ are internally $\gamma$-stably isotopic,
then $\Sigma_1,\Sigma_2\subseteq X$ are 
$B$-stably isotopic for any $B\in \overline{\Ss}(\Sigma_1,\Sigma_2,\gamma)\subseteq\{\Stab,\StabTwist\}$.
  Moreover, the statement still holds if we require isotopies to be relative to the boundary.
\end{thm}
\begin{proof}   
If $B(\phi)\in\overline{\Ss}(\Sigma_1,\Sigma_2,\gamma)$ then there exists a framing $f$ of $\gamma$ such that the surfaces $\Sigma_1$ and $\Sigma_2$ are $\nu\gamma$-standard and the smooth automorphism $\phi$ of $X\# S^4$ restricts to the identity on $\Sigma_1\cup\Sigma_2$.
In particular, we have the following diffeomorphism of quadruples
\[\phi:(X,\Sigma_1,\Sigma_2,\nu\bar\gamma)\to (X\# S^4,\Sigma_1,\Sigma_2,\nu U),\] where the connected sum is performed away from $\Sigma_1\cup\Sigma_2$ and $U$ is the unknotted loop in $S^4\bslash D^4\subseteq S^4$. 
Moreover, $\phi$ induces a diffeomorphism of triples
\[\Phi: (X^*,\Sigma_1^*,\Sigma_2^*)\to (X\# B(\phi), \Sigma_1,\Sigma_2)\]

However, we are under the assumption of Proposition \ref{Prop Extern by internal}, so we get that $\Sigma_1^*$ and $\Sigma_2^*$ are smoothly isotopic in $X^*$. Follows that $\Sigma_1$ and $\Sigma_2$ are smoothly isotopic in $X\# B(\phi)$.\end{proof}

\subsubsection{Framings and 2-disks\label{sec: proprieties of SS 2}}
In this section we will change of perspective regarding framings of simple loops in a 4-manifold.
So far, for a curve $\alpha\subseteq X$, a framing was a diffeomorphism $f:\nu\alpha\to S^1\times D^3$.
Let us consider the following sections of the tangent bundle of $X$ restricted to $\alpha$ given by the chart $f$, namely $\mathcal{F}_i(x):=df^{-1} (f(x),(0,e_i))\in T_x X$ for $i=1,2,3$, we can notice that they form a basis for the normal bundle $N \alpha$ of $\alpha$, where $N\alpha\cong TX|_\alpha/T\alpha$.
We call the \emph{framing for the normal bundle $N \alpha$ of $\alpha$ induced by $f$}, the map
\[\mathcal{F}(\alpha,f)=(\mathcal{F}_1,\mathcal{F}_2,\mathcal{F}_3):S^1\to \mathcal{F}(N\alpha)\]
where $\mathcal{F}(N \alpha)$ is the frame bundle of $N \alpha$.
Two framings $f$ and $f'$ for the curve $\alpha$ are called \emph{equivalent} if $\mathcal{F}(\alpha,f)$ and $\mathcal{F}(\alpha,f')$ are homotopic sections of $\mathcal{F}(N \alpha)$.
\begin{oss}\label{oss: opposite framing}
    It is a well known fact that in dimension $4$ a curve with trivial normal bundle has exactly two equivalence class of framings. So given a framing $f$ for a curve $\alpha$ we denote as $f_{op}$ any other framing in the other equivalence class and we call it the opposite framing. Notice that $(f_{op})_{op}$ is in the same equivalence class as $f$.  \\
    If $\alpha$ is framed curve with framing $f$, then we denote as $\alpha_{op}$ the curve $\alpha$ equipped with the framing $f_{op}$.
\end{oss}
\begin{oss}\label{rem: Framing given framing}
    For any section of $\mathcal{F}$ of $\mathcal{F}(N\alpha)$ there exists a framing $f:\nu\alpha\to S^1\times D^3$ such that $\mathcal{F}(\alpha,f)=\mathcal{F}$.
\end{oss}
\begin{oss}\label{rem: framing from 2 section}
    For every two linearly independent sections $s_1, s_2$ of $N\alpha$, there exists a section $s_3$ of $N\alpha$ such that $\mathcal{F}=(s_1,s_2,s_3)$ is a section of $\mathcal{F}(N\alpha)$. Moreover,  $\mathcal{F}$ is uniquely defined by $s_1$ and $s_2$ up to homotopy.
\end{oss}
\begin{oss}
    If $\alpha$ is equipped with framings $f$ and $f'$, and they are equivalent framings, then given an open set  $V$ containing $\nu\alpha$, there exists an ambient isotopy $\phi_t$ of $X$ supported $V$ such that $f\circ\phi_1|_{\nu\alpha}=f'$.
\end{oss}
\begin{definition}\label{def: connsumloop}
     Suppose that are given two framed and disjoint curves  $\alpha$ and $\beta$ in a 4-manifold $X$ and assume that in a local chart $D^4\subseteq X$ they appear as in the local model pictured in Figure \ref{fig: sumframed 1}. In  $D^4\approx D^1\times D^3$ the curves lie in $\{0\}\times D^2\subseteq D^1\times D^3$, they are oriented with opposite orientations and they are lines with direction $\pm e_1\in \R^3$.
Suppose that the both curves have framings $\mathcal{F}(\alpha)$ and $\mathcal{F}(\beta)$ for the normal bundles $N \alpha$ and $N\beta$ given by the sections 
\begin{equation}\label{eq: framing A}
(e_1,0),(0,\pm e_2), (0, e_3)\in \R\times\R^3\cong T_p(D^1\times D^3)\end{equation}
for any point $p\in (\alpha\cup\beta)\cap (D^1\times D^3)$. We will replace this local pictures with Figure \ref{fig: sumframed 2} and call the resulting curve $\alpha\#\beta$ the \emph{band-sum} of $\alpha$ and $\beta$. Notice that the band-sum comes equipped with an orientation, moreover we will equip this curve with a framing induced by a framing $\mathcal{F}(\alpha\#\beta)$ of its normal bundle, see Remark \ref{rem: Framing given framing}. Define $\mathcal{F}(\alpha\#\beta)$ to coincide with $\mathcal{F}(\alpha)$ and $\mathcal{F}(\beta)$ outside $D^4$, meanwhile in $D^4\approx D^1\times D^3$ the framing is the one given by 
\begin{equation} \label{eq: framing B}
    (e_1,0), (0,n(x)), (0,e_3)\in \R\times\R^3\cong T_p(D^1\times D^3),
\end{equation}
    where $(0,n(x))$ is the normal vector in $x$ to the oriented curve $\alpha\#\beta\subseteq\{0\}\times D^2\subseteq D^1\times D^3$.\\
We will abuse of notation and denote as $\alpha\#\beta$ the framed curve $\alpha\#\beta$ with framing induced by $\mathcal{F}(\alpha\#\beta)$.
\end{definition}

\begin{figure}
    \centering
     \begin{subfigure}[b]{0.45\textwidth}
         \centering
         \includegraphics[width=0.45\textwidth]{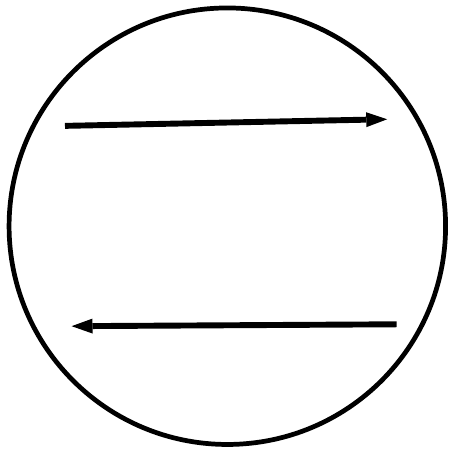}
         \caption{
         \label{fig: sumframed 1}}
     \end{subfigure}
      \begin{subfigure}[b]{0.45\textwidth}
         \centering
         \includegraphics[width=0.45\textwidth]{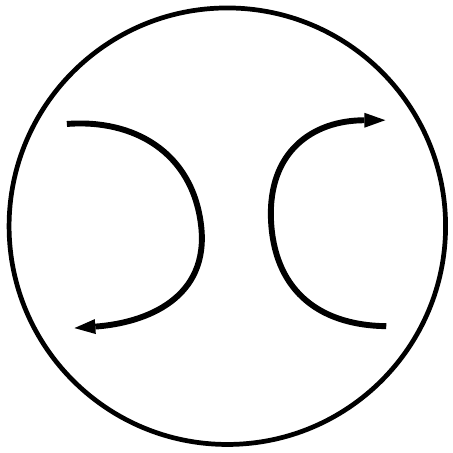}
         \caption{
         \label{fig: sumframed 2}}
     \end{subfigure}
    \caption{The figures represent $D^1 \times D^3$ where the $D^1$ factor is suppressed. \label{fig: sum of framed curve}}
\end{figure}
\begin{oss} The following equality holds
    $(\alpha\#\beta)_{op}=\alpha_{op}\#\beta=\alpha\#\beta_{op}$.
\end{oss}
\begin{oss}
In Definition \ref{def: connsumloop} the framings used in Equations (\ref{eq: framing A}) and (\ref{eq: framing B}) are called product framing. \\
    If $a$ is an arc in a 3-manifold $ Y^3$ and it is equipped with a framing $\mathcal{F}(a)$ of $Na$, 
then to the arc $\{0\}\times a\subseteq D^1\times Y$ has normal bundle $N(\{0\}\times a)$ isomorphic to $\epsilon^1\oplus Na$, where $\epsilon_1$ is the trivial rank $1$ bundle. The product framing of $N(\{0\}\times a)$ is defined as a section of $\epsilon^1$ together with $\mathcal{F}(a)$, compare also with Remark \ref{rem: framing from 2 section}. 
In Definition \ref{def: connsumloop} we took the arcs in $D^3$ in Figure \ref{fig: sum of framed curve} equipped with the blackboard framing.
\end{oss}

\begin{definition}\label{def: compatible framing}
    Let $\beta$ be a simple loop in a 4-manifold $X$ and $D\subseteq X$ be an immersed 2-disk bounding $\beta$. The 2-disk $D$ induce a framing $f_D$ of $\beta$ called the framing induced by the unique normal framing of $D$, see for example \cite[Chapter 5.6]{Gompf_Stipsicz_1999}. The framing $f_D$ is sometimes called a spin framing of $\beta$.
    If $\beta$ is equipped with a framing $f$, then we say that the framing of $\beta$ is \emph{compatible with $D$} if $(\beta,f)$ is equivalent to $(\beta,f_D)$. Otherwise, we say that the framing of $\beta$ is \emph{not compatible with $D$}.
\end{definition}

\begin{lem}\label{lem: ConSumOfLoop and Disk}
    Let $\alpha,\beta,\alpha\#\beta\subseteq X$ and $D^4\subseteq X$ be as in Definition \ref{def: connsumloop}. If $\beta$ bounds an immersed 2-disk $D\subseteq X$ and $V$ is an open neighborhood of $\alpha\cup D\cup D^4$, then there exists an ambient isotopy $\phi_t$ of $X$ supported $V$ such that
    \[\phi_1(\alpha\# \beta)=\begin{cases}
        \alpha, & \text{ if the framing of $\beta$ is compatible with $D$}\\
        \alpha_{op}, &\text{ if the framing of $\beta$ is not compatible with $D$}\\
    \end{cases}\]
\end{lem}
\begin{proof}
    We will separate the proof in two cases.
    
    {\bf Case 1:} Assume first the framing of $\beta$ is compatible with $D$.
    Up to perturbing a little bit the interior of $D$ we can assume that $D$ and $\alpha$ are disjoint. Furthermore, by assumption we can assume that $\beta$ is equipped with the framing $f_D$ induced by $D$.
    Next we will perform a "slide" of $\beta$ over $D$, this emulates the operation of sliding a framed curve over a 2-handle, the only difference here is that $D$ is not an embedded 2-disk, but it is immersed. So, we move a small arc of $\alpha$ across $D$, carefully passing through double points of $D$. The resulting curve is $\alpha\#\beta$ and it  has the correct framing.
    
   {\bf Case 2:} If the framing of $\beta$ is not compatible with $D$, then $\beta_{op}$ has framing compatible with $D$. Applying the first part of the proof we get an isotopy $\phi_t$ such that
   \[\phi_1(\alpha\#\beta)=(\phi_1(\alpha\#\beta)_{op})_{op}=\left(\phi_1(\alpha\#\beta_{op})\right)_{op}=(\alpha)_{op}.\]
   This concludes the proof.
\end{proof}
Now we can reformulate Definition \ref{def: Ss} in terms of 2-disks.
\begin{prop}\label{prop: equivalent definition}
Let $\Sigma_1,\Sigma_2\subseteq X$,  the framed loop $\gamma\subseteq \Sigma_1\cap\Sigma_2$ and the framed loop $\overline{\gamma}\subseteq X$ be as in Definition \ref{def: Ss}.
      The manifold $\Stab$ (respectively $\StabTwist$) is in the set $\Ss(\Sigma_1, \Sigma_2,\gamma)$ if and only if the framing of $\overline{\gamma}$ is compatible (resp. not compatible) with an immersed 2-disk $D\subseteq X\bslash(\Sigma_1\cup \Sigma_2\cup\partial X)$ bounding $\overline{\gamma}$.
\end{prop}
\begin{proof}
Consider the unknotted loop $U\subseteq S^4\bslash D^4\subseteq X\#S^4$ bounded by a 2-disk $D_U$  in $S^4\bslash D^4$.

If $B(\phi)\in \Ss(\Sigma_1, \Sigma_2,\gamma)$, then $D:=\phi^{-1}(D_U)\subseteq X\# S^4\cong X$ is an embedded 2-disk bounding $\overline{\gamma}$ and disjoint from $\Sigma_1\cup \Sigma_2\cup \partial X$. The framing of $\overline{\gamma}$ is compatible with $D$ if and only the framing $f_\phi=f\circ \phi^{-1}$ of $U$ is compatible with $D_U$, i.e. when $B(\phi)\cong\Stab$.

If $D$ is an immersed 2-disk bounding $\gamma$ disjoint from $\Sigma_1\cup\Sigma_2\cup\partial X$, then it is sufficient to consider a band sum $\overline{\gamma}\# U$, where $U$ now is equipped with a framing compatible with $D_U$. We use Lemma \ref{lem: ConSumOfLoop and Disk} in the two possible way to find an automorphism $\phi_D$ of $X\# S^4$ with compact support disjoint from $\Sigma_1\cup \Sigma_2\cup \partial X$ sending the loop $\overline{\gamma}$ to the loop $U$. Moreover, $B(\phi_D)$ is $\Stab$ if and only if $\overline{\gamma}$ has a framing compatible with $D$.
\end{proof}

The next lemma will give us some way to concretely compute the cardinality of set $\Ss(\Sigma_1,\Sigma_2,\gamma)$.
\begin{lem}\label{lem: proprieties Ss}
Let $X$ be a 4-manifold and $\gamma\subseteq X$ a framed loop with  framing $f$. Assume that $\Sigma_1,\Sigma_2\subseteq X$ are $\nu\gamma$-standard surfaces.
\begin{enumerate}
    \item \label{lem: prop Part1}The set ${\Ss}(\Sigma_1,\Sigma_2,(\gamma,f))\not=\emptyset$ if and only if the curve $\overline{\gamma}$, the small push off of $\gamma$, is nullhomotopic in $X\bslash(\Sigma_1\cup\Sigma_2)$.
    \item  \label{lem: SS item 3}
    The curve $\overline{\gamma}$ is nullhomotopic in $X\bslash(\Sigma_1\cup \Sigma_2)$ and there exists an immersed sphere of odd self intersection in $X\bslash (\Sigma_1\cup \Sigma_2)$ if and only if the set  ${\Ss}(\Sigma_1,\Sigma_2,(\gamma,f))=\{\Stab,\StabTwist\}$
\end{enumerate}   
\end{lem}
\begin{proof} Point 1 is an immediate corollary of Proposition \ref{prop: equivalent definition}. 
Now we are going to prove point 2.
    If $\overline{\gamma}$ is nullhomotopic, then there  exists an automorphism $\phi$ as prescribed as Definition \ref{def: Ss} by the first point of this lemma.
    The proof of \cite[Proposition 5.2.4]{Gompf_Stipsicz_1999} describes an ambient isotopy that exchange the framing of a nullhomotopic loop in the presence of an immersed odd intersection 2-sphere $S\subseteq X$. Let $\psi_t$ be such an isotopy that exchange the framings of $U$, we can assume that the support of this isotopy is in $(X\#S^4)\bslash\nu(\Sigma_1\cup \Sigma_2\cup\partial X)$. 
    The automorphism $\phi':=\psi_1\circ\phi$ satisfies the condition of Definition \ref{def: Ss} and by construction we have that 
    \[B(\phi')\in \left(\{\Stab,\StabTwist\}\bslash \{B(\phi)\}\right).\]
    Therefore, $\Ss(\Sigma_1, \Sigma_2, (\gamma,f))=\{\Stab,\StabTwist\}$.\\
Vice versa, if ${\Ss}(\Sigma_1, \Sigma_2,\gamma)=\{\Stab,\StabTwist\}$, then we have a diffeomorphism 
$h:(X\# S^4,(\Sigma_1\cup \Sigma_2), U)\to (X\# S^4,(\Sigma_1\cup \Sigma_2), U)$, where $U$ is the unknotted loop in $S^4\bslash D^4$ and $h$ reverse the framing of $U$. Therefore, the map $h$ is inducing a diffeomorphism of couples
\[(X\# (\Stab),\Sigma_1\cup \Sigma_2)\to (X\# (\StabTwist),\Sigma_1\cup \Sigma_2).\]
In particular, $(X\bslash(\Sigma_1\cup \Sigma_2))\#(\Stab)$ is diffeomorphic to $(X\bslash(\Sigma_1\cup \Sigma_2))\#(\StabTwist)$. By \cite[Exercise 5.2.6.(b)]{Gompf_Stipsicz_1999} there is an immersed 2-sphere of odd  self intersection in $X\bslash(\Sigma_1\cup \Sigma_2)$. 
\end{proof}
\subsubsection{Proof Theorem \ref{mainthm: int to ext}\label{sec: thm A}}
If $\Sigma_1,\Sigma_2\subseteq X$ are internally $\gamma$-stably isotopic, then they are internally 1-stably isotopic (with a trivial stabilization) by definition. By Theorem \ref{thm Extern by internal} they are $B$-stably isotopic for any $B\in\overline{\Ss}(\Sigma_1,\Sigma_2,\gamma)$. Lemma \ref{lem: proprieties Ss} says that $\overline{\Ss}(\Sigma_1,\Sigma_2,\gamma)$ is not empty if and only if there exists a framing for $\gamma$ with a small push off nullhomotopic in $X\bslash(\Sigma_1\cup\Sigma_2\cup\partial X)$, i.e. if there exists a nullhomotopic push off of $\gamma$ in $X\bslash(\Sigma_1\cup\Sigma_2)$. $\hfill\qed$

\subsubsection{Properties of stabilization sets\label{sec: proprieties SS}}
The next lemma allows us to consider the stabilization set of a single surface instead of a pair. Therefore, we will focus later on the properties of the stabilization set of a single surface, even though most of the statements hold (with small twinks) for the stabilization set of two surfaces.
\begin{lem} \label{lem: Sdi2 in fuzione di S1}Let $\Sigma_1,\Sigma_2\subseteq X$ be $\nu\gamma$-standard surfaces in a 4-manifold $X$.
    If $\Sigma_2\subseteq\nu \Sigma_1$, then 
    ${\Ss}(\Sigma_1,\Sigma_2,\gamma)={\Ss}(\Sigma_1,\gamma)$
   and 
    $\overline{\Ss}(\Sigma_1,\Sigma_2,\gamma)=\overline{\Ss}(\Sigma_1,\gamma)$.
\end{lem}
\begin{proof}
    It is always true that ${\Ss}(\Sigma_1,\Sigma_2,\gamma)\subseteq {\Ss}(\Sigma_1,\gamma)$ by Proposition \ref{prop: equivalent definition}. 
   Let $D$ be an immersed 2-disk bounding $\overline{\gamma}$, the small push off of $\gamma$, in the complement of $\Sigma_1\cup \partial X$, then, under our assumption, we can perturb the interior of $D$ so that it lies in the complement of $\Sigma_1\cup \Sigma_2\cup \partial X$. This shows that ${\Ss}(\Sigma_1,\gamma)\subseteq{\Ss}(\Sigma_1,\Sigma_2,\gamma)$ by Proposition \ref{prop: equivalent definition}. It follows that $\overline{\Ss}(\Sigma_1,\Sigma_2,\gamma)=\overline{\Ss}(\Sigma_1,\gamma)$ by definition.
\end{proof}

\begin{constr}\label{constr: framings}
If $\gamma\subseteq X$ has a framing $f$ for which a surface $\Sigma\subseteq X$ is $\nu\gamma$-standard, then for all $n\in\Z$ we can define the framing 
\begin{equation}\label{eq: framings nucompatible}
    f_n:=\tau^n\circ f:\nu\gamma\to S^1\times D^3,
\end{equation}
where $\tau:S^1\times D^3\to S^1\times D^3$ is the map $\tau(\theta, x)=(\theta, R_\theta x)$ with $R_\theta$ being the rotation of $D^3$ of angle $\theta$ with rotation axis the line $\R\times\{0\}^2\subseteq\R^3$. The surface $\Sigma$ is $\nu\gamma$-standard for all framings $f_n$ with $n\in\Z$. For any framing $f_n$ the small push off of $\gamma$ is a curve $\overline{\gamma}(n)$. We  fix a base point $x_0\in\overline{\gamma}(0)\in X$ and an oriented meridian $\mu$ of $\Sigma$. The following equation holds in the fundamental group of $X\bslash \Sigma$:
\begin{equation} \label{eq: pi1 framings}
    [\overline{\gamma}(n)]=[\overline{\gamma}(0)]\cdot [\mu]^n\in\pi_1(X\bslash\Sigma,x_0)
\end{equation}    
\end{constr}

If $\Sigma\subseteq X$ is a $\nu\gamma$-standard surface, then the framing $\mathcal{F}(\gamma)$ of the normal bundle of $\gamma$ in a point is given by a first vector, which is normal to $\gamma$ and tangent to $\Sigma$, and other two in the normal bundle of $\Sigma$. Therefore, different framings of $\gamma$ (for which $\Sigma$ is $\nu\gamma$-standard) represent different homotopy classes in \[[S^1,\{A\in SO(3): A e_1=\pm e_1\}]\cong [S^1, O(2)]\cong \Z\times \Z_2.\]  The $\Z$ factor is realized by Construction \ref{constr: framings} and the $\Z_2$ factor is realized by the automorphism $Id_{S^1}\times R^{e_2}_\pi$ of $ S^1\times D^3$, where $R^{e_2}_\pi: D^3\to D^3$ is a rotation of angle $\pi$ with axis $e_2$. The next lemma follows easily.
\begin{lem}\label{lem: BarSs is union of Ss}
    Let $\gamma\subseteq X$ be a simple curve in a 4-manifold $X$ with a framing $f_0$, assume that $\Sigma\subseteq X$ is a $\nu\gamma$-standard surface.
    We have that
    \[\overline{\Ss}(\Sigma,\gamma)=\bigcup_{n\in\Z} \Ss\left(\Sigma,(\gamma,f_n)\right),\]
    where the framings $f_n$ are defined as in Construction \ref{constr: framings}.
\end{lem}

The rest of this section aims to better understand the relationship between the framed curves $\overline{\gamma}(n)$ defined in the Construction (\ref{constr: framings}) for $n\in\Z$ and the dependence of the stabilization set on the chosen curve $\gamma\subseteq \Sigma$.

\begin{lem} \label{lem: gamma1 in terms of gamma0} Let $\gamma\subseteq X$ be a simple curve in a 4-manifold $X$ with framing $f_0$, assume that $\Sigma\subseteq X$ is a $\nu\gamma$-standard surface. Let $D\subseteq X$ be a 2-disk fiber of $\nu\Sigma$ bounding the meridian $\mu\subseteq X$ of $\Sigma$ over a point of $\gamma$. If $\mu$ is equipped with a framing compatible with $D$, then the framed loop $\overline{\gamma}(1)$ is isotopic to $\overline{\gamma}(0)_{op}\# \mu$ relative to $\Sigma\cup \partial X$, where $\overline{\gamma}(0)_{op}\# \mu$ is a band sum constructed as in Definition \ref{def: connsumloop} and away form $\Sigma$.
\end{lem}
\begin{proof}
    It is sufficient to restrict our attention to $\nu\gamma$, we will look at the curves through the diffeomorphism  $f_1:\nu\gamma\to S^1\times D^3$.
    Set \[\alpha:=f_1(\overline{\gamma}(0))\subseteq S^1\times (D^1\times D^2),\]  \[\sigma:=f_1(\overline{\gamma}(1))\subseteq S^1\times(\{0\}\times D^2)\subseteq S^1\times (D^1\times D^2)\] and \[\beta:=f_1(\mu)=\{p\}\times (\{0\}\times S^1)\subseteq S^1\times(D^1\times D^2).\]
    The aim is to show that the framed loops $\alpha_{op}\#\beta,\sigma\subseteq S^1\times D^3$ are isotopic  relative to $(S^1\times \partial D^3)\cup(S^1\times D^1)$.
    Notice that $\alpha=\tau(\sigma)\subseteq S^1\times D^3$ by equation (\ref{eq: framings nucompatible}). Recall that $\tau(\theta,x)=(\theta,R_\theta x)$ where $R_\theta$ is the rotation of $\theta$ around the axis defined by $e_1$ in $D^3\cong D^1\times D^2$. 
    Therefore, we have that \[\alpha,\beta,\sigma\subseteq S^1\times(\{0\}\times D^2)\subseteq S^1\times (D^1\times D^2)\]
    The framings of three loops in  $S^1\times (D^1\times D^2)$ is determinate by the framings of the loops as subsets of $S^1\times(\{0\}\times D^2)$. 
    A picture of $\alpha,\beta$ and $\sigma$ as subset of $S^1\times \{0\}\times D^2$ is depicted in Figure \ref{fig: 1 ciao} and \ref{fig: 3 ciao}. The loops $\sigma$ and $\beta$ have the blackboard framing, however for $\alpha$ have not. In particular, the framing of $\sigma$ is determined by the section of its normal bundle in $S^1\times\{0\}\times D^2$ which takes value the radial direction $x$ at each point $(\theta, 0,x)\in\sigma$. Since $\alpha=\tau(\sigma)$,  as framed curves, we have that the framing of $\alpha$ is determined by a section of its normal bundle in $S^1\times\{0\}\times D^2$ which takes value the radial direction $R_\theta x$ at each point $(\theta, 0,R_\theta x)\in\alpha$. Therefore, $\alpha$ has a framing with an additional \emph{full twist} and we can represent $\alpha_{op}$ in Figure \ref{fig: 1 ciao} by framing it with the black board framing. Figure \ref{fig: 2 ciao} shows 
 the framed curve $\alpha_{op}\#\beta$, which is clearly isotopic to $\sigma$ in Figure \ref{fig: 3 ciao}, as framed curves.
 This concludes the proof that $\overline{\gamma}(0)_{op}\#\mu$ is isotopic to $\overline{\gamma}(1)$.
\end{proof}
\begin{figure}[h]
    \centering
     \begin{subfigure}[b]{0.30\textwidth}
         \centering
         \includegraphics[width=0.65\textwidth]{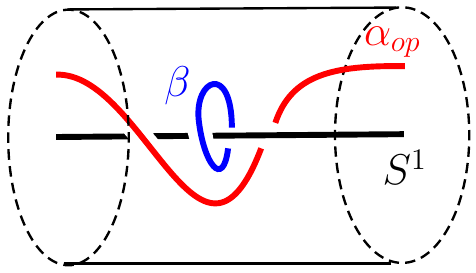}
         \caption{\label{fig: 1 ciao}}
     \end{subfigure}
      \begin{subfigure}[b]{0.30\textwidth}
         \centering
         \includegraphics[width=0.65\textwidth]{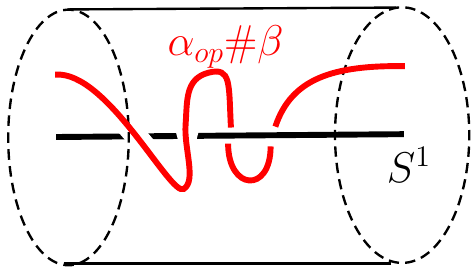}
         \caption{\label{fig: 2 ciao}}
     \end{subfigure}
     \begin{subfigure}[b]{0.30\textwidth}
         \centering
         \includegraphics[width=0.65\textwidth]{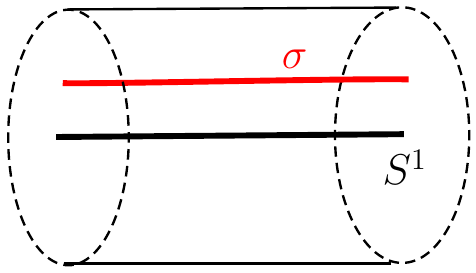}
         \caption{\label{fig: 3 ciao}}
     \end{subfigure}
    \caption{\label{fig: ciao}The figures depict a full cylinder $I\times D^2$ which represent $S^1 \times \{0\}\times D^2$ where  and the boundary $\partial I\times D^2$ should be identified. Each of the loops showed in the pictures is equipped with the blackboard framing. The black curve labeled $S^1$ in all picture is $S^1\times\{0\}\times\{0\}$.}
\end{figure}

\begin{oss}\label{oss: framed loop}
Iterating the proof above, it can be proved that
    $\overline{\gamma}(n)=\left(\overline{\gamma}(0) \# (n\mu)\right)_{n\cdot op}$, where $n\mu$ is the band sum of $n$ parallel push off of the meridian $\mu$ and ${n\cdot op}$ denotes the opposite framing if $n$ is odd.
    \end{oss}
 
\begin{prop}\label{lem: geometricDualSphere}
    Let $\gamma\subseteq X$ be a framed curve in a 4-manifold $X$ with framing $f_0$, assume that $\Sigma\subseteq X$ is a $\nu\gamma$-standard surface. 
    If there exists an immersed 2-sphere $S\subseteq X$ intersecting $\Sigma$ in exactly one point on the component of $\Sigma$ containing $\gamma$, then there exists an ambient isotopy $\psi_t$ of $X$ relative to $\Sigma\cup\partial X$ which sends the loop $\bar\gamma(1)$ to $\overline{\gamma}(0)$. Moreover, the following equality of framed loops holds:\\
    \[\psi_1(\overline{\gamma}(1))=\begin{cases}
        \overline{\gamma}(0)_{op}, &\text{ if $S$ has even self intersection}\\
        \overline{\gamma}(0), &\text{ if $S$ has odd self intersection}\\
    \end{cases}.\]
    In particular, if $S$ has odd self intersection $\Ss(\Sigma,(\gamma,f_0))=\Ss(\Sigma,(\gamma,f_1))$ and if $S$ has even self intersection  $\Ss(\Sigma,(\gamma,f_0))\cup\Ss(\Sigma,(\gamma,f_1))=\{\Stab,\StabTwist\}$.
\end{prop}
\begin{proof}
By Lemma \ref{lem: gamma1 in terms of gamma0} $\overline{\gamma}(1)_{op},\overline{\gamma}(0) \# \mu$ are isotopic as framed curves relative to $\Sigma\cup\partial X$, where $\mu$ is a framed meridian of $\Sigma$. 
The immersed 2-sphere $S$ can be written as union of two 2-disk $D_1\cup D_2$ with boundary $\mu$. We can  assume that $D_1$ is in the complement of $\Sigma$, $D_2$ intersects $\Sigma$ in one point and the framing of $\mu$ is compatible with $D_2$. \\
By Lemma \ref{lem: ConSumOfLoop and Disk}
 there exists and isotopy $\psi_t$ of $X$ that restricts to the identity in a neighborhood of $\Sigma$, and
  \[\psi_1(\overline{\gamma}(1)_{op})=\begin{cases}
        \overline{\gamma}(0), &\text{ if the framing of $\mu$ is compatible with $D_1$}\\
        \overline{\gamma}(0)_{op}, &  \text{ if the framing of $\mu$ is not compatible with $D_1$}\\
    \end{cases}\]
    Moreover, we have that the framing of $\mu$ is compatible with $D_1$ if and only if the immersed 2-sphere $S$ has even self intersection.

Now notice that, if $B(\phi)\in \Ss(\Sigma,(\gamma,f_0))$, where $\phi$ is 
 a diffeomorphism as in Definition \ref{def: Ss}, then $B(\phi\circ \psi_1)\in\Ss (\Sigma,(\gamma,f_1))$. The rest of the statement follows immediately.
\end{proof}

Now we would like to study the stabilization sets of different curves on the same surface.\\
Let $\alpha, \beta\subseteq X$ be framed curves on a surface $\Sigma\subseteq X$ and that $\Sigma$ is both $\nu\alpha$-standard and $\nu\beta$-standard. Assume that $\alpha$ intersect $\beta$ transversely in one point. 
Define $\gamma:=\alpha+\beta$ to be the curve in $\Sigma$ obtained from $\alpha\cup\beta$ by solving the intersection point, that is replacing a small neighborhood of the intersection point $\alpha\cap \beta$ with the model in Figure \ref{fig: Asumframed 1}.
The curve $\gamma\subseteq X$ can be equipped with a framing induced by the framed curves $\alpha$ and $\beta$. To do so consider a small 2-disk $D\subseteq \Sigma$ containing the intersection point $\alpha\cap\beta$ and consider a trivialization of its tubular neighborhood $\phi:\nu D\to D\times D^2$. The framing $\mathcal{F}(\alpha)=(\mathcal{F}_1(\alpha),\mathcal{F}_2(\alpha),\mathcal{F}_3(\alpha))$ of the normal bundle of $\alpha$ has $\mathcal{F}_1(\alpha)(p)$ in $T_p\Sigma$ by assumption. We can assume that $\phi_*(\mathcal{F}_2(\alpha)(p))=(0,e_1)\in T_p D\times\R^2\cong T_{\phi(p)}(D\times D^2) $ and $\phi_*(\mathcal{F}_3(\alpha)(p))=(0,e_2)$  for all point $p\in\alpha$. A similar condition can be stated for $\beta$. We define a framing for the normal bundle of $\gamma$ to be the restriction of the framings of $\alpha$ and $\beta$ on $(\alpha\cup\beta)\bslash D$ and to be $\mathcal{F}(\gamma)(p)=(n(p),(0,e_1),(0,e_2))$ at the point $p$ for any $p\in D\cap\gamma$, where $n(p)$ is a normal vector of $\gamma$ in $\Sigma$.
Denote as $\overline{\alpha}$,
$\overline{\beta}$ and $\overline{\gamma}$ the small push offs of the corresponding framed curves.
\begin{prop}\label{prop: Sumofframed curves on a surface}
The framed curves $\overline{\gamma}$ and $(\overline{\alpha}\#\overline{\beta})_{op}$ are isotopic in $X$ relative to $\Sigma$, where $\overline{\alpha},\overline{\beta},\overline{\gamma}\subseteq X$ are defined in the preceding paragraph.
\end{prop}
\begin{proof}
    We will focus on the 4-disk $D^4\cong D\times D^2$ centered at the point $\alpha\cap \beta$. 
    We can assume the framed curve $\overline{\alpha}$ and $\overline{\beta}$ (possibly up to a small isotopy to make them disjoint) in this $D^4$ appear as in Figure \ref{fig: Asumframed 3}, where $D^4$ is written as $D^1\times D^3$ and the $D^1$ factor is a normal direction of the surface. The surface $\Sigma$ in $D^4$ looks like a 2-disk under the curve $\overline{ \alpha}$ and $\overline{\beta}$. Figure \ref{fig: Asumframed 2} shows the framed curve $\overline{\alpha}\#\overline{\beta}$ and Figure \ref{fig: Asumframed 1} shows the framed curve $\overline{\gamma}$. A Reidemeister I move describes an isotopy between the arcs in Figures \ref{fig: Asumframed 2}-\ref{fig: Asumframed 3}, but it does not preserve the black board framing, instead it adds a full twist. It follows that the framed curves $\overline{\gamma}$ and $(\overline{\alpha}\#\overline{\beta})_{op}$ are isotopic in $X$ relative to $\Sigma$.
\end{proof}

\begin{figure}
    \centering
     \begin{subfigure}[b]{0.30\textwidth}
         \centering
         \includegraphics[width=0.65\textwidth]{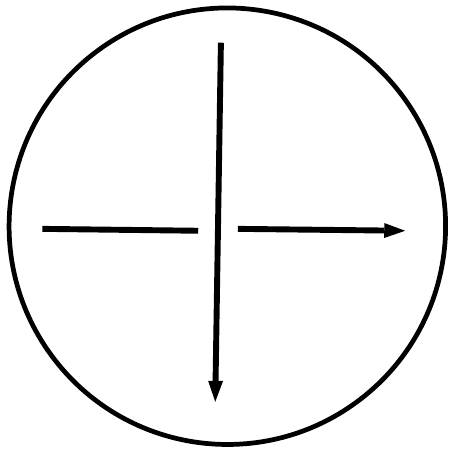}
         \caption{
         \label{fig: Asumframed 3}}
     \end{subfigure}
      \begin{subfigure}[b]{0.30\textwidth}
         \centering
         \includegraphics[width=0.65\textwidth]{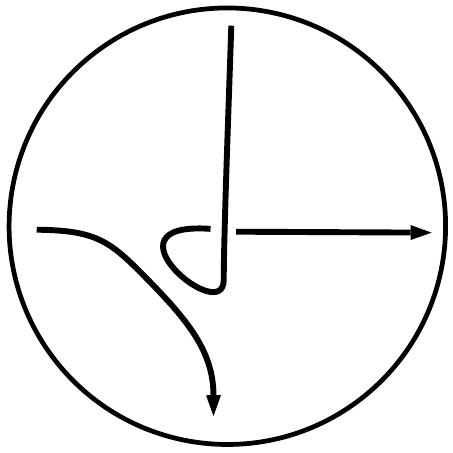}
         \caption{
         \label{fig: Asumframed 2}}
     \end{subfigure}
     \begin{subfigure}[b]{0.30\textwidth}
         \centering
         \includegraphics[width=0.65\textwidth]{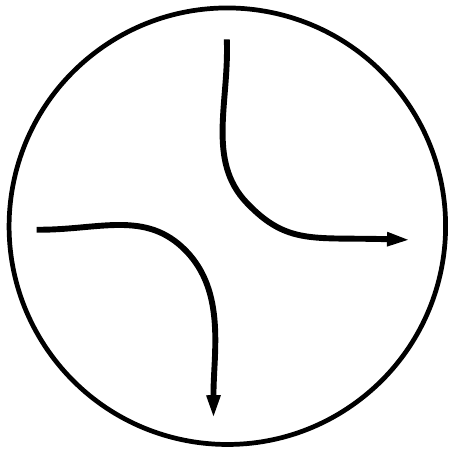}
         \caption{
         \label{fig: Asumframed 1}}
     \end{subfigure}
    \caption{ In the proof of Proposition \ref{prop: Sumofframed curves on a surface} the figures represent $D^1 \times D^3$ where the $D^1$ factor is suppressed. All the curves in the pictures are equipped with the black board framing. \label{fig: Asum of framed curve}} 
\end{figure}

It is possible to generalize Lemma \ref{prop: Sumofframed curves on a surface} as follows:
\begin{prop}\label{prop: curves obtaiend by dehn twist}
    Let $\alpha,\beta\subseteq \Sigma$ be two framed curves on a surface $\Sigma\subseteq X$ which is both $\nu\alpha$-standard and $\nu\beta$-standard. Assume that the curves $\alpha$ and $\beta$ intersect transversely in $n$ points. Let $\gamma$ the curve obtained from $\alpha$ by the Dehn twist $D_\beta$ defined by $\beta$, i.e. $\gamma:=D_\beta(\alpha)$.
    There exists a framing for $\gamma$ such that $\Sigma$ is $\nu\gamma$-standard and such that the framed curve $\overline{\gamma}$ is isotopic to $(\overline{\alpha}\#n\overline{\beta})_{n\cdot op}$ in $X$ relative to $\Sigma$. Here $\#n\overline{\beta}$ denotes $n$ band sums with $n$ parallel copies of the framed loop $\overline{\beta}$, possibly with  reversed orientation and $n\cdot op$ means that the opposite framing should be considered if $n$ is odd.
\end{prop}

The following is a useful proposition about the stabilization sets for surfaces with a cyclic fundamental group of the complement, with special attention to the trivial case.
We say that a compact connected surface $\Sigma\subseteq X$ is a \emph{characteristic surface} if, for any immersed, orientable, closed surface  $S\subseteq X$ transversal to $\Sigma$, the number of points of the intersection  $\Sigma\cap S$ equal the self intersection of $S$ modulo 2. Otherwise, we say that $\Sigma$ is an \emph{ordinary surface}. In the case of $X$ closed oriented and simply connected 4-manifold, $\Sigma$ is characteristic if and only if $[\Sigma]\in H_2(X;\Z)$ is a characteristic class, i.e. its modulo 2 reduction is the Poincare dual of the second Stiefel-Whitney class $w_2(X)\in H^2(X;\Z_2)$. For a non-connected, but still compact, surface $\Sigma\subseteq X$, we say that a component $C\subseteq\Sigma$ is a \emph{characteristic component} if $C$ is a characteristic surface in $X\bslash (\Sigma\bslash C)$. Otherwise, we say that $C$ is an \emph{ordinary component} of $\Sigma$.

\begin{prop}
    \label{prop: Propriety of S}
    Let $\Sigma\subseteq X$ be a surface in a simply connected 4-manifold $X$.
    \begin{enumerate}
          
        \item \label{prop: SS item 1}If $\pi_1(X\bslash \Sigma)$ is the cyclic group $\Z/k\Z$ for $k\in\N$ (possibly $k=0$) and $\gamma\subseteq \Sigma$ is a simple curve, then 
        $\overline{\Ss}(\Sigma,\gamma)\not=\emptyset$. 
        Moreover, if $f_0$ is a framing for $\gamma$ such that $\Sigma$ is $\nu\gamma$-standard, the small push off of $\gamma$ is nullhomotopic in $X\bslash\Sigma$ and $\Sigma$ is connected, then \[\overline{\Ss}(\Sigma,\gamma)=\bigcup_{n\in\Z} \Ss\left(\Sigma,(\gamma,f_{kn})\right).\]

         \item \label{prop: SS item 2} If $\Sigma\subseteq X$ has simply connected complement and $\gamma\subseteq \Sigma$ is a simple loop in an ordinary component of $\Sigma$, then $\overline{\Ss}(\Sigma,\gamma)=\{\Stab,\StabTwist\}$.
        \item  \label{prop: SS item 3} If $\Sigma\subseteq X$ has simply connected complement and $\gamma\subseteq \Sigma$ is a simple loop in a characteristic component of $\Sigma$, then $\overline{\Ss}(\Sigma,\gamma)$ has exactly one element. 
       \item \label{prop: SS item 4} If a simple loop $\gamma\subseteq\Sigma$ is nullhomotopic in $\Sigma$, then $\Stab\in\overline{\Ss}(\Sigma,\gamma)$. 
         \item \label{prop: SS item 5} If $\Sigma\subseteq X$ has simply connected complement and it has a $T^2$ factor, then there exists a simple non-separating loop $\gamma\subseteq \Sigma$ such that $\StabTwist\in\overline{\Ss}(\Sigma,\gamma)$.
    \end{enumerate}
\end{prop}
\begin{proof} Point \ref{prop: SS item 1}.
Let $\pi_1(X\bslash \Sigma)$ be a cyclic group and $\gamma\subseteq \Sigma$ be a simple curve. Fix a framing $f_0$ of $\gamma$ such that $\Sigma$ is $\nu\gamma$-standard. Consider $f_n$ defined as in equation (\ref{eq: framings nucompatible}). Since $\pi_1(X\bslash \Sigma)$ is cyclic and generated by the meridian $\mu$ equation (\ref{eq: pi1 framings}) implies that there exists $n\in\Z$ such that $\overline{\gamma}(n)$ is nullhomotopic in $X\bslash\Sigma$. By Lemma  \ref{lem: proprieties Ss} the stabilization set $\Ss(\Sigma,(\gamma,f_n))$ is not empty, and so it is  $\overline{\Ss}(\Sigma,\gamma)$.\\
Assuming now that $f_0$ has a small push off that is nullhomotopic in the complement of $\Sigma$, by equation (\ref{eq: framings nucompatible}) we have that $\overline{\gamma}(n)$ is nullhomotopic in $X\bslash\Sigma$ if and only if $n$ is divisible by $k$. It follows that 
\[\overline{\Ss}(\Sigma,\gamma)=\bigcup_{n\in\Z} \Ss\left(\Sigma,(\gamma,f_{n})\right)=\bigcup_{n\in\Z} \Ss\left(\Sigma,(\gamma,f_{kn})\right)\]
by Lemma \ref{lem: BarSs is union of Ss} and Lemma  \ref{lem: proprieties Ss}.

Point \ref{prop: SS item 2}.
Under our assumption $\Sigma$ has a geometrically dual immersed 2-sphere of even self intersection, which intersect $\Sigma$ in the component containing the loop $\gamma$ (see the proof of \cite{Auchly1stb} for more details on this fact).
We can apply Proposition \ref{lem: geometricDualSphere} to conclude that $\overline{\Ss}(\Sigma,\gamma)=\{\Stab,\StabTwist\}$.

Point \ref{prop: SS item 3}.
Since $X\bslash\Sigma$ is simply connected, we have $\overline{\Ss}(\Sigma,\gamma)$ is not empty by Proposition \ref{prop: Propriety of S} point \ref{prop: SS item 1}, moreover, there exists a framing $f_0$ of $\gamma$ such that $\overline{\Ss}(\Sigma,\gamma)=\bigcup_{n\in\Z}\Ss(\Sigma,(\gamma,f_n))$. Our hypothesis imply that $\Sigma$ has a geometrically dual immersed 2-sphere of odd self intersection, which intersect $\Sigma$ in the component containing the loop $\gamma$.  By Proposition \ref{lem: geometricDualSphere} we get $\overline{\Ss}(\Sigma,\gamma)=\Ss(\Sigma,(\gamma,f_0))$.  Assuming by contradiction that $\overline{\Ss}(\Sigma,\gamma)$ has not exactly one element, it should have two. By Lemma \ref{lem: proprieties Ss} point \ref{prop: SS item 3},
there exists an immersed 2-sphere $A\subseteq X\bslash\nu\Sigma$ with odd self intersection, which is in contradiction with our assumption! Therefore,  the stabilization set $\overline{\Ss}(\Sigma,\gamma)$ has exactly one element.

Point \ref{prop: SS item 4}.
The curve $\gamma$ bounds a 2-disk $D\subseteq\Sigma$, fix a trivialization for a tubular neighborhood of $D$. We can equip $\gamma$ with a framing $f$ compatible with $D$ such that the surface $\Sigma$ is $\nu\gamma$-standard.
The small push off $\overline{\gamma}$ bounds a 2-disk $\overline{D}$ that is a push off of $D$, and the framing of $\overline{\gamma}$  is compatible with $\overline{D}$. It follows that $\Stab\in \Ss(\Sigma,(\gamma,f))\subseteq\overline{\Ss}(\Sigma,\gamma)$.

Point \ref{prop: SS item 5}.
Pick two curves $\alpha$ and $\beta$ on the $T^2$ factor of $\Sigma$ that intersect each other transversely in one point. There exist framings for $\alpha$ and $\beta$ such that the small push off $\overline{\alpha}$ and $\overline{\beta}$ are nullhomotopic in $X\bslash\Sigma$. Therefore, the corresponding  stabilizing sets are not empty. If the manifold $\StabTwist$ is in $\Ss(\Sigma, \alpha)$ or in $\Ss(\Sigma,\beta)$, then we can conclude picking $\gamma:=\alpha$ or $\gamma:=\beta$. Otherwise $\{\Stab\}=\Ss(\Sigma,\alpha)\cup\Ss(\Sigma,\beta)$, in this case pick the framed curve $\gamma:=\alpha+\beta\subseteq \Sigma$ as in Proposition \ref{prop: Sumofframed curves on a surface}, and we get that $\overline{\gamma}$ is isotopic to a framed curve $(\overline{\alpha}\#\overline{\beta})_{op}=\bar\alpha\#(\bar\beta_{op})$. The loop $\bar\alpha$ bounds an embedded 2-disk  $D_\alpha$ in $X\bslash\Sigma$ and the framing of $\alpha$ is compatible with $D_\alpha$. Similarly, $\beta$ has a framing compatible with a 2-disk $D_\beta\subseteq X\bslash\Sigma$. By Lemma \ref{lem: ConSumOfLoop and Disk} the framed curve $\overline{\gamma}$ is isotopic to a framed curve $\overline{\beta}_{op}$. It follows that $\StabTwist\in \Ss(\Sigma,\gamma)$ by Proposition \ref{prop: equivalent definition}.
\end{proof}

\begin{oss}\label{rem: Rocklin}
Assume that $\Sigma$ is a closed, oriented, connected surface in a closed 4-manifold $X$. Assume also that $\Sigma$ is a characteristic surface.
    We have to point out an interplay between the extended stabilization set $\overline{\Ss}(\Sigma,\gamma)$ and the Rokhlin quadratic form $q:H_1(\Sigma;\Z)\to\Z_2$ defined in \cite{rokhlin1972proof}. 
    One can check that the following implication holds: if $\Stab\in \overline{\Ss}(\Sigma,\gamma)$, then $q([\gamma])=0$. 
    The inverse implication in general is false, since it possible that  $q([\gamma])=0$ and $\overline{\Ss}(\Sigma,\gamma)=\emptyset$. 
\end{oss}

\section{Application: external stabilization of surfaces\label{sec: application}}
Theorem \ref{mainthm: int to ext} gives us a way to use what we know about internal stabilization to prove results about external stabilization.
 \cite[Section 3]{InternalStabilizationBaykurSunukian} explicitly describes isotopies between exotically knotted surfaces, from the literature, after one internal stabilization. These isotopies are compositions of "elementary" isotopies of two types.
 The first type is given by the repositioning of a stabilization arc through the manifold. The second isotopies are local in nature and are performed in a well-chosen neighborhood of the surfaces, in particular some of them are described in \cite[Proposition 9]{InternalStabilizationBaykurSunukian}.
 If we want to apply Theorem \ref{mainthm: int to ext}, we must first show that these surfaces are smoothly isotopic after one internal stabilization relative to the neighborhood of a well-chosen loop $\gamma$, i.e. internally $\gamma$-stably isotopic. Therefore, we need to check that any elementary isotopy can be chosen to fix ${\nu_{1/2}\gamma}$. As one can imagine, it is not always true that the first type of isotopy can be chosen to be relative to $\nu_{1/2}\gamma$. Instead, the second type of isotopy is not a big problem in the right framework and with a suitable choice of $\gamma$, as we will see.


The following proposition will be very useful to repositioning cords  without moving a loop $\gamma\subseteq\Sigma$.
\begin{prop}\label{Prop repositioning arc}
    Let $\Sigma\subseteq X$ an oriented $\nu\gamma$-standard surface. Consider two orientation preserving internal stabilization of $\Sigma$ away from $\gamma$
    defined by 3-dimensional handles $h_i$ with core arc $\alpha_i$ for $i=1,2$.
    If $X$ is simply connected,  $\pi_1(X\bslash\Sigma)$ is cyclic generated a meridian of $\Sigma$ over a point of $\partial\alpha_1$ and the sets $\partial \alpha_1 $ and $\partial\alpha_2$ are homotopic in $\Sigma\bslash\gamma$, then the surfaces defined as internal stabilizations of $\Sigma$ by the 3-dimensional handles $h_1$ and $h_2$ are smoothly isotopic in $X$ relative to $\nu_{1/2}\gamma\cup \partial X$.
\end{prop}
\begin{proof}
Since the internal stabilization is determined, up to isotopy, by the homotopy class of the core cords (see \cite[Section 2]{Boyle_1988}), it will be sufficient to show the following claim: the cords $\alpha_1 $ and $\alpha_2$ are homotopic in $X\bslash(\nu_{1/2}\gamma)$ keeping the boundary of the arcs on the surface $\Sigma\bslash\gamma$. 
To start, by our hypothesis, we can position $\alpha_1$ and $\alpha_2$ to make them share the endpoints on $\Sigma$.

We can \emph{align} the cords $\alpha_1$ and $\alpha_2$ such that they coincide in a small tubular neighborhood of $\Sigma$, see figure \ref{fig: stabilizationRelGamma}. It is defined a loop $l:=(\alpha_1\cup\alpha_2)\bslash\nu\Sigma$ and our aim should be to prove that this loop is nullhomotopic in $X\bslash\Sigma$. But this in not always the case, the loop $l$ represents an element in the fundamental group $\pi_1(X\bslash\Sigma)$ that by assumption is cyclic and generated by a meridian $\mu$ over an endpoint of the arc $\alpha_1$. So a priory the class of the loop $[l]$ is $[\mu]^d$ for some $d\in\Z$. The problem lies in the way in which we had \emph{aligned} the cords $\alpha_1$ and $\alpha_2$. As explained in the proof of \cite[Lemma 3]{InternalStabilizationBaykurSunukian} we can compose any of the two cords with a meridian over the points $\partial\alpha_1$, this will change the homotopy class of $l$. Therefore, we can assume the homotopy class $[l]$ to be trivial after performing $d$ times this operation.
\end{proof}
\begin{figure}
    \centering
    \includegraphics[width=0.3\textwidth]{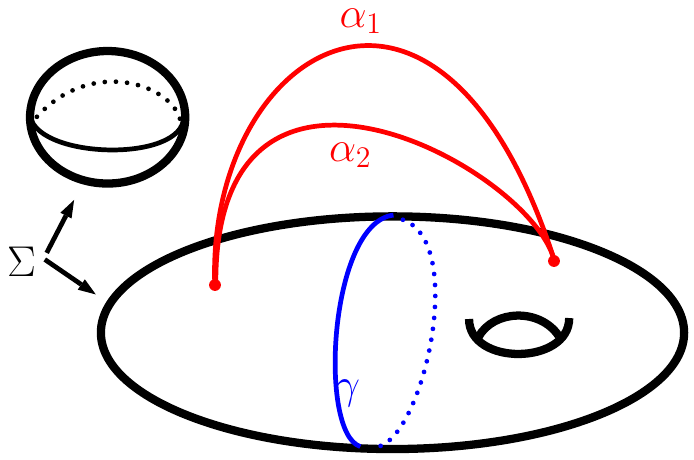}
    \caption{
    \label{fig: stabilizationRelGamma}}
\end{figure}
\subsection{Rim-Surgery\label{sec: Rim Surgery}}
In this section we will show that any surface $\Sigma_{K,\beta}\subseteq X$ produced via rim-surgery \cite{Fintushel1997SurfacesI4} from $\Sigma\subseteq X$ becomes smoothly isotopic to $\Sigma$ in $X\# (\Stab)$ or in $X\# (\StabTwist)$.
Let us first recall the rim-surgery operation. 
\begin{constr}[Rim-surgery \cite{Fintushel1997SurfacesI4}]\label{const: Rim surgery}
    Let $X$ be a 4-manifold, fix a framed loop $\beta\subseteq X$ with a framing $f:\nu\beta\to S^1\times D^3$ and consider a $\nu\beta$-standard surface $\Sigma\subseteq X$, i.e.
\[\Sigma=(\Sigma\bslash\nu\beta)\cup \left(f^{-1}\left(S^1\times D^1\right)\right)\subseteq X.\]
For any oriented knot $K\subseteq S^3$ we define a surface $\Sigma_{K,\beta}\subseteq X$ by setting it to be equal to $\Sigma$ outside $\nu\beta$ and equal to $S^1\times (D^1\# K)\subseteq S^1\times D^3\cong{\nu\beta}$,
\begin{equation}
    \Sigma_{K,\beta}=\Sigma_{K,\beta}(f):=(\Sigma\bslash\nu\beta)\cup \left(f^{-1}\left(S^1\times(D^1\# K)\right)\right)\subseteq X,
\end{equation}where $(D^1\# K)$ denotes the arc in $D^3$ knotted with the knot $K$.
\end{constr} Notice that the surface $\Sigma_{K,\beta}$  can be defined starting from any ${\nu\beta}$-standard surface $\Sigma$, and it is not necessary to ask for connectedness, compactness or orientability.

\begin{oss}
    If $\beta$ bounds a 2-disk in $\Sigma$, then there exists a framing $f_0$ of $\beta$ such that $\Sigma_{K,\beta}\subseteq X$ is smoothly isotopic to $\Sigma\# K^*\subseteq X\# S^4$, where $K^*\subseteq S^4$ is the spun of the knot $K$. The framing $f_n$ produces $\Sigma\# K^*_n$, where $K^*_n$ is the $n$-twist spun of the knot $K$ \cite{zeeman1965twisting}. 
\end{oss}

We would like to compare the surfaces $\Sigma,\Sigma_{K,\beta}\subseteq X$. They share the same homology class by construction and if $X\bslash\Sigma$ is simply connected, then also $X\bslash \Sigma_{K,\beta}$ is. The surfaces $\Sigma,\Sigma_{K,\beta}\subseteq X$ can form an exotic pair, as explained in \cite{Fintushel1997SurfacesI4,fintushel2005surfacesAddendum,mark2013knotted}.

\begin{prop} \label{Prop Rim-Surgery internal}
Let $\Sigma,\Sigma_{K,\beta}\subseteq X$ be defined as in Construction \ref{const: Rim surgery}. Furthermore, assume that $\Sigma$ is an orientable surface. 
There exists a framed loop $\gamma$ which is a parallel copy of $\beta$ in $\Sigma$ such that the surfaces $\Sigma,\Sigma_{K,\beta}\subseteq X$ are ${\nu\gamma}$-standard.
  If  $\pi_1(X\bslash \Sigma)$ is trivial, then the surfaces $\Sigma,\Sigma_{K,\beta}\subseteq X$ are internally $\gamma$-stably isotopic rel. boundary.
\end{prop}
\begin{oss}
       The argument in \cite[Section 3.1]{InternalStabilizationBaykurSunukian} gives us the statement under the additional assumption that $\beta$ is a non-separating loop in $\Sigma$ (and thus $\Sigma$ has positive genus). The proof of Proposition \ref{Prop Rim-Surgery internal} follows and shares the same ideas, but it removes this restriction. 
\end{oss}
\begin{proof}
    During this proof we will identify $S^1\times D^3$ with ${\nu\beta}\subseteq X$ by the framing of $\beta$.
    Pick a plane diagram for the knotted arc $A_0=D^1\# K\subseteq D^3$ and orient it starting from the base point $1\in \partial D^1\subset D^1\subseteq D^3$. We can assume also that the connected sum is performed inside $\frac{1}{2} D^3\subseteq D^3$. Starting from the base point of $A_0$ and following the orientation enumerate the crossing $c_1,\dots, c_n$ and define inductively the \emph{unknotting sequence} $A_0,\dots,A_n$ of the arc $A_0$ as follows. 
    Starting from the base point and following the orientation,  stop when you are entering the crossing region $c_i$ for the first time. If you stop on the overcrossing arc, then define $A_i=A_{i-1}$.
    Otherwise, define $A_i$ to be the arc $A_{i-1}$ with a crossing change at $c_i$, see Figure \ref{fig: crossings}. Notice that following the orientation on the arc $A_n$ we enter each new crossing region always from the overcrossing arc. It follows that $A_n$ smoothly isotopic to the unknotted arc $D^1\subseteq D^3$ relative to $D^3\bslash (\frac{1}{2}D^3)$, see for example \cite[Exercise 2.B.11]{Rolfsen_2004}).
    Consider the surfaces $\Sigma_0,\dots,\Sigma_n$ in $X$ that differ only in the interior of ${\nu\beta}$ for which $\Sigma_0=\Sigma_{K,\beta}$ and
    \[\forall i=1,\dots,n\qquad({\nu \beta},{\nu \beta}\cap\Sigma_i)\cong (S^1\times D^3, S^1\times A_i).\]
    Notice that $\Sigma_n,\Sigma\subseteq X$ are smoothly isotopic relative to the complement of ${\nu_{1/2}\beta}\cong S^1\times (\frac{1}{2}D^3)$.\\
    
    Take $\gamma:=S^1\times \{3/4\}\subseteq S^1\times D^1\subset S^1\times D^3\cong {\nu\beta}\subseteq X$ and  ${\nu\gamma}:=S^1\times (\frac{1}{8}D^3+\frac{3}{4}e_1)
    $ with the canonical framing. Notice that $\Sigma_i$, $\Sigma$ and $\Sigma_{K,\beta}$ are ${\nu\gamma}$-standard.
    We would like to prove that $\Sigma_{i-1},\Sigma_{i}\subseteq X$ are internally $\gamma$-stably isotopic, i.e. they become smoothly isotopic relative to ${\nu_{1/2}\gamma}$ after a trivial internal stabilization across $\gamma$.
    
    If $\Sigma_{i-1}=\Sigma_{i}$ then there is nothing to prove, otherwise the crossing $c_i$  as in Figure \ref{fig: undecrossing} has been replaced with the one in \ref{fig: Overcrossing}. We can connect a small 3-disk around  the crossing $c_i$ with an arc passing under the diagram of the knot to a small 3-ball centered in a point in the interval $(7/8, 1)\subseteq D^1 \subseteq D^3$, moreover this arc can be chosen to be disjoint from ${\nu\gamma}$. Therefore, there exists a 3-disk $D$ within the arcs $A_{i-1}$ and $A_i$ are positioned as in Figure \ref{fig: Goodpos1} and Figure \ref{fig: Goodpos6} respectively. This gives us a local representation of the surfaces $\Sigma_{i-1}$ and $\Sigma_i$ in $S^1\times D\subseteq S^1\times D^3\cong{\nu\beta}$.
A trivial internal stabilization across $\gamma$ to $\Sigma_{i-1}$ gives us the internal stabilization along the dotted red cord in Figure \ref{fig: Goodpos1}, indeed we can move the stabilization arc by Proposition \ref{Prop repositioning arc} since $\pi_1(X\bslash\Sigma_{i-1})$ is trivial. The same holds for the surface $\Sigma_i$ and the red arc described in Figure \ref{fig: Goodpos6}.
    \\
    We now perform a collection of isotopies, relative to ${\nu_{1/2}\gamma}$, that changes the  surfaces only in $S^1\times D$ as indicated in Figure \ref{fig: GoodPos}, where the $S^1$ factor is suppressed. We can use \cite[Proposition 9]{InternalStabilizationBaykurSunukian} to move between Figures \ref{fig: Goodpos1}-\ref{fig: Goodpos2}, \ref{fig: Goodpos3}-\ref{fig: Goodpos4} and  \ref{fig: Goodpos5}-\ref{fig: Goodpos6}. It is left to describe the isotopy between Figures \ref{fig: Goodpos2}-\ref{fig: Goodpos3} and  \ref{fig: Goodpos4}-\ref{fig: Goodpos5}. The argument is similar in the two cases, so we will focus on the first one. To move the cord $\alpha_1$ to the cord $\alpha_2$ we would like to use Proposition \ref{Prop repositioning arc}, so we have to verify the assumptions. So, consider the surface $S$ which is equal to $ \Sigma_i$ outside $S^1\times D$ and intersect it in $S^1\times (t_1 \cup t_2 \cup u )\subseteq S^1\times D\subseteq S^1\times D^3\cong{\nu\beta}$. Notice that $S$ can be written as two disjoint surfaces $T\sqcup U$,   where $U$ is the connected component that contains $S^1\times u$ and $T$ is the union of all other components. By construction $S^1\times t_1$ and $S^1\times t_2$ are in the same connected component of $T$, since we are under the assumption of first entering the crossing region of $c_i$ from the overcrossing arc. Observe that $\gamma\subseteq U$ and that $U=S^1\times \Tilde{u}\subseteq S^1\times D^3$, where $\Tilde{u}$ is an oriented loop in $D^3$. We can see that $\Tilde{u}$ is an unknotted loop in $D^3$. Indeed, $\Tilde{u}$ has a planar diagram for which, following the orientation, we enter each new crossing region from the overcrossing arc, here we are using the definition of the arc $A_{i-1}$ and our construction of $D$. Therefore, $U$ is a smoothly unknotted 2-torus and the fundamental group of its complement is isomorphic to the integers $\Z$. Since a meridian  of $\Sigma$ is nullhomotopic ($\pi_1(X\bslash\Sigma)=1$), it bounds an immersed 2-disk that we can assume to be disjoint from $\nu\beta$, it follows that also any meridian of $T$ is nullhomotopic in $X\bslash (T\sqcup U)=X\bslash S$. We have proved that the fundamental group of the complement of $S$ is infinite cyclic, so we are under the assumption of Proposition \ref{Prop repositioning arc} and we can reposition the cord $\alpha_1$ to the cord $\alpha_2$ with a smooth isotopy relative to ${\nu_{1/2}\gamma}$.\\
    This proves that $\Sigma_{i-1},\Sigma_i\subseteq X$ are internally $\gamma$-stably isotopic, it implies that $\Sigma_{K,\beta},\Sigma\subseteq X$ are internally $\gamma$-stably isotopic.
    \end{proof}

\begin{figure}[h]
     \centering
     \begin{subfigure}[b]{0.45\textwidth}
         \centering
         \includegraphics[width=0.40\textwidth]{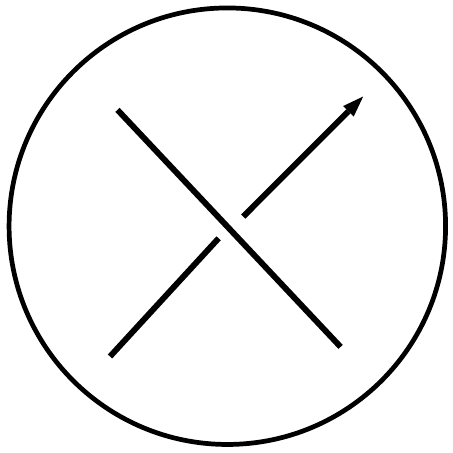}
         \caption{Undercrossing
         \label{fig: undecrossing}}
     \end{subfigure}
     \begin{subfigure}[b]{0.45\textwidth}
         \centering
         \includegraphics[width=0.40\textwidth]{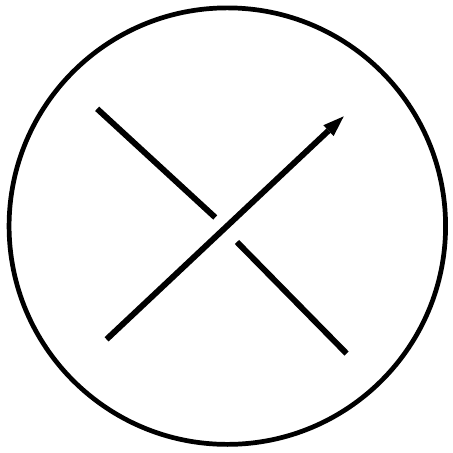}
         \caption{Overcrossing
         \label{fig: Overcrossing}}
     \end{subfigure}
     \caption{\label{fig: crossings}}
\end{figure}

\begin{figure}[h]
     \centering
     \begin{subfigure}[b]{0.25\textwidth}
         \centering
         \includegraphics[width=0.8\textwidth]{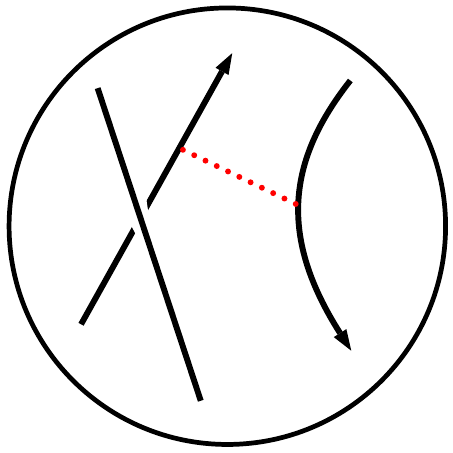}
         \caption{
         \label{fig: Goodpos1}}
     \end{subfigure}
     \hfill
     \begin{subfigure}[b]{0.25\textwidth}
         \centering
         \includegraphics[width=0.8\textwidth]{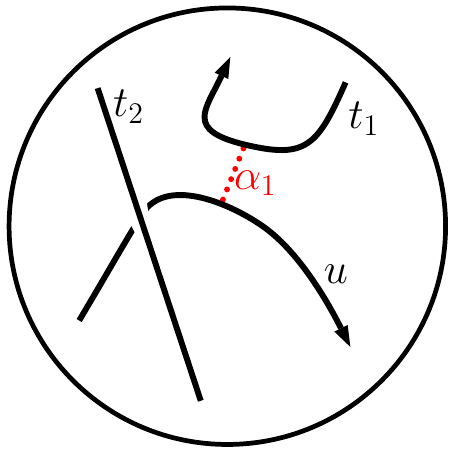}
         \caption{
         \label{fig: Goodpos2}}
     \end{subfigure}
     \hfill
     \begin{subfigure}[b]{0.25\textwidth}
         \centering
         \includegraphics[width=0.8\textwidth]{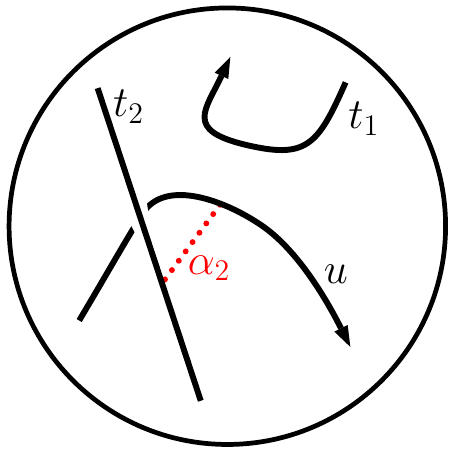}
         \caption{         \label{fig: Goodpos3}}
     \end{subfigure}

     \begin{subfigure}[b]{0.25\textwidth}
         \centering
         \includegraphics[width=0.8\textwidth]{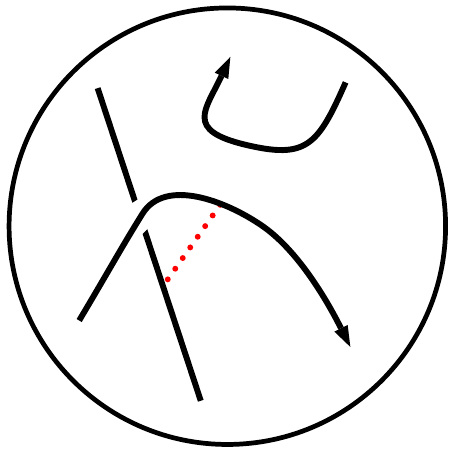}
         
         \caption{\label{fig: Goodpos4}}
     \end{subfigure}
     \hfill
     \begin{subfigure}[b]{0.25\textwidth}
         \centering
         \includegraphics[width=0.8\textwidth]{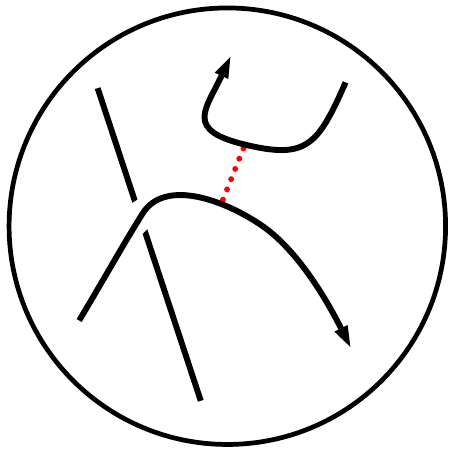}
                  \caption{\label{fig: Goodpos5}}
     \end{subfigure}
     \hfill
     \begin{subfigure}[b]{0.25\textwidth}
         \centering
         \includegraphics[width=0.8\textwidth]{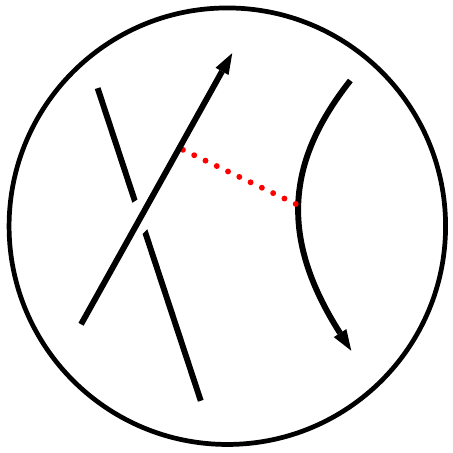}
         \caption{\label{fig: Goodpos6}}
     \end{subfigure}
     \caption{\label{fig: GoodPos}The arcs $A_{i-1}$ and $A_i$ in the 3-disk $D\subseteq D^3$ are represented in Figures \ref{fig: Goodpos1} and \ref{fig: Goodpos6}. The surfaces $\Sigma_{i-1}$ and $\Sigma_i$ differs only in $S^1\times D\subseteq S^1\times D^3\cong{\nu\beta}$. The dotted red arcs are stabilization chords. }
\end{figure}
\begin{thm}\label{thm: ExtStab RimSurgery}
Let $\Sigma,\Sigma_{K,\beta}\subseteq X$ be defined as in Construction \ref{const: Rim surgery}. Furthermore, assume $\Sigma$ to be an orientable surface.
 If  $\pi_1(X\bslash \Sigma)$ is trivial, then for any $B\in\overline{\Ss}(\Sigma,\beta)\subseteq\{\Stab,\StabTwist\}$ the surfaces $\Sigma,\Sigma_{K,\beta}\subseteq X$  are $B$-stable isotopic rel. boundary.
 Moreover, 
  \begin{equation}\label{eq: StabSet RimSurgery}
      \overline{\Ss}(\Sigma,\beta)=\begin{cases}\{\Stab,\StabTwist\}, &\text{ if $\beta$ is in a ordinary component of $\Sigma$}\\
    \{\Stab\}\text{ or }\{\StabTwist\}, &\text{ if $\beta$ is in a characteristic component of $\Sigma$}
    \end{cases}.
  \end{equation}
\end{thm}
\begin{proof}
    We can combine Proposition \ref{Prop Rim-Surgery internal} and Theorem \ref{mainthm: int to ext} to obtain that $\Sigma,\Sigma_{K,\beta}\subseteq X$ are  $B$-stably isotopic rel. boundary for any $B\in\overline{\Ss}(\Sigma,\Sigma_{K,\beta},\gamma)$, where $\gamma$ is a parallel copy of $\beta$ in $\Sigma$. We have $\Sigma_{K,\beta}\subseteq\nu\Sigma$ by definition  and we get $\overline{\Ss}(\Sigma,\Sigma_{K,\beta},\gamma)=\overline{\Ss}(\Sigma,\gamma)=\overline{\Ss}(\Sigma,\beta)$ by Lemma \ref{lem: Sdi2 in fuzione di S1}. 
    To compute the stabilization set as in equation (\ref{eq: StabSet RimSurgery}) we can use Proposition \ref{prop: Propriety of S}: point \ref{prop: SS item 2} if $\beta$ is in an ordinary component of $\Sigma$ and point \ref{prop: SS item 3} if it is in a characteristic one. 
\end{proof}

\subsection{\label{sec: twistRim}Twist rim-Surgery}
Twist rim-surgery \cite{Kim_2006} is a variation of the rim surgery operation, but it can be expressed in terms of a rim-surgery operation with a different choice of framing. 
\begin{constr}[$n$-twist rim-surgery \cite{Kim_2006,Kim_Ruberman_2008}]\label{constr: twist rim}
    Let $X$ be a smooth connected 4-manifold, fix a framed loop $\beta\subseteq X$ with a framing $f_0$ and consider a smoothly embedded $\nu\beta$-standard surface $\Sigma\subseteq X$.
Moreover, we require that the small push off of $\beta$ is nullhomotopic in the complement of $\Sigma$. 
For any $n\in\Z$ and for any oriented knot $K\subseteq S^3$ we define the surface 
\begin{equation}
\Sigma_{K,\beta,n}=\Sigma_{K,(\beta,f_0),n}:=\Sigma_{K,\beta}(f_n)\subseteq X,
\end{equation}  
where $\Sigma_{K,\beta}(f_n)$ is the surface produced by rim-surgery (Construction \ref{const: Rim surgery}) with framing $f_n$ (Construction \ref{constr: framings}).
\end{constr}

    Construction \ref{constr: twist rim} is equivalent to the definition of rim-surgery given in \cite[Definition 2.1 and 2.3]{Kim_2006}.  Using the twist rim-surgery operation it is possible to define an exotic pair of surfaces $\Sigma, \Sigma_{K,\beta,n}\subseteq\C P^2$, see \cite[Corollary 3.5 and Theorem 4.5]{Kim_2006}, where $\Sigma$ is an algebraic surface with degree $\geq 3$.

\begin{thm}\label{thm: external Twist Rim Surgery}\label{thm: External Twist Rim Surgery}
Let $\Sigma,\Sigma_{K,\beta,n}\subseteq X$ be defined as in Construction \ref{constr: twist rim}. Let us further assume that $\Sigma$ is an orientable surface and $\beta\subseteq\Sigma$ is non-separating.
 If  $\pi_1(X\bslash \Sigma)$ and $\pi_1(X\bslash\Sigma_{K',\beta,n})$ are cyclic groups generated by the meridians for any knot $K'\subseteq S^3$, then $\Sigma,\Sigma_{K,\beta,n}\subseteq X$  are $B$-stably isotopic rel. boundary for any $B\in\overline{\Ss}(\Sigma,\beta)\subseteq\{\Stab,\StabTwist\}$. 
 Moreover, the set $\overline{\Ss}(\Sigma,\beta)$ is not empty.
\end{thm}
\begin{proof}
Fix $\gamma\subseteq\Sigma$ to be a parallel copy of $\beta\subseteq\Sigma$.
    The surfaces $\Sigma,\Sigma_{K,\beta,n}\subseteq X$ are internally 1-stable isotopic by \cite[Section 3.2]{InternalStabilizationBaykurSunukian}, moreover,  the proof can be modified to show that the surfaces are internally $\gamma$-stably isotopic relative to the boundary. To do so, every time is needed to move a stabilization arc we apply Proposition \ref{Prop repositioning arc}, here are essential the assumptions that $\beta$ is non separating ($\Sigma$ has positive genus) and that the fundamental group of the complement of $\Sigma_{K',\beta,n}$ is always cyclic for any knot $K'$. It follows from Theorem \ref{mainthm: int to ext} that for any $B\in\overline{\Ss}(\Sigma,\Sigma_{K,\beta,n},\gamma)$ the surfaces $\Sigma,\Sigma_{K,\beta,n}\subseteq X$ are $B$-stable smoothly isotopic. By Lemma \ref{lem: Sdi2 in fuzione di S1} we conclude that $\overline{\Ss}(\Sigma,\Sigma_{K,\beta,n},\gamma)=\overline{\Ss}(\Sigma,\beta)$. This set is not empty by Proposition \ref{prop: Propriety of S} point \ref{prop: SS item 1}.
\end{proof}
\begin{oss}
    In Theorem \ref{thm: External Twist Rim Surgery}, if $\pi_1(X\bslash\Sigma)$ is the cyclic group $\Z/d\Z$ (possibly $d=0$) and $(d,n)=1$, then $\pi_1(X\bslash\Sigma_{\beta,K',n})\cong \Z/d\Z$ for any knot $K'\subseteq S^3$ by \cite[Proposition 2.4]{Kim_Ruberman_2008}.
\end{oss}
\subsubsection{Concrete examples\label{sec: twist rim application}}
Let us consider the infinite exotic family $\{T_i\}_{i\in\N}$ of punctured 2-tori in $B^4$ bounding a knot $K$ in $S^3$ of \cite[Theorem 1.1]{MillerZenke2021transverse}.  The surface $T_0$ is obtained from a Seifert surface of $K$ by pushing the interior of it in the interior of $B^4$. It follows that $\pi_1(B^4\bslash T_0)\cong \Z$. All the 2-tori $T_i$ are obtained from a 2-torus $T_0$ by doing 1-twist rim surgery on a non separating curve $\beta\subseteq T_0$. The framing of $\beta$ is induced by a locally flat embedded 2-disk $D$ with interior in the complement of $B^4\bslash T_0$. We can perturb the interior of 2-disk $D$ to find a smoothly immersed 2-disk $D'$ and easily conclude that $(\Stab)\in \overline{\Ss}(T_0,\beta)$. By Theorem \ref{thm: External Twist Rim Surgery} the surfaces $T_0,T_i\subseteq B^4$ are $(\Stab)$-stably isotopic rel. boundary for any $i\in\N$.\\
The same result can be obtained using \cite[Remark 2.6 and Corollary 2.7]{MillerZenke2021transverse}. For doing this, it is sufficient to find a smoothly embedded 2-disk $D''\subseteq B^4\#(\Stab)$ bounding $\beta$, with interior disjoint from $T_0$ and inducing on $\beta$ the same framing of $D$. We can construct $D''$ from $D'$ by removing its self intersection using the new $\Stab$ summand (see \cite[Lemma 1]{norman1969dehn}).\\

Theorem \ref{thm: External Twist Rim Surgery} is only valid for surfaces with cyclic fundamental group of the complement, but it can also be used to prove the $B$-stable isotopy for exotic surfaces with  more complicated fundamental group, as the next example will show.

 In \cite[Theorem B]{hayden2021exotically} there are constructed two infinite collection of closed surfaces  in a closed or convex symplectic 4-manifold $X$, the ones of point $(b)$ have non-cyclic fundamental group of the complement. Let $\{F_i\}_{i\in\N}$ be one of those two collections.  Also this collection is obtained from one element $F_0$ by 1-twist rim surgery. The twist rim surgery is performed along a curve $\beta\subseteq F_-\subseteq F_0$ where $F_-$ is $F_0\cap B^4$ for some 4-ball $B^4$ in $X$. Also in this case $\beta\subseteq F_-$ is non separating, $F_-$ is obtained from the Seifert surface of a knot $K\subseteq\partial B^4=S^3$ and $\pi_1(B^4\bslash F_-)\cong\Z$. By the proof of Theorem \ref{thm: External Twist Rim Surgery}, the surfaces $F_0\cap B^4, F_i\cap B^4\subseteq B^4$ are internally $\beta^+$-stably isotopic relative to $\partial X$, where $\beta^+$ is a parallel copy of $\beta$ in $F_-$. It follows that  $F_0, F_i\subseteq X$ are internally $\beta^+$-stably isotopic, and so $B$-stably isotopic for any $B\in \overline{\Ss}(F_0\subseteq X,\beta)$ by Theorem \ref{mainthm: int to ext} and Proposition \ref{lem: Sdi2 in fuzione di S1}.  Moreover, the stabilization set is not empty since $\overline{\Ss}(F_-\subseteq B^4,\beta)\subseteq\overline{\Ss}(F_0\subseteq X,\beta)$ and by Proposition \ref{prop: Propriety of S}.

\subsection{\label{sec: AnnulusRim}Annulus rim-Surgery}
Annulus rim-surgery \cite{finashin2002knotting} is an operation, in which rim-surgery is performed at once along two loops connected by an annulus in the exterior of the surface.
\begin{constr}[Annulus rim-surgery \cite{finashin2002knotting}]\label{const: annulus rim surgery}
Consider a surface $F\subseteq X$ with a \emph{membrane} $M\subseteq X$, that is an embedded annulus $S^1\times D^1$ intersecting $F$ only with its boundary $\partial M= M\cap F$.
Assume that there exists a neighborhood $U\subseteq X$ of $M$ and a diffeomorphism of triples
\[\phi: (U,M,F\cap U)\to (S^1\times D^3, S^1\times ([-1/2,1/2]\times\{0\}^2), S^1\times f),\] where $f\subseteq D^3$ are the "vertical" unknotted arcs bounded by  the band \[b:=([-1/2,1/2]\times D^1\times\{0\})\cap D^3.\]
If $F$ is an oriented surface, then we require that the two  components of $S^1\times f$ have an orientation induced by an orientation of $S^1\times b$ matching with the orientation of $F$ (through $\phi$).
Consider a knot $K\subseteq S^3$ and define the surface
\begin{equation}
    F_{\phi,K}:=(F\bslash U)\cup (\phi^{-1}(S^1\times f_k))\subseteq X,
\end{equation}
where $f_K$ is a pair of arcs bounded by the band $b_K$ which is obtained from the band $b$ by knotting the knot $K$, see \cite[Figure 1]{finashin2002knotting} or \cite[Figure 6.a]{InternalStabilizationBaykurSunukian}.
\end{constr}
Annulus rim-surgery can produce smoothly inequivalent surfaces in $\C P^2$ \cite[Theorem 1.1]{finashin2002knotting} with the same homology class and fundamental group of the complement \cite[Propositions 1.2 and 1.3]{finashin2002knotting}.
\begin{thm}\label{thm: external annulus Surgery}  Let $F,F_{\phi,K}, M\subseteq X$ be defined as in Construction \ref{const: annulus rim surgery} and let us further assume that $F$ is orientable.
Let $\gamma\subseteq F$ be a curve which is parallel to one of the two component of $\partial M$.\begin{itemize}
    \item If  $\pi_1(X\bslash F)$ and $\pi_1(X\bslash F_{\phi,K})$ are cyclic generated by a meridian, and $F\bslash  \gamma$ is connected, then the surfaces $F,F_{\phi,K}\subset X$ are internally $\gamma$-stably isotopic rel. boundary.
    \item If $\pi_1(X\bslash (F\cup M))$ is abelian, $F\bslash \partial M$ is connected, and $X$ is compact and simply connected, then  $F,F_{\phi,K}\subset X$ are $B$-stably isotopic rel. boundary for any $B\in\overline{\Ss}(F,F_{\phi,K},\gamma)\subset\{\Stab,\StabTwist\}$ and $\overline{\Ss}(F,F_{\phi,K},\gamma)$ is not empty.
\end{itemize}
\end{thm}
\begin{proof}
    For first point of the theorem it is sufficient to adapt the proof in \cite[Section 3.3]{InternalStabilizationBaykurSunukian} using Proposition \ref{Prop repositioning arc}.\\
    If $\pi_1(X\bslash (F\cup M))$ is abelian, $F\bslash \partial M$ is connected, and $X$ compact and simply connected, then the set of hypothesis of the first point hold by \cite[Proposition 1.2]{finashin2002knotting}, and so  $F,F_{\phi,K}\subset X$ are internally $\gamma$-stably isotopic. The surfaces $F,F_{\phi,K}\subset X$ are $B$-stably isotopic for any $B\in\overline{\Ss}(F,F_{\phi,K},\gamma)\subseteq\{\Stab,\StabTwist\}$ by Theorem \ref{mainthm: int to ext}. There exists a nullhomotopic push off $\overline{\gamma}$ of $\gamma$ in $X\bslash(F\cup M)$ since $\pi_1(X\bslash(F\cup M))$ is cyclic generated by the meridian of $F$ by \cite[Lemma 2.1]{finashin2002knotting}. It follows that $\overline{\gamma}\subseteq X\bslash(F\cup F_{\phi,K})$ is nullhomotopic, since $F\cup F_{\phi_K}$ sits in a small neighborhood of $F\cup M$ which deformation retracts to $F\cup M$. Therefore, $\overline{\Ss}(F,F_{\phi,K},\gamma)$ is not empty by Theorem \ref{mainthm: int to ext}.
\end{proof}

\subsection{\label{sec: nullhom 2tori}Null-homologous 2-tori}
We start this section recalling the construction introduced in \cite{KnotSurgery} and explored in \cite{Hoffman_Sunukjian_2020}.
We would like to study the couple of surfaces $U,T_{F,K}\subseteq X$, which is composed by an unknotted 2-torus $U$, i.e. $U$ bounds a $S^1\times D^2$ in $X$, and a 2-torus $T_{F,K}\subseteq X$ that is defined as follows.
\begin{constr}[Null-homologus 2-tori \cite{Hoffman_Sunukjian_2020}]\label{const: nullhomologus tori}
Let $X$ be a 4-manifold and let $T\subseteq X$ be an embedded 2-torus with trivial normal bundle in $X$. Fix an identification of the tubular neighborhood of $T$
\[F:{\nu T}\to S^1\times (S^1\times D^2).\]
Fix an oriented knot $K\subseteq S^3$ with unknotting number equal to $1$. Choose a diagram $\mathcal{D}$ for $K$ and select a crossing $c$ such that a crossing change at $c$ gives a diagram for the unknot. \\ 
Performing $\pm 1$ Dehn surgery along a loop $l=l(\mathcal{D}, c)\subseteq S^3$ has the effect of changing the crossing $c$ of the diagram $\mathcal{D}$ of $K$. The knot $K_0$, obtained by this crossing change, is an unknot and in its complement there is the loop $l$. By a (canonical) identification of $S^3\bslash\nu K_0$ with  $S^1\times D^2$, the loop $l$ defines a knot $\widetilde{l}=\widetilde{l}(\mathcal{D},c)\subseteq S^1\times D^2$.
\[(S^3\bslash\nu K_0,l)\cong (S^1\times D^2,\widetilde{l}).\]
Finally, consider the 2-torus $S^1\times \widetilde{l}\subseteq S^1\times (S^1\times D^2)$ and define
\begin{equation}
    T_{F,K}=T_{F, K}(\mathcal{D},c):=F^{-1}\left(S^1\times \widetilde{l}\right)\subseteq \nu T\subseteq X.
\end{equation}
\end{constr}

The 2-torus $T_{F,K}$ is always nullhomologus and $\pi_1(X\bslash T_{F,K})$ has infinite cyclic fundamental group of the complement under the additional assumption of trivial $\pi_1(X\bslash T)$, see \cite[Section 2]{Hoffman_Sunukjian_2020}.

Let
$\gamma=F^{-1}(S^1\times\{1\}\times\{0\})$.
The framing of $\gamma$ induced by $F$ is the map
\begin{equation}\label{eq: definition framing induced by the restriction}
    f:=(id_{S^1}\times \alpha)\circ F|_{\nu\gamma}:\nu\gamma\to \left(S^1\times \nu\{(1,0)\}\right)\cong S^1\times D^3,
\end{equation}
where $\alpha: \nu\{(1,0)\}\to D^3$ is an identification of the tubular neighborhood of the point $(1,0)$ with $D^3$. The curve $\gamma$ equipped with the framing $f$ makes the 2-torus $T$ be $\nu\gamma$-standard.

\begin{thm}\label{thm: ExtStab nullhomologus 2-tori}
Let $T_{F,K}\subseteq X$ be 2-torus defined as in Construction \ref{const: nullhomologus tori}.
Let $\gamma\subseteq T_{F,K}$ be the simple loop with  the framing $f$ induced by the framing $F$ of $T$ (\ref{eq: definition framing induced by the restriction}).
For all the manifolds in the extended stabilization set $B\in \overline{\Ss}(T_{F,K},  \gamma)\subseteq\{\Stab,\StabTwist\}$ the 2-torus $T_{F,K}$ is smoothly unknotted in $X\# B$,  where the connected sum is performed away from $T_{F,K}$.\\
If $X$ is a spin manifold, then ${\Ss}(T, (\gamma,f))= \overline{\Ss}(T_{F,K}, \gamma)$ and it contains exactly one element.\\
If $X$ is  non spin manifold, then $\overline{\Ss}(T_{F,K}, \gamma)=\{\Stab,\StabTwist\}$.
\end{thm}
\begin{proof} We can assume $T_{F,K}$ and $U$ to be $\nu\gamma$-standard, where $U$ is a smoothly unknotted 2-torus in $X$.
It follows by \cite[Section 3.6]{InternalStabilizationBaykurSunukian} that $T_{F,K},U\subseteq X$ becomes smoothly isotopic after a single trivial stabilization away from $\nu\gamma$ and  that it can be taken  to be relative to $\nu_{1/2}\gamma$. By Proposition \ref{Prop repositioning arc}, the surfaces $T_{F,K},U\subseteq X$ are internally $\gamma$-stably isotopic. By Theorem \ref{mainthm: int to ext} we have that for any $B\in\overline{\Ss}(T_{F,K}, U, \gamma)$ the 2-tori $T_{F,K}, U\subseteq X$ are $B$-stably isotopic, in particular $T_{F,K}\subseteq X\# B$ is smoothly unknotted. 
     Since $U\subseteq X$ can be chosen to be in a small tubular neighborhood of $T_{F,K}\subseteq X$, then by Lemma \ref{lem: Sdi2 in fuzione di S1} we have that $\overline{\Ss}(T_{F,K}, U, \gamma)=\overline{\Ss}(T_{F,K},  \gamma)$.

Now we will prove the content of equation (\ref{eq: StabSet nullhologus 2-tori}).
\begin{equation}\label{eq: StabSet nullhologus 2-tori}
    \emptyset \not={\Ss}(T, (\gamma,f))\subseteq {\Ss}(T_{F,K}, (\gamma,f))=\overline{\Ss}(T_{F,K}, \gamma).
\end{equation}
    Since the 2-torus $T_{F,K}$ has complement with infinite cyclic fundamental group (see \cite[Section 2]{Hoffman_Sunukjian_2020}) we can apply Proposition \ref{prop: Propriety of S}.\ref{prop: SS item 1}, so  ${\Ss}(T_{F,K}, (\gamma,f))=\overline{\Ss}(T_{F,K},  \gamma)$ with framing $f$ of $\gamma$ given by the restriction of $F$ to $\nu\gamma\subseteq\nu T$. 
    Clearly ${\Ss}(T,T_{F,K}, \gamma)\subseteq{\Ss}(T_{F,K}, \gamma)$ 
    and by Lemma \ref{lem: Sdi2 in fuzione di S1} we have that ${\Ss}(T, \gamma)={\Ss}(T,T_{F,K}, \gamma)$. Lemma \ref{lem: proprieties Ss}, and the hypothesis on  the complement of $T$ tell us that $ \emptyset \not={\Ss}(T, \gamma)$.

    If $X$ is a spin manifold, then by Lemma \ref{lem: proprieties Ss} we have that ${\Ss}(T_{F,K}, \gamma)$ can't have two elements. The relation expressed in equation (\ref{eq: StabSet nullhologus 2-tori}) implies that ${\Ss}(T_{F,K}, \gamma)$ has one element and it coincides with ${\Ss}(T, \gamma)$.
    
  We can assume that the inequality
   \begin{equation}\label{eq: sunukjian inequality}
        b_2(X)\geq |\sigma(X)|+6
    \end{equation}
    holds by replacing $X$ with $X\# 3(\Stab)$, here $\sigma(X)$ denotes the signature of $X$. Notice that under this replacement the stabilization set does not change.
         Now, we can apply \cite[Theorem 7.2]{SunukjianNathan} to conclude that the 2-tori $T_{F,K}$ and $U$ are topologically isotopic (see also \cite[Theorem 1.1]{Hoffman_Sunukjian_2020}).  If $X$ is not spin, there exists an immersed 2-sphere with odd self intersection disjoint from $U$ since $X\cong X\# S^4$ and $U$ can be moved in the $S^4$ factor. By the topological ambient isotopy between $T_{F,K}$ and $U$,  we can conclude that there exists a continuous map from the 2-sphere to $X\bslash T_{F,K}$ representing an odd class of $H_2(X;\Z)$, but this map can be approximated by a smooth immersion. Now by Lemma \ref{lem: proprieties Ss} we have that ${\Ss}(T_{F,K}, \gamma)=\{\Stab,\StabTwist\}$.
 \end{proof}

\begin{oss}\label{oss: torus framings and stabilization sets}

If $X$ is a spin manifold, then the only element of stabilization set $\overline{\Ss}(T_{F,K},\gamma)=\Ss(T,\gamma)$ depends on the chosen parametrization of $T$ as $S^1\times S^1$ since this determinate the curve $\gamma$ (see Proposition \ref{prop: Propriety of S}.\ref{prop: SS item 5}). But even fixed a parametrization for $T$  the set $\Ss(T,\gamma)$ still depends also on the chosen framing  $F$ for the 2-torus $T$.
Indeed, given any framing $F_0:\nu T\to S^1\times S^1\times D^2$ for a 2-torus $T$ with trivial normal bundle inside a 4-manifold, it can be composed with the diffeomorphism
    \[t:S^1\times S^1\times D^2\to S^1\times S^1\times D^2 \] defined as
    $t(\alpha,\beta,z)=(\alpha,\beta,\alpha\cdot z)$. So, define $F_n:=t^n\circ F_0$ for any $n\in \Z$, set the curve $\gamma_n:=F_n^{-1}(S^1\times \{p\}\times\{0\})$ with the framing $f_n$ given by the restriction of $F_n$. We can notice that $f_n=\tau^n\circ f_0$ as defined in Construction \ref{constr: framings}. Therefore, since $X$ is spin and $T$ has simply connected complement, we can apply Proposition \ref{lem: geometricDualSphere} to deduce that $\Ss(T,(\gamma,f_0)$ is different from $\Ss(T,(\gamma,f_1))$.
\end{oss}

\section{External stabilization of surfaces obtained by mixing and satellite operations \label{sec: externalPrescribedGroup}}
In this section we will discuss some stabilization behavior for the surfaces produced in \cite{torres2023topologically}. The main results of this section are  Theorem \ref{thm: Stable isotopy for TORRES Z_2} and Corollary \ref{cor: torres patterns}.
First we will review very schematically Torres's construction.
Assume that is given a pair of smooth, closed, simply connected 4-manifold $X$ and $X'$ which are homeomorphic. Assume also that there exists a diffeomorphism
\begin{equation}
    \phi: X'\#(\Stab)\to X\# (\Stab)
\end{equation} such that
\begin{equation}\label{eq: condition basic}
    \phi_*([\Sigma])=[\Sigma]\in H_2(X\#(\Stab);\Z),
\end{equation} where $\Sigma$ is a closed surface in $(\Stab)\bslash D^4$.
This defines a surface $\Sigma':=\phi(\Sigma)\subseteq X\#(\Stab)$ in the same homology class of $\Sigma\subseteq(\Stab)\bslash D^4\subseteq X\#(\Stab)$. 
The surface $\Sigma'$ is defined by the following datum: the manifolds $X$ and $X'$, the surface $\Sigma\subseteq \Stab$ and the diffeomorphism $\phi$.
\begin{equation}\label{def: S'}
    \Sigma'=\Sigma'(X,X',\Sigma,\phi)\subseteq X\# (\Stab)
\end{equation}
It is easy to see that the surfaces $\Sigma,\Sigma'\subseteq X\# (\Stab)$ are topologically equivalent, i.e. there exists a self homeomorphism  of $X\# (\Stab)$ sending $\Sigma$ to $\Sigma'$. Indeed, given any homeomorphism $h:X'\to X$ one can define a homeomorphism $(h^{-1}\# id_{\Stab})\circ\phi: X\# (\Stab)\to X\# (\Stab)$.\\
Furthermore, if the map induced in homology
\begin{equation}
    \phi_{*}:H_2( X';\Z)\oplus H_2(\Stab;\Z)\to H_2(X;\Z)\oplus H_2(\Stab;\Z)
\end{equation}
can be written as 
\begin{equation}\label{eq: condition for Quinn}
    \phi_*=h_*\oplus Id \text{ for some homeomorphism }h:X'\to X,
\end{equation}
then $\Sigma,\Sigma'\subseteq X\# (\Stab)$ are topologically ambient isotopic   by the work of \cite[Theorem 1.1]{Quinn_1986} (see also \cite{gabai2023pseudo}). 
\begin{definition}
\label{def: mixing}
    We say that a surface $\Sigma'\subseteq X\#(\Stab)$ is \emph{a mixing of} $\Sigma\subseteq(\Stab)\bslash D^4\subseteq X\#(\Stab)$ if it is defined as in equation (\ref{def: S'}).
     The surface $\Sigma'$ is \emph{a good mixing of} $\Sigma$ if it is a mixing of $\Sigma$ and it satisfies the condition in equation (\ref{eq: condition for Quinn}). 
\end{definition}

 Now fix a group $G$ that can be finite cyclic or the binary-icosahedral group. In \cite{torres2023topologically} it is proved that there exists a 2-sphere $S_Q\subseteq \Stab$ of zero self intersection such that surgery along this sphere gives a 4-manifold $Q$ that is a $\Q$-homology 4-sphere with fundamental group isomorphic to $G$, see also \cite{sato1991locally}. We can insert the 2-sphere $S_Q$ in the previous construction a define
 \begin{equation}\label{def: S'_G}
     S'_Q:=\Sigma'(X,X',S_Q,\phi)\subseteq X\#(\Stab).
 \end{equation}
 Under opportune assumption on the Seiberg-Witten invariants of $X$ and $X'$, it is shown  that the 4-manifolds $X\# Q$ and $X'\# Q$ are not diffeomorphic, it follows that the two 2-spheres $S_Q,S'_Q\subseteq X\#(\Stab)$ are not smoothly equivalent. Moreover, using a infinite family of pairwise non diffeomorphic manifolds $\{X'_n\}_{n\in\N}$ one can produce an infinite collection of exotically embedded 2-spheres with fundamental group of the complement isomorphic to $G$.

\subsection{Smooth stable equivalence}

One can prove the following proposition about smooth stable \emph{equivalence}, while we will get \emph{isotopy} only with some restrictions in the next section.
\begin{prop} \label{prop: stable equivalence}
Let $\Sigma\subseteq (\Stab)\bslash D^4\subseteq X\#(\Stab)$ be a surface, where $X$ is a closed and simply connected 4-manifold. Assume that $\Sigma'\subseteq X\#(\Stab)$ is a mixing of $\Sigma$, see Definition \ref{def: mixing}. The following holds:
\begin{itemize}
    \item The surfaces $\Sigma,\Sigma'\subseteq X\#(\Stab)$ are topologically equivalent.
    \item If $\Sigma'$ is a good mixing of $\Sigma$, then $\Sigma,\Sigma'\subseteq X$ are topologically ambient isotopic.
    \item If $X$ has indefinite intersection form (or $b_2(X)\leq 8$), then $\Sigma'$ is smoothly equivalent to a good mixing of $\Sigma$ in $X\#(\Stab)$.
        \item The surfaces $\Sigma,\Sigma'\subseteq X\#(\Stab)$ are smoothly $(\Stab)$-stably equivalent.
\end{itemize}
    
\end{prop}
\begin{proof}
Let $\Sigma'=\Sigma'(X,X',\Sigma,\phi)$ be defined as in equation (\ref{def: S'}).
The first two point follows from the discussion contained in the paragraph after Definition \ref{def: mixing}.\\
If the manifold $X$ has indefinite intersection form (or $b_2(X)\leq 8$), then by \cite{WallSelfDiffeo} the diffeomorphism $\phi'=\alpha\circ\phi$  satisfy equation (\ref{eq: condition for Quinn}), where $\alpha$ is a smooth automorphism of $X\#(\Stab)$. So $\alpha(\Sigma')$ is a good mixing of $\Sigma$.\\
To prove the last point, notice that $\Sigma'\subseteq X\#(\Stab)$ is smoothly equivalent to $\Sigma\subseteq X'\#(\Stab)$. Moreover, $\Sigma\subseteq (X'\#(\Stab))\#(\Stab)$ is smoothly equivalent to $\Sigma\subseteq (X\#(\Stab))\#(\Stab)$, where the diffeomorphism is given by exchanging the old and the new $(\Stab)$ factors and then applying the diffeomorphism $\phi\# Id_{\Stab}$.
\end{proof}
\begin{oss}
    The proof of the last point of Proposition \ref{prop: stable equivalence} in general gives a map that is not isotopic to the identity. For example the induced map in the second homology group can easily be different from the identity map.
\end{oss}


\subsection{Satellite operations on 2-knots in 4-manifolds and smooth stable isotopy\label{sec: satelliteoperations}}
In this section we will set our attention on surfaces embedded in the tubular neighborhood of embedded 2-spheres. 
In classical knot theory, dimension 3, the satellite operation is well known. It  ties a new knot $K(P,\phi)$ in a neighborhood of a given knot $K$ following a pattern $P$. This pattern $P$ is a knot or a link in the neighborhood of the unknot, and so the operation is defined via an identification $\phi$ of the tubular neighborhood of the unknot and the one of the knot $K$. The same ideas can be applied in dimension 4 using 2-sphere and surfaces. The main two  differences to have in mind are the following:
\begin{itemize}
    \item There are infinite isomorphism class of orientable rank 2 vector bundle over the 2-sphere and they are determined by the Euler number. We can not restrict ourselves to the trivial bundle over $S^2$, which is the normal bundle of unknotted 2-sphere.
    \item Any automorphism of an orientable rank 2 vector bundle over the 2-sphere is isotopic to the identity. Therefore, if two bundles are isomorphic, then there is a unique way to identify them up to isotopy!
\end{itemize}
\begin{prop}\label{prop: rank2 vecotr bundle over the sphere}
    If $V$ is a rank 2 real and oriented vector bundle over $S^2$, then any orientation preserving automorphism of $V$ is isotopic to the identity. 
\end{prop}
\begin{proof}
Let $S^2$ be the union of two 2-disks $D_1, D_2\subseteq S^2$. The bundle $V$ restricted over $D_1$ is isomorphic $D_1\times \R^2$ and an automorphism is determined by a map $\varphi_1: D_1\to GL^+(2;\R)\simeq S^1$. So, any automorphism $\phi: V\to V$ can be assumed to be isotopic the identity when restricted to $D_1$. The restriction of $\phi$ to the bundle over $D_2$ is determined by a map $\varphi_2:D^2\to GL^+(2:\R)$ which restricts to the identity on $\partial D_2$. This map is  homotopic to the identity map relative to $\partial D_2$, and so $\phi$ is isotopic to the identity. 
\end{proof}
\begin{definition}
    Consider $S$ to be a embedded 2-sphere in an oriented 4-manifold $M$. Assume that $N$ is an oriented rank 2-vector bundle isomorphic to the normal bundle $\nu S$ of $S\subseteq M$, and call $\phi:N\to \nu S$ an identification. Given any surface $P\subseteq N$ we call the surface $\phi(P)\subseteq M$ the satellite of $S$ with pattern $P$.
The surface $\phi(P)$ does not depend on the map $\phi$ up to an isotopy supported in $\nu S$ and it comes with a canonical embedding map given by $\phi\circ j:P\to \nu S$, where $j:P\hookrightarrow N$ is the inclusion map.
\end{definition}

\begin{oss}\label{rem: mixing satellite}
    A surface $\Sigma\subseteq S^2\times D^2\subseteq (\Stab)\bslash D^4$  can be considered as a \emph{a satellite} of the 2-knot $S:=S^2\times\{0\}\subseteq S^2\times D^2\subseteq X\# (\Stab)$. Moreover, the surface $\Sigma'=\Sigma'(X,X',\Sigma,\phi)\subseteq X\# (\Stab)$ can be considered to be the satellite with \emph{pattern} $\Sigma$ of the 2-knot $S':=\phi(S)$. Notice that the identification of the two tubular neighborhood is given by the restriction of the map $\phi$ to $\nu S'$.
    \[\phi|_{\nu S}:\nu S\to \phi(\nu S)=\nu S'\]
\end{oss}

This motivates us to study satellite operations in dimension 4 and their behaviors with isotopies.

\begin{prop}\label{prop: Patten isotopy} Let $S_1,S_2\subseteq M$ be  oriented 2-spheres in the interior of an oriented 4-manifold $M$. 
    Assume that we can identify the tubular neighborhoods via a diffeomorphism $\phi:\nu S_1\to\nu S_2$ and consider the inclusion maps \[j_i:\nu S_i\hookrightarrow M\quad i=1,2.\]
        If the 2-spheres $S_1,S_2\subseteq M$ are smoothly isotopic, then the maps $j_1$ and $j_2\circ \phi:\nu S_1\hookrightarrow M$ are smoothly isotopic. In particular, for any pattern $P\subseteq\nu S_1$ the surfaces $P,\phi(P)\subseteq M$ are smoothly isotopic.
\end{prop}

\begin{proof}
     If $\iota:S_1\hookrightarrow\nu S_1$  the standard inclusion map, then the embeddings $ j_1\circ\iota$ and $j_2\circ \phi\circ \iota$ have images that are ambient isotopic. It follows that these maps are isotopic since $S_1$ is a 2-sphere and every orientation preserving self diffeomorphism is isotopic the identity map.  By the uniqueness of the tubular neighborhood,  see for example \cite[Theorem 5.3]{Hirsch_1997}, the maps $j_1$ and $ j_2\circ \phi\circ f$ are isotopic, where $f:\nu S_1\to \nu S_1$ is an isomorphism of vector bundles.
     Since any isomorphism of an oriented rank 2 vector bundle over the 2-sphere is isotopic the  identity, see Proposition \ref{prop: rank2 vecotr bundle over the sphere}, we can conclude the proof.
\end{proof}

Thanks to  \cite{Auchly1stb}, we are ready to prove that any pair of surfaces obtained as satellite of 2-spheres with simply connected complement  are externally stably isotopic.

\begin{thm}\label{thm: patterns}
Let $M$ be a closed, oriented and simply connected 4-manifold and $\alpha\in H_2(M;\Z)$ a homology class.
 Let $S_1$ and $S_2$ be embedded 2-spheres in $X$ with simply connected complement and homology class $\alpha$.
    Fix a pattern $P\subseteq\nu S_1$ and an isomorphism of tubular neighborhoods $\phi:\nu S_1\to\nu S_2$.
    \begin{enumerate}
   \item If $\alpha$ is an ordinary class, then the surfaces $P,\phi(P)\subseteq M$ are $(\Stab)$-stably isotopic.
   \item If $\alpha$ is a characteristic class, then the surfaces $P,\phi(P)\subseteq M$ are $(\StabTwist)$-stably isotopic.
    \end{enumerate}

    Moreover, the result hold also for the embedding maps of the surfaces.
    \end{thm}
\begin{proof}
    It is sufficient to apply \cite{Auchly1stb} and Proposition \ref{prop: Patten isotopy}. 
\end{proof}
\begin{cor}\label{cor: Isotopy pattern S2xS2}
    Let $X$ to be a closed, oriented and simply connected 4-manifold, the 2-sphere $S:=S^2\times\{p\}\subseteq(\Stab)\bslash D^4\subseteq X\#(\Stab)$ and $S'\subseteq X\#(\Stab)$ be a mixing of $S$. For any pattern $P\subseteq \nu S\subseteq X\#(\Stab)$ and isomorphism of tubular neighborhood $\phi: \nu S\to \nu S'$,
the surfaces $P,\phi (P)\subseteq X\#(\Stab)$ are $(\Stab)$-stably isotopic.  Moreover, the same results hold for the embedding maps.
     If $S'$ is a good mixing of $S$, then $P,\phi (P)\subseteq X\#(\Stab)$ are topological ambient isotopic.
\end{cor}
\begin{proof}
    It is sufficient to apply Theorem \ref{thm: patterns} choosing $M:=X\#(\Stab)$, $S_1:=S^2\times\{0\}$, $S_2:=\phi(S_1)$ and $P\subseteq S^2\times D^2=\nu S_1$. By assumption $[S_1]=[S_2]$ and it is an ordinary class.\\
    If $S'$ is a good mixing of $S$, then $\phi(P)$ is a good mixing of $P$, therefore they are topologically ambient isotopic by Proposition \ref{prop: stable equivalence}.
\end{proof}

\subsection{The 2-sphere $S_{\Z/k}$ is the satellite of a 2-sphere with simply connected complement}
The aim of this section is to show that some of the 2-spheres $S_Q\subseteq\Stab$ defined by \cite{torres2023topologically} are contained in $S^2\times D^2$, which is a tubular neighborhood of $S:= S^2\times\{0\}$. After that, the main results of this section, Theorem \ref{thm: Stable isotopy for TORRES Z_2} and Corollary \ref{cor: torres patterns}, will follow from the previous section. We will focus on the case in which $G$ is finite cyclic group $\Z/k\Z$ and we will denote the corresponding 2-sphere $S_{\Z/k}$ instead of $S_Q$.
\\
Fix $k\in\N$ and let $Q=\Sigma_{k,0}$ denote the $\Q$-homology 4-sphere defined as in \cite[end of Section 2.5]{torres2023topologically}. Figure \ref{fig: SZ2 before surgery} is a Kirby diagram for $\Sigma_{2,0}$. For other values of $k$ the Kirby diagram of $\Sigma_{k,0}$ is obtained by replacing the 2-handle which attaching sphere loops twice through the dotted 1-handle with one that loops $k$ times. Doing loop surgery on $\Sigma_{k,0}$ along the meridian of the 1-handle gives a 4-manifold $\Sigma_{k,0}^*$ that is diffeomorphic to $\Stab$. The 2-sphere $S_{\Z/k}$ is defined to be the image through  a diffeomorphism $f:\Sigma_{k,0}^*\to\Stab$ of the belt 2-sphere $B\subseteq \Sigma_{k,0}^*$ of the surgery, see \cite[Section 2.6]{torres2023topologically}.   
\begin{equation}\label{eq: S_G}
    S_{\Z/k}:=S_{\Sigma_{k,0}}=f(B)\subseteq\Stab.
\end{equation}

\begin{definition}\label{def: P(k)}
Define the embedded surface $P(k)\subseteq S^2\times D^2$ for $k\in\N$ to be obtained by tubing $k$ parallel and cooriented copies of $S^2$ of the form $S^2\times\{p_1,\dots, p_k\}\subseteq S^2\times D^2$ along $k-1$ arcs $\{q_1\}\times\alpha_1,\dots,\{q_{k-1}\}\times\alpha_{k-1}$. Here the points $p_i\in D^2$ are distinct (and in the interior of $D^2$), the arcs $\alpha_i$ have disjoint interior and the endpoints of $\alpha_i$ are $p_i$ and $p_{i+1}$. In particular, we require that the union $\alpha_1\cup\dots\cup\alpha_{k-1}$ is an arc, see Figure \ref{fig: pK}. The tubing is performed respecting the orientation of each 2-sphere, so that $P(k)$ is an oriented 2-sphere.
\end{definition}
\begin{figure}[h]
    \centering
    \includegraphics[width=0.4\linewidth]{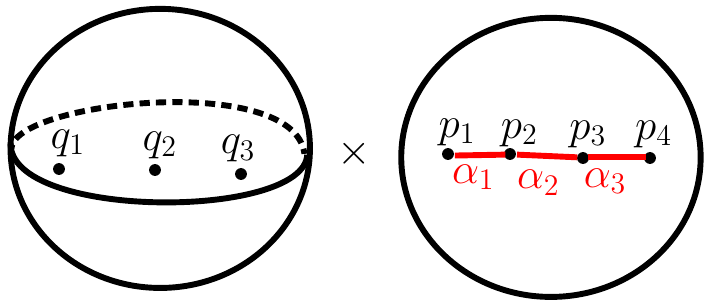}
    \caption{}
    \label{fig: pK}
\end{figure}
From this definition one can check that the complement of $P(k)$ in $S^2\times D^2$ is infinite cyclic generated by one element and the loop $\{p\}\times \partial D^2\subseteq S^2\times D^2$ is homotopic to $k$ times the generator. In particular, the 2-sphere $P(k)\subseteq S^2\times D^2\subseteq S^2\times S^2$ has complement with finite cyclic fundamental group of order $k$.

We will now see that the 2-spheres $S_{\Z/k},P(k)\subseteq S^2\times S^2$ are smoothly equivalent.

\begin{lem}\label{lem. S_G is P(k)}
For any $k\geq 1$ the embedded 2-spheres $S_{\Z/k}\subseteq \Stab$ (\ref{eq: S_G}) and $P(k)\subseteq S^2\times D^2\subseteq\Stab$ (Definition \ref{def: P(k)}) are smoothly equivalent.
Equivalently, the embedded 2-sphere $S_{\Z/k}\subseteq \Stab$ is smoothly equivalent to a satellite 2-knot of $S=S^2\times\{p\}\subseteq\Stab$ with pattern  $P(k)$.
Equivalently, the $\Q$-homology 4-sphere $\Sigma_{k,0}$ is diffeomorphic to the result of sphere surgery along the 2-sphere $P(k)\subseteq\Stab$.
\end{lem}
\begin{proof}
By definition $S_{\Z/k} $ and the belt 2-sphere $B$ of the loop surgery defining $\Sigma_{k,0}^*$ are smoothly equivalent.
The idea of the proof is to follow the surface $B\subseteq\Sigma_{k,0}^*$ in a sequence of Kirby diagrams related by isotopies and slides starting from $\Sigma_{k,0}^*$ and arriving to the standard Kirby diagram of $S^2\times S^2$ (in \cite[Section 2]{Auckly_Kim_Melvin_Ruberman_2015} it is explained how to slide surfaces in Kirby diagrams).
    This sequence is depicted and commented in Figure \ref{fig: S_Z_2} for the case $k=2$. A similar sequence of steps can be obtained for any $k\geq 2$.
\end{proof}

\begin{figure}
     \centering
     \begin{subfigure}[b]{0.3\textwidth}
         \centering
         \includegraphics[width=0.8\textwidth]{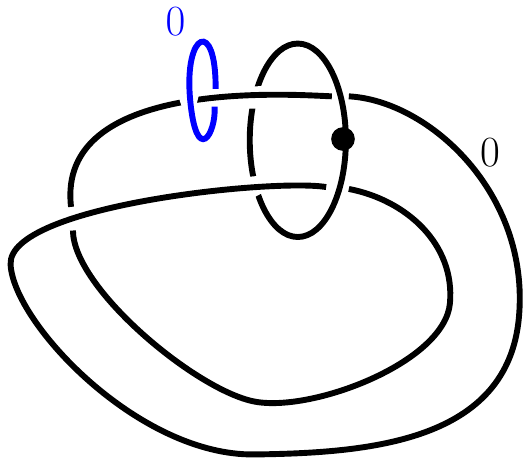}
         \caption{\label{fig: SZ2 before surgery}}
     \end{subfigure}
     \hfill
     \begin{subfigure}[b]{0.3\textwidth}
         \centering
         \includegraphics[width=0.8\textwidth]{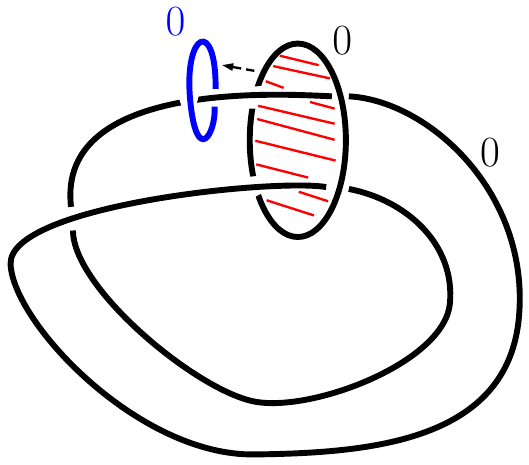}
         \caption{\label{fig: SZ2 after sugery}}
     \end{subfigure}
     \hfill
\begin{subfigure}[b]{0.3\textwidth}
         \centering
         \includegraphics[width=0.8\textwidth]{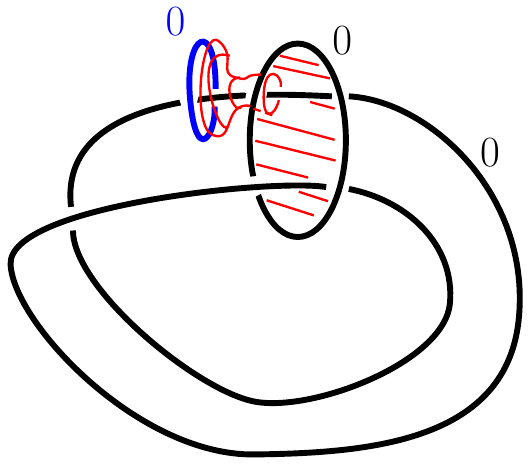}
         \caption{\label{fig: SZ2 slide}}
     \end{subfigure}
     \hfill
     \begin{subfigure}[b]{0.3\textwidth}
         \centering
         \includegraphics[width=0.8\textwidth]{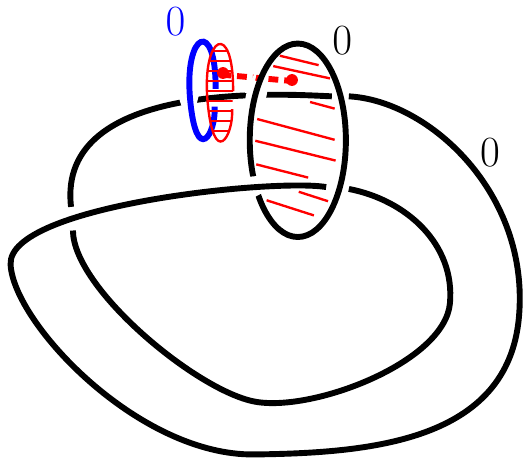}
         \caption{\label{fig: SZ2 d}}
     \end{subfigure}
     \hfill
     \begin{subfigure}[b]{0.3\textwidth}
         \centering
         \includegraphics[width=0.8\textwidth]{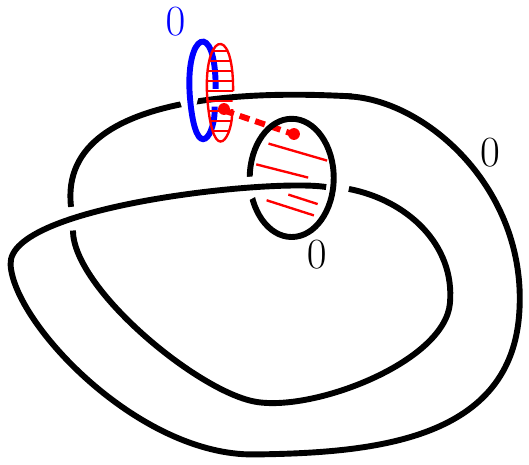}
         \caption{\label{fig: SZ2 e}}
     \end{subfigure}
     \hfill
     \begin{subfigure}[b]{0.3\textwidth}
         \centering
         \includegraphics[width=0.55\textwidth]{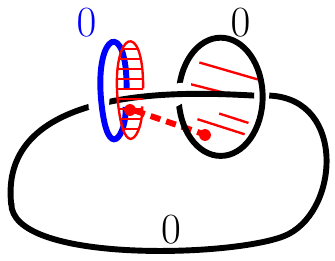}
         \caption{\label{fig: SZ2 f}}
     \end{subfigure}
     \hfill\\
     \begin{subfigure}[b]{0.3\textwidth}
         \centering
         \includegraphics[width=0.6\textwidth]{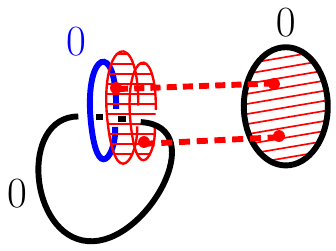}
         \caption{\label{fig: SZ2 g}}
     \end{subfigure}
     \hfill
     \begin{subfigure}[b]{0.3\textwidth}
         \centering
         \includegraphics[width=0.6\textwidth]{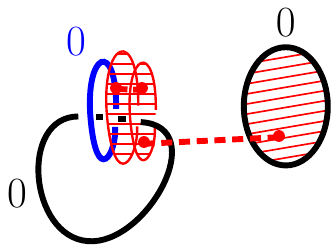}
         \caption{\label{fig: SZ2 h}}
     \end{subfigure}
     \hfill
     \begin{subfigure}[b]{0.3\textwidth}
         \centering
         \includegraphics[width=0.3\textwidth]{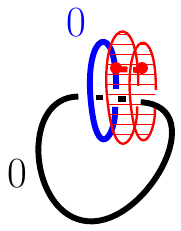}
         \caption{\label{fig: SZ2 last}}
     \end{subfigure}
     \hfill
     \caption{\label{fig: S_Z_2} The Figures \ref{fig: SZ2 before surgery}-\ref{fig: SZ2 last} are Kirby diagrams, they are completed by a 4-handle and, except for the last one, a 3-handle. Figure \ref{fig: SZ2 before surgery} is $Q_{\Z_2}=\R P^2\widetilde{\times} S^2$ the orientable $S^2$ bundle over $\R P^2$ with $0$ Euler number depicted as in \cite[Figure 5.46]{Gompf_Stipsicz_1999}. Figure \ref{fig: SZ2 after sugery} is the result of loop surgery along a $0$-framed meridian of the dotted circle in Figure \ref{fig: SZ2 before surgery}. The belt 2-sphere $B$ of the surgery is the union of the red dashed region and a 2-disk on the corresponding 2-handle. 
      The following sequence of Kirby diagram had been already be considered in \cite[Lemma 1]{torres2023topologically}, our contribution is to follow the 2-sphere $B$ in this sequence.
     Figure \ref{fig: SZ2 slide} is obtained  by the handle slide indicated by an arrow in the previous picture. In Figure \ref{fig: SZ2 d} we see the surface in the previous picture as the result of an internal stabilization along the dotted red. Notice that both 2-disks indicated in the figure are completed by the 2-disks sitting in the nearby 2-handle. By isotopies, we get the next two Figures, \ref{fig: SZ2 e} and \ref{fig: SZ2 f}. With a new handle slide, and steps similar to the previous ones, we obtain Figure \ref{fig: SZ2 g}, here three surfaces are dashed and two dotted arcs which connect them. The second 2-disk from the left have interior pushed into the 4th dimension so that the arc on top, lying in $S^3$, is not intersecting that surface. We can move this last arc along the other one via an isotopy, see Figure \ref{fig: SZ2 h}. Finally, we can remove the 2-sphere on the right since is smoothly unknotted. See Figure \ref{fig: SZ2 last}. In the last figure we can recognize the surface $P(2)\subseteq\Stab$.  }
\end{figure}

\begin{oss}
    In \cite{KANENOBU1985123} is given a definition of a \emph{k-cable} of a $(d-2)$ dimensional knot in $S^{d}$. In dimension $d=4$  a $k$-cable of a 2-knot $S\subseteq S^4$ (or more generally to an embedded 2-sphere with trivial normal bundle) correspond to the satellite surface of $S$ with pattern $P(k)$. 
\end{oss}
\begin{thm}\label{thm: Stable isotopy for TORRES Z_2}
Let $X$ be a closed, simply connected 4-manifold and let $S_{\Z/k}\subseteq\Stab$ (\ref{eq: S_G}).
    Assume that $S'_{\Z/k}\subseteq X\#(\Stab)$ is a mixing of $S_{\Z/k}\subseteq (\Stab)\bslash D^4$.
    The surfaces $S_{\Z/k},S_{\Z/k}'\subseteq X\#(\Stab)$ are $(\Stab)$-stably isotopic.
\end{thm}
\begin{proof}
    By Lemma \ref{lem. S_G is P(k)} the 2-sphere $S_{\Z/k}$ is smoothly equivalent the satellite of $S=S^2\times\{p\}$ with pattern $P(k)$. There exists a diffeomorphism of pairs:
    \[\psi: (X\#(\Stab), S_{\Z/k})\to (X\#(\Stab), P(k)).\]
    It follows that $\psi(S_{\Z/k}')\subseteq X\#(\Stab)$ is a mixing of $P(k)$, so it is a satellite surface of a mixing  $S'\subseteq X\#(\Stab)$ of $S$, with pattern $P(k)$ (see Remark \ref{rem: mixing satellite}).
    The surfaces $\psi(S_{\Z/k}'),\psi(S_{\Z/k})\subseteq X\#(\Stab)$ are $(\Stab)$-stably isotopic by Corollary \ref{cor: Isotopy pattern S2xS2}, therefore, the same holds for $S_{\Z/k}', S_{\Z/k}\subseteq X\#(\Stab)$.
\end{proof}

The following corollary shows that applying a non-trivial satellite operation to each 2-sphere in an exotic collection can give rise to an exotic collection.
 
\begin{cor}[\cite{sato1991locally, torres2023topologically}]\label{cor: torres patterns}
 For any $k\geq 1$ let $P(k)\subseteq S^2\times D^2$ be the 2-sphere as in Definition \ref{def: P(k)}.
    There exists a smooth, closed, oriented and simply connected 4-manifold $Z$ and an infinite collection $\Co(1)$ of smoothly embedded and oriented 2-spheres in $Z\#(\Stab)$ with trivial normal bundle, such that for any $k\in\N$ the collection of 2-spheres
    \[\Co(k):=\{S(k)\subseteq Z\#(\Stab)| S(k)\text{ is the satellite surface of $S\in\Co(1)$ with pattern $P(k)$}\}\]
     have the following proprieties:
    \begin{enumerate}
        \item for any element $S(k)\in\Co(k)$, the surfaces $S(k)$ and $P(k)\subseteq S^2\times D^2\subseteq (\Stab)\bslash D^4\subseteq Z\#(\Stab)$ are topologically isotopic.
       In particular, the fundamental group of the complement of $S(k)$ is isomorphic to $\Z/k\Z$ and its homology class is $k[S^2\times\{p\}]\in H_2(Z\#(\Stab);\Z)$.
        \item The collection of 4-manifolds obtained by performing sphere surgery on $Z\#(\Stab)$ along the elements of $\Co(k)$ are pairwise non diffeomorphic.
        In particular, any distinct pair of elements $S(k),S'(k)\in\Co(k)$ are not smoothly equivalent, i.e. there is no diffeomorphism of pairs
        \[(Z\# (\Stab),S(k))\to (Z\# (\Stab),S'(k))\]
        Moreover, the surfaces $S(k),P(k)\subseteq Z\#(\Stab)$  are $(\Stab)$-stably isotopic.
    \end{enumerate}
    In particular, $Z$ can be taken to be a 4-manifold of the form $\C P^2\#j\overline{\C P^2}$ for $2\leq j\leq 7$.
\end{cor}

\begin{proof}
Fix $\{Z_i\}_{i\in\N}$ to be any infinite collection of 4-manifolds that satisfy the assumptions of \cite[Theorem B]{torres2023topologically} and set $Z=Z_1$. 
Each element $Z_i$ can be used to define a 2-sphere $S_i$ which is a good mixing  of $P(1)=S^2\times \{p\}\subseteq (S^2\times S^2)\bslash D^4\subseteq Z\times(\Stab)$, see Definition \ref{def: mixing}. 
\[S_i=S_i(Z,Z_i,P(1),\phi_i)\subseteq Z\#(\Stab).\]
Here we can assume that the maps $\phi_i$ satisfy equation (\ref{eq: condition for Quinn}) since $Z$ has indefinite intersection form, see the third point of Proposition \ref{prop: stable equivalence}.
This gives us the collection $\Co(1):=\{S_i\}_{i\in\N}$.
This is the same construction described in \cite[Theorem B]{torres2023topologically} choosing $M=S^4$. \\
          By Corollary \ref{cor: Isotopy pattern S2xS2} all the elements in the collection $\Co(k)$ are topologically ambient isotopic to $P(k)$  and they become smoothly isotopic after one external stabilization with $\Stab$.\\
The 2-sphere $S_i$ is smoothly equivalent to the 2-sphere $P(1)\subseteq Z_i\#(\Stab)$, therefore performing sphere surgery along $S_i$ gives a 4-manifold diffeomorphic to $Z_i$.
Therefore, the 2-sphere $S_i(k)$, obtained as satellite of $S_i$ with pattern $P(k)$, is smoothly equivalent to the 2-sphere $P(k)\subseteq Z_i\#(\Stab)$, which is the satellite of $P(1)$ with pattern $P(k)$. So sphere surgery along $S_i(k)$ gives a 4-manifold diffeomorphic to $Z_i\#\Sigma_{k,0}$ by Lemma \ref{lem. S_G is P(k)}. By \cite[Theorem B]{torres2023topologically} the manifolds $\{Z_i\#\Sigma_{k,0}\}_{i\in\N}$ are pairwise non diffeomorphic. 
\end{proof}


\section{Stabilization of brunnianly exotic 2-links\label{sec: unlinked}}
\begin{definition}\label{DoublyToroidal}
A smooth,  closed, oriented and simply connected 4-manifold $M$ is said to be \emph{doubly toroidal} if there exists a pair of disjoint  smoothly embedded 2-tori $\{T_i\subseteq X: i=1,2\}$ in $M$  of self intersection $[T_i]^2=0$ and such that complement of those 2-tori $M\bslash (T_1\sqcup T_2)$ is simply connected.
\end{definition}
 For any doubly toroidal $M$ the authors of \cite{(Un)knotted} produce an infinite family of embedded $g$-component 2-links $\{\Gamma_{K,g}:K\subseteq S^3\}$ in $M\# g (S^2\times S^2)$, parameterized by knots in $S^3$  with different Alexander polynomial. 
 Here we will consider the case in which the fundamental group of the complement of the 2-links $\Gamma_{K,g}$ is free with $g$-generators.
With two additional proprieties on $M$ they proved that those 2-links are smoothly inequivalent but pairwise topologically ambient isotopic. Moreover, for $g=1$ the collection of $\{\Gamma_{K,1}:K\subseteq S^3\}$ is collection of 2-spheres  topologically unknotted and smoothly knotted. We will denote as $S_K$ the elements of this collection.

The main objectives of these sections is to show that the surfaces $S_K$ becomes unknotted after an internal stabilization, that the 2-link $\Gamma_{K,g}\subseteq M\#g(\Stab)$ is brunnian and it  becomes smoothly unlinked after one external stabilization with $\Stab$, i.e. $\Gamma_{K,g}$ is bounding a disjoint union of $g$ smoothly embedded 3-disk in $M\#(g+1)(\Stab)$. Under some additional hypothesis, we will achieve those objectives in Theorem \ref{thm: Int unknotted 1} and Theorem \ref{thm: (un)linked stabilization and bruniannity}.

\subsection{The setting\label{sec: setting unlinked}}
Let us do a quick recap, without proofs and technical details, of the constructions involved in \cite{(Un)knotted}. 
Starting with $M$ a doubly toroidal 4-manifold, with a given pair of \emph{framed} 2-tori $\{T_i\hookrightarrow M: i=1,2\}$ as in Definition \ref{DoublyToroidal}, and $K$ a knot in $S^3$, we can define the smooth 4-manifold $M_K$ to be the result of knot surgery operation along $T_1$ (see \cite{KnotSurgery} for further details on this operation).
\begin{equation}
    M_K:=(M\bslash \nu T_1)\cup_{\varphi_K} (S^1\times (S^3\bslash \nu K)).
\end{equation}
Note that by this definition the 2-torus $T_2$ is  a submanifold of $M_K$, so we can define $Z_K$ as the generalized fiber sum between $M_K$ and the manifold $N_g$ along the 2-tori $T_2\subset M_K$ and $x\times y \subset  N$. In Section \ref{sec: long KT} we will review the construction of $N_g$ and the framed 2-torus $x\times y\subseteq N_g$.
\begin{equation}\label{eq X_K^Z}
    Z_{K,g}:=(M_K\bslash\nu T_2)\cup (N_g\bslash \nu(x\times y)).
\end{equation}
One can check that $\pi_1(Z_{K,g})$ is isomorphic to the free group with $g$ generators and in \cite{(Un)knotted} is fixed a set of $g$ simple loop $\mathcal{L}$ representing generators for $\pi_1(Z_{K,g})$ (after connecting them to a base point). Moreover, the elements of $\mathcal{L}$ are a subset of  $N\bslash \nu(x\times y)$, and a framing for each of those loops is fixed. \\
Finally, we can perform loop surgery along each element in the collection $\mathcal{L}$ in $Z_{K,g}$ with respect to the chosen framings and obtain a smooth manifold $Z_{K,g}^*$. 
\[Z_{K,g}^*:=(Z_{K,g}\bslash\nu \mathcal{L})\cup \left(\bigsqcup_g (D^2\times S^2)\right).
\]
Similarly, since $\mathcal{L}\subseteq N_g$, we can define 
\[N_g^*:=\left(N_g\bslash(\nu\mathcal{L})\right)\cup\left(\bigsqcup_g (D^2\times S^2)\right).\]
In \cite[Proposition 13]{(Un)knotted} the authors prove that there exists a diffeomorphism $\phi_K:Z_K^*\to M\# g(S^2\times S^2)$ and the 2-link $\Gamma_{K,g}$ is defined to be the image through $\phi_K$ of the belt spheres of the loop surgeries, that is
\begin{equation}\label{eq spheres}
    \Gamma_{K,g}:=\phi_K(\Gamma_g)\subset M\# g(S^2\times S^2),
\end{equation}
where \begin{equation}\label{eq: Gamma g}
    \Gamma_g:=\bigsqcup_g(\{0\}\times S^2)\subseteq N_g^*\bslash\nu(x\times y)\subseteq Z_{K,g}^*.
\end{equation}

\begin{constr}[Nullhomotopic 2-links \cite{(Un)knotted}]\label{constr: brunnianl exotic 2-links}
    Even if we will omit it from the notation, we would like to remark that the 2-link $\Gamma_{K,g}$ are well determined by the following datum:
\begin{equation}\label{eq: def 2-links unknotted}
\Gamma_K=\Gamma_{K,g}=\Gamma_K(M,g,T_1,T_2,\phi_K)\subseteq M\#g(\Stab),
\end{equation}
where $M$ is a doubly toroidal manifold with \emph{framed} 2-tori $T_1$ and $T_2$, $g$ is a positive natural number defining the number of components, $K\subseteq S^3$ is an oriented knot and $\phi_K:Z_{K,g}^*\to M\# g(\Stab)$ is a diffeomorphism (see \cite[Propositions 13 and 14]{(Un)knotted}).
\end{constr}

\subsubsection{The Kodaira Thurston manifold and $N_g$\label{sec: long KT}}

It is well known that the Kodaira Thurston manifold $N$ can be described as the Cartesian product  $S^1\times Y$, see \cite[Section 3.2]{(Un)knotted}. The 3-manifold $Y$ is the mapping torus of a Dehn twist $D_a: T^2\to T^2$ around an embedded loop $a$ in $T^2$ which together with another loop $b$ forms a basis for $H_1(T^2)$. Using the notation introduced in \cite{(Un)knotted} we can write the Kodaira Thurston manifold $N$ as a $a\times b$ bundle over $x\times\ y$.
\begin{equation}\label{eq: Kodaira as T2 over T2}
    a\times b\hookrightarrow N\to x\times y,
\end{equation}
where $x$ denotes the $S^1$ factor of $N$ and the loop $y$ is a section of $Y$ as $T^2$ bundle over $S^1$.

The proof of \cite[Lemma 2]{(Un)knotted} gives us another way to describe the 3-manifold $Y$: it is the result of a $(1,1)$-Dehn surgery starting from the 3-torus $T^3$ performed along a simple loop $S^1\times\{1\}\times\{1\}$ in $S^1\times S^1\times S^1=T^3$. By this description starting with a Kirby diagram of $T^2\times D^2$, which has the 3-torus $T^3$ as boundary, we can produce a 4-manifold with boundary $Y$ by adding a 2-handle in correspondence to the loop $a=(S^1\times\{1\})\times\{1\}$ with the $1$-framing, see Figure \ref{fig:KT} with $n=1$. 
\begin{oss}\label{ossframing}
    The loop $y$ is the blue loop in Figure $\ref{fig:KT}$, we can fix a framing for the loop $\{p\}\times y\subseteq N$ by fixing the $0$ framing for the loop in the picture.
\end{oss}\begin{figure}[h]
    \centering
    \includegraphics[width=0.45\textwidth]{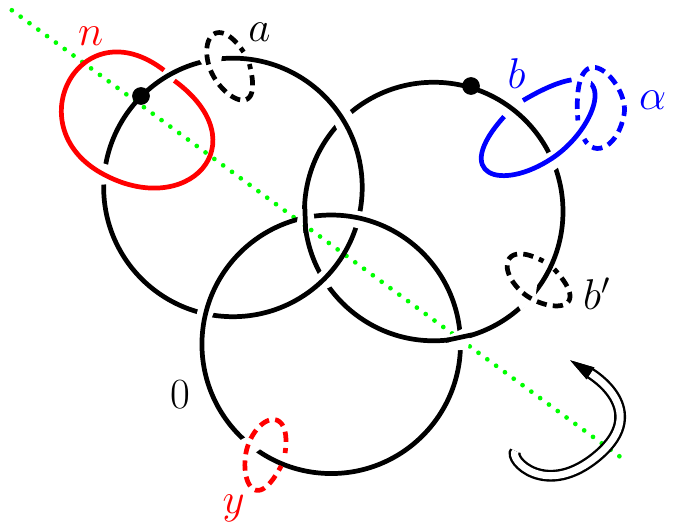}
    \caption{This picture is a Kirby diagram for a 4-manifold with boundary, there are no 3 or 4-handles. The handles drawn in black, two 1-handles and a 2-handle with framing $0$ in the shape of a Borromean link, represents the manifold $T^2\times D^2$. With the addition of the red 2-handle with framing $n$ the boundary is the 3-manifold $Y(n)$, where $Y=Y(1)$. The rotation of $\pi$ around the dotted green line describe an automorphism of $Y(n)$, since replacing all $1$-handles with $0$-framed 2-handles does not change the diffeomorphism type of the boundary component. The loops $y,a,b,b',\alpha$ are framed in $Y(n)$ with the 0-framing.\label{fig:KT}}
\end{figure}
\begin{constr}(The 4-manifold $N_g$ \cite[Section 3.3]{(Un)knotted})\label{constr: long KT}
For $g\in \N^+$, we define the manifold $N_g$ and a 2-torus $(x\times y)\subseteq N_g$ by induction. We set $N_1$ to be the Kodaira Thurston manifold $N$ and $N_{g+1}:=N_g\#_{T^2} N$, where the generalized fiber sum is performed along the framed 2-tori $(x\times y)\subseteq N_g$ and $(x\times y')\subseteq N$ where $y'$ is a parallel (and framed) push off of $y$ in $Y$. To complete the induction step we define the framed 2-torus $(x\times y)\subseteq N_{g+1}$ to be the framed 2-torus $(x\times y)\subseteq (N\bslash\nu(x\times y'))\subseteq N_{g+1}$.
\end{constr}
\begin{oss}
     The manifold $N_g$ is the Cartesian product of $S^1$ and a 3-manifold $Y_g$ as described in \cite[Section 3.3]{(Un)knotted}. So, the notation $(x\times y)\subseteq N_g$ still makes sense since we can interpret $x$ to be the $S^1$ factor and $y$ to be a (framed) loop in $Y_g$. 
\end{oss}

\subsubsection{The Moishezon trick} In this section we will see when a loop surgery on a 4-manifold can be replaced by a 0-log transform on a 2-torus and another loop surgery. This is a trick due to  Moishezon \cite{moishezon1977complex} and it is useful to compute the diffeomorphism type of 4-manifolds in many cases. Let $T\subseteq X$ be a framed 2-torus with a framing $F:\nu T\to S^1\times S^1\times D^2$. Consider the loop $\beta:=\{*\}\times S^1\times\{0\}\subseteq S^1\times S^1\times D^2$ equipped with the product framing and let $X^*$ be the manifold obtained from $X$ as result of loop surgery along the framed loop $F^{-1}(\beta)$. Consider the manifold $\widehat{X}$ which is the result of a $0$-log transform along $T$, that is
\[\widehat{X}:=X\bslash\nu T\cup_{\psi_0} (S^1\times S^1\times D^2), \]
where the gluing map $\psi_0:\partial\nu T\to S^1\times S^1\times\partial D^2$ sends $F^{-1}(\beta)$ to $\{*\}\times\{*\}\times\partial D^2$. 
Let $m\subseteq\partial\nu T$ be a meridian of $T$ equipped with a framing induced by a 2-disk fiber of $\nu T$. Consider the manifold $\widehat{X}_m$ which is the result of loop surgery along the framed loop $m\subseteq \widehat{X}$.
\begin{thm}\label{thm: mois}(\cite{moishezon1977complex}) There exists a diffeomorphism $  \widehat{X}_m\to X^*$ which restricts to the identity on $\partial X$.\end{thm} 
\begin{proof}[Sketch of proof]
    Gompf explains in \cite[Proof of Lemma 3]{Gompf_1991} that the manifolds $X, X^*$ and $\widehat{X}$ are identical outside the sets $\nu T,(\nu T)^*$ and $\nu\widehat{T}$ respectively. Here $\widehat{T}\subseteq \widehat{X}$ denote the 2-torus $S^1\times S^1\times \{0\}\subseteq\widehat{X}$. Moreover, Gompf shows Kirby diagrams for those three pieces: Figures \ref{fig: Boro S0} and \ref{fig: Boro S10} are Kirby diagrams for $\nu\widehat{T}$ and  $(\nu T)^*$ respectively, while a Kirby diagram for $\nu T$ can be obtained from Figure \ref{fig: Boro S10} by replacing the 2-handle bounding the red region with a dotted 1-handle. One can easily identify the loop $m\subseteq\widehat{X}$ in these pictures,  and see that $\nu\widehat{T}$ after a loop surgery along $m$ becomes diffeomorphic to $(\nu T)^*$ relative to the boundary. Figures \ref{fig: Boro S1}-\ref{fig: Boro S10} shows this diffeomorphism given by handle slides and cancellations.
\end{proof}
\subsubsection{Cobordism argument\label{Cobordism argument}}
In this section we provide a tool to compute the diffeomorphism type of 4-manifold obtained as a generalized fiber sum with an extra $\Stab$ summand.
The following proposition is the result of an argument used in the proof of \cite[Theorem 1]{DissolvingKnotSurgery}, our contribution is to state it in a general form that specify the framings of the loops and to take care of the boundary. A similar result is given in \cite[Lemma 4]{Gompf_1991} and similar techniques are used in \cite{baykur2013round}.

\begin{prop}\label{Baykur}
    Let us assume that $X_i$ for $i=1,2$ are two  4-manifolds each with a 2-torus  $T_i\subseteq X_i$ of self intersection $[T_i]^2=0$ and a framing $F_i:{\nu T_i}\to S^1\times S^1 \times D^2$. We denote as $M:=X_1\#_{T^2}X_2$ the manifold obtained from the manifolds $X_1$ and $X_2$ by performing a generalized fiber sum along the framed 2-tori $(T_1,F_1)$ and $(T_2,F_2)$.
    Define the framed loops \begin{equation}
        \alpha_i:F_i^{-1}(S^1\times\{1\}\times\{0\})\quad\text{and}\qquad \beta_i:=F_i^{-1}(\{1\}\times S^1\times\{0\}),
    \end{equation}
    where the framing is the one induced by $F_i$ (\ref{eq: definition framing induced by the restriction}). Let $\mu$ be the framed meridian of $T_1$ with framing induced by the 2-disk fiber of $\nu T_1$. There exists a framed loop $\alpha$ (and $\beta$) embedded in $X_1\#X_2$  defined as a band sum between $\alpha_1$ and $\alpha_2$ (respectively $\beta_1$ and $\beta_2$), such that the manifolds $(X_1\# X_2)_{\alpha,\beta}$ and $M_\mu$, obtained by doing loop surgery along the corresponding framed loops, are orientation preserving diffeomorphic.
    \begin{equation}\label{eq: cobordism diffeo}
        M_\mu\cong(X_1\# X_2)_{\alpha,\beta}
    \end{equation}
    Moreover, we can assume that the diffeomorphism (\ref{eq: cobordism diffeo}) is the identity
on the boundary $\partial M = \partial X_1 \sqcup\partial X_2$.
\end{prop}

\begin{figure}[h]
    \centering
    \includegraphics[width=0.4\textwidth]{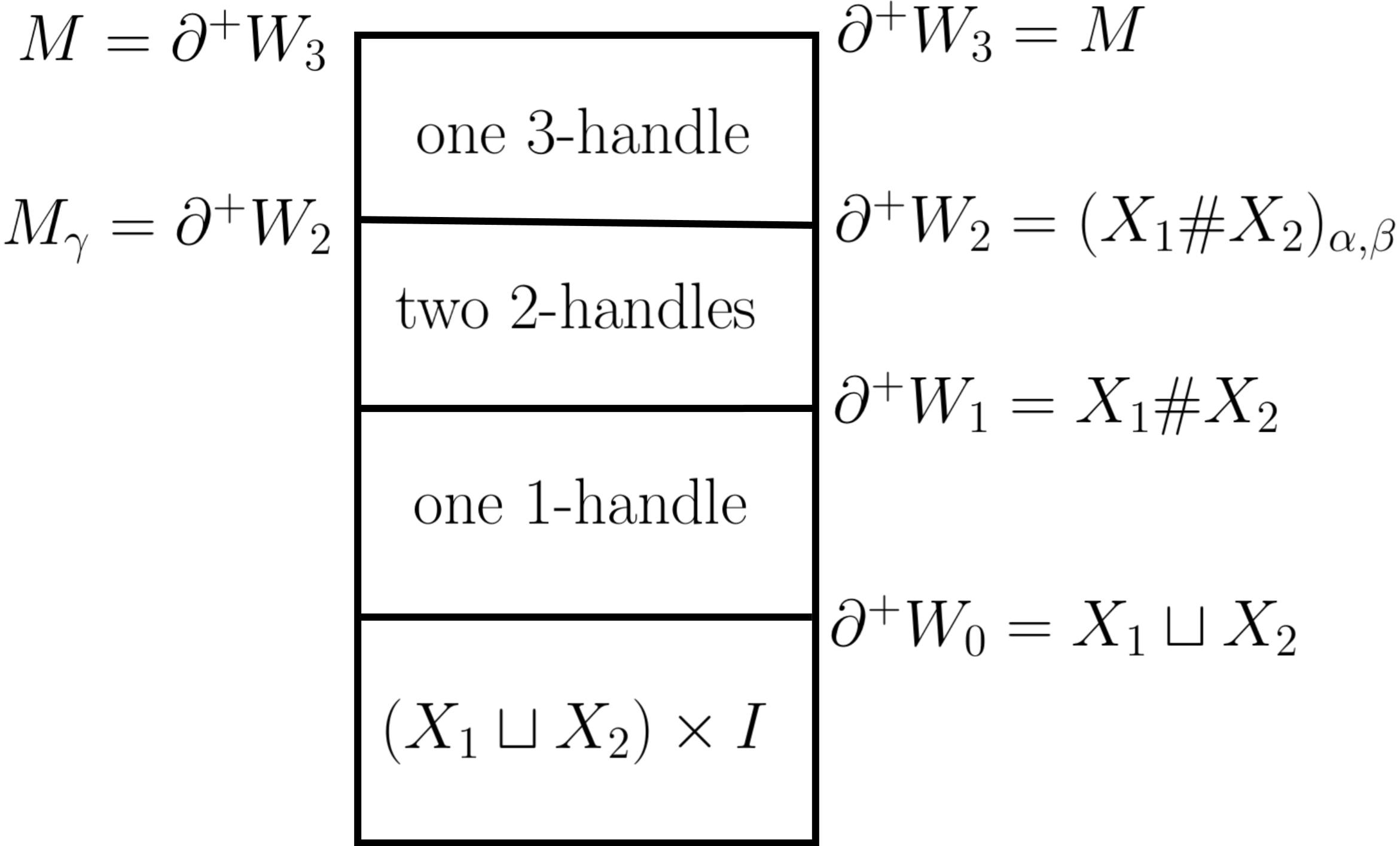}
    \caption{This picture is a schematic representation of the cobordism $W$ constructed in the proof of Proposition \ref{Baykur}.}
    \label{fig:Cobordism}
\end{figure}
   \begin{proof}[Sketch of proof.]
   We will construct a 5-dimensional cobordism $W$ between $X_1\sqcup X_2$ and $M$ and then we will consider a middle layer of this cobordism as the result of different surgeries, one coming from upstairs and one from downstairs, see figure \ref{fig:Cobordism}.
   
  The cobordism $W$ is obtained from  $(X_1\sqcup X_2)\times I$ by attaching a $ (T^2\times D^2)\times I$ glued along $T^2\times D^2\times\partial I$ on ${\nu T_1}\times \{1\}\sqcup{ \nu T_2}\times \{1\}$. We see $W$ as the result of adding four handles to $(X_1\sqcup X_2)\times I$, these handles come from the standard handles decomposition of the $T^2$ in the product $T^2\times D^2\times I$, therefore, one 1-handle, two 2-handles and one 3-handle.
   Let us denote as $W_i$ the cobordism obtained attaching only the $k$-handles with $k\leq i$, that is \[W_{-1}=(X_1\sqcup X_2)\times I \qquad W_i=W_{-1}\cup\bigcup_{k\leq i} k-handles.\]
   Let us denote as $\partial^+W_i$ the \emph{upper} boundary components of $W_i$, see Figure \ref{fig:Cobordism}, i.e. 
   $\partial^+W_i=\partial W_i\bslash ((X_1\sqcup X_2)\times \{0\})$.
There are not $0$-handles, therefore, $\partial^+ W_{-1}=\partial^+ W_0= ((X_1\sqcup X_2))\times \{1\}$.
   There is only one 1-handle with attaching points in the two components of $\partial^+ W_0$, so $\partial^+ W_1$ is diffeomorphic to the connected sum $X_1\# X_2$.
   There are two 2-handles, corresponding to the two 1-handles of $T^2$, to compute the diffeomorphism type $\partial^+ W_2$ we have to perform two loop surgery along the framed attaching circles $\alpha$ and $\beta$ in $\partial^+ W_1$, so there exists the following diffeomorphism
 \[\partial^+ W_2\cong (X_1\# X_2)_{\alpha,\beta}.  \]
 One can check that the framed loop $\alpha$ (and similarly for $\beta$) is the band sum of $\alpha_1$ and $\alpha_2$, by  explicitly writing down the attaching circle of the corresponding 2-handle.
   Since there are no $4$ or $5$-handles 
   $\partial^+ W_3=\partial^+ W_4=\partial^+ W_5=\partial^+ W\cong M$. There is only one $3$-handle, but we can see it, by the dual handle decomposition, as a $2$-handle attached to $W\bslash W_3$ on $\partial^+ W_3$, call $\mu$ the corresponding framed attaching loop. Notice that $\mu$ should be the framed meridian of the attaching 2-sphere of the 3-handle, which is made up of two 2-disks, one in $T_1$ and on in  $T_2$ and an annulus between them in $(X_1\# X_2)_{\alpha,\beta}$. It follows that $\mu$ is a framed meridian of $T_1$ in $\partial\nu T_1\subseteq M$. We have proved the existence of a diffeomorphism $\partial^+W_2\cong M_\mu$ and this concludes the proof.
   \end{proof}

    \begin{definition} \label{def: torus spin compatible framing} Given a coprime pair $(p,q)$ of integers, consider $\gamma_{q/p}\subseteq S^1\times S^1\times\{0\}$ the $q/p$ simple curve, i.e. a simple curve satisfying  \[   [\gamma_{q/p}]=p[S^1\times\{*\}\times\{0\}]+q[\{*\}\times S^1\times\{0\}]\in H_1(S^1\times S^1\times \{0\};\Z).\] We equip the curve $\gamma_{q/p}\subseteq (S^1\times S^1)\times D^2$ with the product framing. We say that a framing $F:\nu T\to S^1\times S^1\times D^2$ for a 2-torus $T\subseteq X$ is $q/p$ \emph{spin-compatible} (resp. \emph{$q/p$ not-spin-compatible}) if $\Stab$ (resp. $\StabTwist$) is an element of ${\Ss}(T,F^{-1}(\gamma_{q/p}))$, where $F^{-1}(\gamma_{q/p})$ is equipped with the framing induced by $F$.   \end{definition}
   The previous definition should be compared with \cite[Definition 2.1]{larson2018surgery}.
   \begin{cor} \label{cor: Cobordism with s2xs2 comapible torus}Let everything be defined as in the Proposition \ref{Baykur}. Assume that $T_1\subseteq X_1$ has simply connected complement.
   Define the framed loop $\widetilde{\alpha_2}$ to be the framed loop $\alpha_2$ if the framing $F_1$ of $T_1\subseteq X_1$ is $0/1$ spin-compatible, or $(\alpha_2)_{op}$ if $F_1$ is  $0/1$ not-spin-compatible. Similarly, we define $\widetilde{\beta_2}$ to be $\beta_2$ if $F_1$ is $1/0$ spin-compatible,  or $(\beta_2)_{op}$ if $F_1$ is $1/0$ not-spin-compatible.  
    There exists a diffeomorphism
       \begin{equation}\label{eq: diffeo cobordism swaps}
           M\#B\cong X_1\#(X_2)_{\widetilde{\alpha}_2,\widetilde{\beta}_2},
       \end{equation}
       where $B$ is $\Stab$ if $T_1$ is an ordinary surface, or it is $\StabTwist$ if $T_1$ is a characteristic surface. The manifold $(X_2)_{\widetilde{\alpha}_2,\widetilde{\beta}_2}$ is  obtained from $X_2$ by performing two loop surgery operations along the two framed loops $\widetilde{\alpha}_2,\widetilde{\beta}_2\subseteq X_2$ (after making them disjoint with a small isotopy). 
       Moreover,  we can assume that the diffeomorphism (\ref{eq: diffeo cobordism swaps}) is the identity
on the boundary.
   \end{cor}
   \begin{proof}
   Proposition \ref{Baykur} gives us  diffeomorphism (\ref{eq: cobordism diffeo}).
       Since $T_1$ is an ordinary surface with simply connected complement, it has a geometrically dual immersed 2-sphere in $X_1$. It follows that $\mu$ bounds an immersed 2-disk $D$ in $M=X_1\#_{T^2} X_2$. We can assume that the framing of $\mu$ is compatible with $D$ if $T_1$ is ordinary, and not compatible if $T_1$ is characteristic. Therefore, $M_\mu$ is diffeomorphic to $M\# B$ rel. boundary.
       The loop $\alpha_1$  bounds immersed 2-disk in $X_1$, which can be chosen to be compatible (resp. not compatible) with the framing of $\alpha_1$ if $F_1$ is $0/1$ spin-compatible (resp. not-spin-compatible). Similarly, for $\beta_1$. It follows that the framed loops $\alpha,\beta\subseteq (X_1\# X_2)$ are isotopic to the framed loops $\widetilde{\alpha}_2,\widetilde{\beta}_2$ by Lemma \ref{lem: ConSumOfLoop and Disk}. This implies the existence of a diffeomorphism, relative to the boundary, between $(X_1\# X_2)_{\alpha,\beta}$ and $ X_1\#(X_2)_{\widetilde{\alpha}_2,\widetilde{\beta}_2}$. 
   \end{proof}
\begin{cor}(\cite{auckly2003families,akbulut2002variations,DissolvingKnotSurgery})\label{cor: stabilization rel boundary}
    Let $T_1,T_2\subseteq M$ and $M_K$ be defined as in Section \ref{sec: setting unlinked}, if the framing of $T_1\subseteq M$ is $1/0$ spin-compatible, then there exists a diffeomorphism
    \begin{equation}\label{eq: diffeo knot surgery}
        M_K\#(\Stab)\cong M\#(\Stab).
    \end{equation} If the framing of $T_1\subseteq M\bslash\nu T_2$ is $0/1$ spin-compatible, then the diffeomorphism (\ref{eq: diffeo knot surgery}) is the identity on $\nu T_2$.
\end{cor}
\begin{oss}
    It is well known that a manifold produced by knot surgery is stabilized by  $\Stab$ or $\StabTwist$, see \cite{auckly2003families},\cite[Theorem 2.2]{akbulut2002variations},\cite{DissolvingKnotSurgery}. Our contribution is to show that it can be stabilized with $\Stab$ and relative to $\nu T_2$ under some condition.
\end{oss}
\begin{proof}
    The manifold $M_K$ can be described as the generalized fiber sum
    \[M_K=M\#_{T^2}(S^1\times S^3_0(K))\]
    along the framed 2-tori $T_1\subseteq M$ and $S^1\times\mu_K\subseteq S^1\times S^3_0(K)$, where $\mu_K$ is the meridian of the knot $K\subseteq S^3$ equipped with the 0-framing and $S^3_0(K)$ is the result of a 0-Dehn surgery along the knot $K\subseteq S^3$ \cite{KnotSurgery}. By Corollary \ref{cor: Cobordism with s2xs2 comapible torus}
    \[M_K\# B\cong M\# (S^1\times S^3_0(K))_{\widetilde{\alpha}_2,\beta_2},\]
    where the diffeomorphism is relative to $\nu T_2\subseteq M$ if  the framing of $T_1\subseteq M\bslash\nu T_2$ is $1/0$ spin-compatible. If $T_1\subseteq M\bslash\nu T_2$ is ordinary, then $B=\Stab$. If $T_1\subseteq M\bslash\nu T_2$ is characteristic, then $B=\StabTwist$, $M\bslash\nu T_2$ is not spin and simply connected. It follows that $M_K\# B$ is always diffeomorphic to $M\# (\Stab)$ rel. $\nu T_2$.
    Now we are going to imitate the proof of \cite[Lemma 6]{(Un)knotted} to show that $(S^1\times S^3_0(K))_{\widetilde{\alpha}_2,\beta_2}$ is diffeomorphic to $\Stab$.
    Since the framing of $\beta_2$ is the framing induced by $S^1\times \mu_K$, we can apply the Moishezon trick \cite{moishezon1977complex} (Theorem \ref{thm: mois}). So, we can replace the loop surgery in $(S^1\times S^3_0(K))_{\widetilde{\alpha}_2,\beta_2}$ along $\beta_2$, with a multiplicity $0$-log transform along the framed 2-torus $S^1\times \mu_K$ and a new loop surgery along a loop $m$. We call $\widehat{X}$ the manifold resulting from $S^1\times S^3_0(K)$ by this torus surgery. 
    Via  surgery presentation of $S^3_0(K)$ we can compute that $\widehat{X}$ is diffeomorphic to $S^1\times S^3$, see Figure \ref{fig: suegery pres of X}. This implies that $\widehat{X}$ after the surgery along $\widetilde{\alpha}_2$ is diffeomorphic to $S^4$, and diffeomorphic to $\Stab$ or $\StabTwist$ after the second surgery along $m$. It is $\Stab$ since $(S^1\times S^3_0(K))_{\widetilde{\alpha}_2,\beta_2}$  has even intersection form, isomorphic to the one of $S^1\times S^3_0(K)$.
\end{proof}
\begin{figure}
    \centering
    \includegraphics[width=0.2\textwidth]{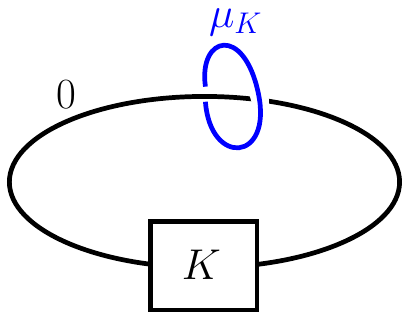}
    \caption{ Let $W$ be the 4-manifold obtained from $D^4$ by attaching a 2-handle along  the knot $K\subseteq S^3$ with the $0$-framing, $S^3_0(K)$ is the boundary of $W$. Let $\widehat{W}$ be the 4-manifold obtained from $W$ attaching a 2-handle along $\mu_K$ with the $0$-framing. The 3-manifold $\partial \widehat{W}$ is obtained from $S^3_0(K)$ by a $0$-Dehn surgery along $\mu_K$, therefore, the following equality holds $\widehat{X}=S^1\times\partial \widehat{W}$.
    One can replace the knot $K$ with the unknot by handles slides, showing that $\partial \widehat{W}1 S^3$.}
    \label{fig: suegery pres of X}
\end{figure}

\subsection{Internal stabilization of the 2-spheres $S_K$\label{sec: internal Unknotted}}
In this section we study the smooth isotopy class of the 2-spheres $S_K\subseteq M\#(\Stab)$ after one internal stabilization, the main result of this section is Theorem \ref{thm: Int unknotted 1}.
\subsubsection{Strategy\label{sec: strategy int unknotted}}
The 2-sphere $S_K$ is by definition $\phi_K(\Gamma_1)$, where $\phi_K:Z_{K,1}^*\to M\# (\Stab)$ is a diffeomorphism. Therefore, $S_K \subseteq M\# (\Stab)$ and $\Gamma_1\subseteq N^*\bslash\nu(x\times y)\subseteq Z_{K,1}^*$ are smoothly equivalent. In \cite{(Un)knotted} the authors prove that $N^*$ is diffeomorphic to $(S^1\times S^1\times S^2)\#(\Stab)$, so one would like to see where $\Gamma_1$ is sent by this diffeomorphism to possibly gain a nice representation of this surface to work with. Unfortunately this seems difficult to realize. If instead of $\Gamma_1$ we try to track the surface obtained by a stabilization $(\Gamma_1)_{\#T^2}$, then it is possible to explicitly find the corresponding surface in $(S^1\times S^1\times S^2)\#(\Stab)$. This is very close to what we are going to do. In particular, we will consider a chain of diffeomorphisms between the manifolds
\begin{equation}\label{eq: strategy internal Unknotted}
    N^*{\rightarrow} \widehat{N}_m\overset{f}{\rightarrow} (S^1\times S^1\times S^2)\#(\Stab)
\end{equation}
keeping track not only of $(\Gamma_1)_{\#T^2}$, but also of the 2-torus $x\times y\subseteq N^*$. The auxiliary manifold $\widehat{N}_m$ is obtained from $\widehat{N}$, defined in (\ref{eq: Nhat}), by performing loop surgery along a nullhomotopic loop $m$. The first diffeomorphism is given by the Moishezon trick \cite{moishezon1977complex} and since it can be realized by a sequence of handle slides in a Kirby diagram, we will be able to follow the surfaces. Moreover, we will see that $(\Gamma_1)_{\#T^2}$ is sent to $\widehat{x\times b}\subseteq \widehat{N}_m$, which is a 2-torus already existing in $\widehat{N}$ and disjoint from the loop $m$, see Lemma \ref{lem: internal step 1}. 

Since $\widehat{N}$ is diffeomorphic to $S^1\times S^1\times S^2$ we can track in this diffeomorphism the 2-tori $x\times y$ and $\widehat{x\times b}$ together with the loop $m$ to be able later to reconstruct upstairs the image of $\widehat{x\times b}$ in $(S^1\times S^1\times S^2)\#(\Stab)$, see Lemma \ref{lem: Identifing N hat}. 
We would like the loop $m$ to be  nullhomotopic in the complement of $\widehat{x\times b}\subseteq\widehat{N}$ to further simplify $f(\widehat{x\times b})$, but it is not.  To solve this issue we will perform a generalized fibersum  with $M_K$ along the framed 2-tori $T_2$ and the one corresponding to $x\times y$, see Lemma \ref{lem: trivializing u}. Lemma \ref{lem: internal final step: remove K} finally identifies a surface $\mathcal{T}\subseteq M\#(\Stab)$ smoothly equivalent to $(\Gamma_1)_{\# T^2}\subseteq Z_{K,1}^*$. Surprisingly, $\mathcal{T}$ is defined as in Construction \ref{const: nullhomologus tori} and it is contained in $\nu T_2\subseteq X\bslash D^4\subseteq M\#(\Stab)$. So we get the final result by applying Theorem \ref{thm: ExtStab nullhomologus 2-tori}.

\subsubsection{Following the surfaces}\label{subsec: follwing the surfaces}
In this section we will use $N$ to denote the Kodaira Thurston manifold (\ref{eq: Kodaira as T2 over T2}).
 The manifold $\widehat{N}$ is defined to be the result of a 0-log transform on the framed 2-torus $x\times b \subseteq N$. The manifold $\widehat{N}$ has an embedded 2-torus $\widehat{x\times b}$ that is core 2-torus of the surgery and it comes with a preferred framing.
 In the decomposition
\begin{equation}\label{eq: Nhat}
    \widehat{N}:=(N\bslash\nu (x\times b))\cup_{\varphi_0} (T^2\times D^2)
\end{equation}
the 2-torus $\widehat{x\times b}$ is the surface $T^2\times\{0\}\subseteq T^2\times D^2\subseteq\widehat{N}$ and the gluing map $\varphi_0:\partial\nu( x\times b)\to T^2\times \partial D^2$ is defined as in \cite[Section 3.2]{(Un)knotted}, i.e. it sends a push off of the loop $\{p\}\times b$ to $\{p\}\times\partial D^2$.\\
 Let $m$ be the meridian of the torus $x\times b$ contained in $\partial\nu(x\times b)$, this loop still defines an embedded loop $m$ in $\widehat{N}$. Also $m$ comes with a preferred framing induced by the 2-disk fiber of the disk bundle $\nu(x\times b)$.

The gluing map $\varphi_0$ can be chosen to agree with the product structure of $N=S^1\times Y$, so that the surgery becomes $1+3$ dimensional operation, where the 3-dimensional operation is a Dehn surgery on $Y$ producing a 3-manifold $\widehat{Y}$. Therefore,
\begin{equation}\label{eq: Nhat product}
    \widehat{N}=S^1\times \left((Y\bslash\nu(b))\cup_{\psi_0} (S^1\times D^2)\right)=S^1\times\widehat{Y},
\end{equation}
where the gluing map $\psi_0:\partial\nu(b)\to S^1\times \partial D^2$ sends the push off of $b$ to $\{p\}\times\partial D^2$.
This defines a core loop $\alpha=S^1\times\{0\}\subseteq S^1\times D^2\subseteq \widehat{Y}$ of the Dehn surgery and a preferred framing for it. 
\begin{oss}\label{oss: framing 2-torus logtranform}
    The 2-torus $\widehat{x\times b}$ can be identified with $S^1\times \alpha\subseteq S^1\times \widehat{Y}$ and the preferred framing of $\widehat{x\times b}$ is induced by a preferred framing of $\alpha$.
Moreover, the loop $m$ can be written as 
\[m=\{1\}\times \overline{\alpha}\subseteq S^1\times(S^1\times\partial D^2)\subseteq\widehat{N},\] where $\overline{\alpha}=S^1\times\{1\}\subseteq S^1\times \partial D^2\subseteq \widehat{Y}$. Equivalently, the framed curve $m$
is the push off of $\{1\}\times \alpha$ with respect to the framing of $\widehat{x\times b}$.
\end{oss}

\begin{lem} \label{lem: internal step 1}
Let $\widehat{N}_m$ be the 4-manifold obtained from $\widehat{N}$ by a loop surgery along the framed loop $m$, where the loop surgery operation is done away from $\widehat{x\times b}$ and $x\times y$, so that $\widehat{x\times b}$ and $x\times y$ still define embedded surfaces in $\widehat{N}_m$. Let $\Gamma_1\subseteq N^*$ be the belt 2-sphere of the loop surgery (\ref{eq: Gamma g}).\\
    There exists a diffeomorphism of triples
    \begin{equation}
        (N^*, x\times y, (\Gamma_1)_{\#T^2}))\to(\widehat{N}_m, x\times y, \widehat{x\times b}),
    \end{equation}
    which restricts to the identity on a neighborhood of $x\times y$ and where $(\Gamma_1)_{\#T^2}$ is a surface obtained from $\Gamma_1$ after an internal stabilization.
\end{lem} 
\begin{proof}
    The proof of this lemma lies on the Moishezon trick \cite{moishezon1977complex} (Theorem \ref{thm: mois}). From the explanation by  Gompf \cite[Proof of Lemma 3]{Gompf_1991} we get a diffeomorphism $\phi:\widehat{N}_m\to N^*$. This diffeomorphism is the identity on $N\bslash\nu(x\times b)$ and it is given by the sequence of handle slides and cancellation appearing in Figure \ref{fig: Boro S} between  $(\nu(\widehat{x\times b}))_m:=\widehat{N}_m\bslash(N\bslash\nu(x\times b))\subseteq \widehat{N}_m$ and $(\nu({x\times b}))^*:=N^*\bslash (N\bslash\nu(x\times b))\subseteq N^*$. It is left to check that $\phi(\widehat{x\times b})=(\Gamma_1)_{\# T^2}$ and this is done by following those surfaces in the slides, isotopies and cancellation of Figure  \ref{fig: Boro S}. In \cite[Section 2]{Auckly_Kim_Melvin_Ruberman_2015} it is explained how to slide surfaces in Kirby diagrams. 
\end{proof}

\begin{figure}[]
     \centering
     \hfill
     \begin{subfigure}[b]{0.2\textwidth}         \centering\includegraphics[width=\textwidth]{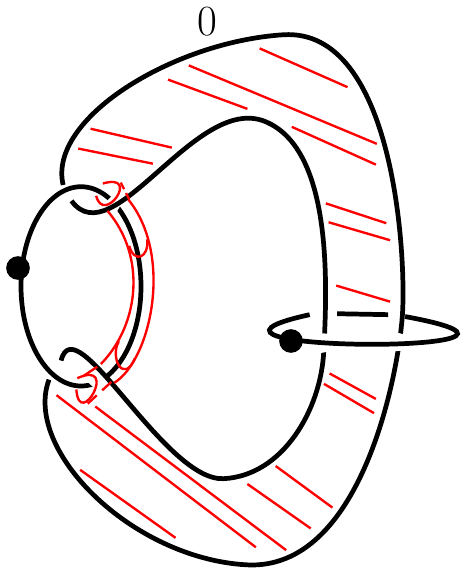}         \caption{\label{fig: Boro S0}}\end{subfigure}     \hfill 
     \begin{subfigure}[b]{0.25\textwidth}
         \centering
         \includegraphics[width=\textwidth]{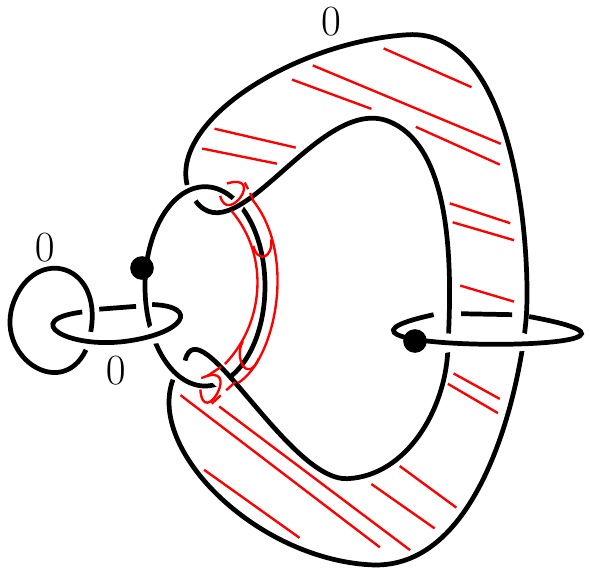}
         \caption{\label{fig: Boro S1}}
     \end{subfigure}
     \hfill
     \begin{subfigure}[b]{0.25\textwidth}
         \centering
         \includegraphics[width=\textwidth]{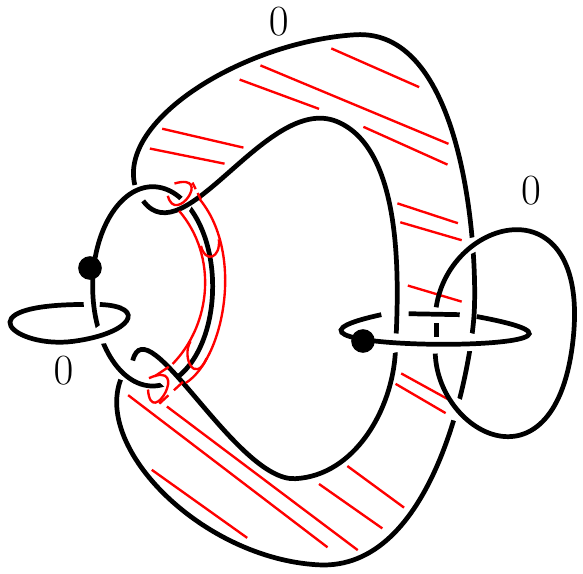}
         \caption{\label{fig: Boro S2}}
     \end{subfigure}
     \hfill
     \begin{subfigure}[b]{0.25\textwidth}
         \centering
         \includegraphics[width=\textwidth]{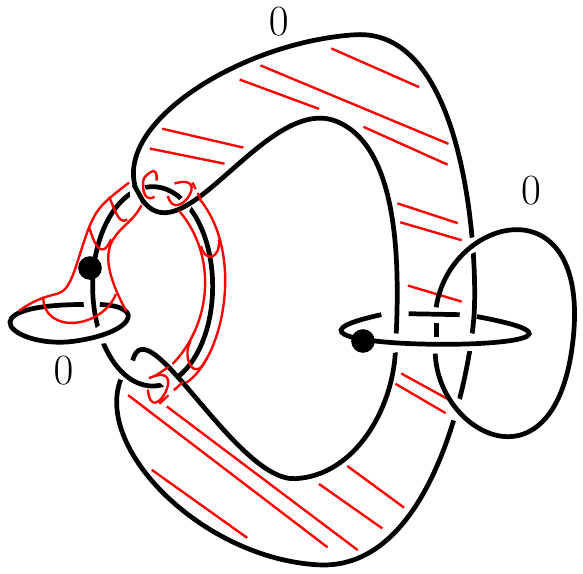}
         \caption{\label{fig: Boro S3}}
     \end{subfigure}
     \hfill
     \begin{subfigure}[b]{0.3\textwidth}
         \centering
         \includegraphics[width=\textwidth]{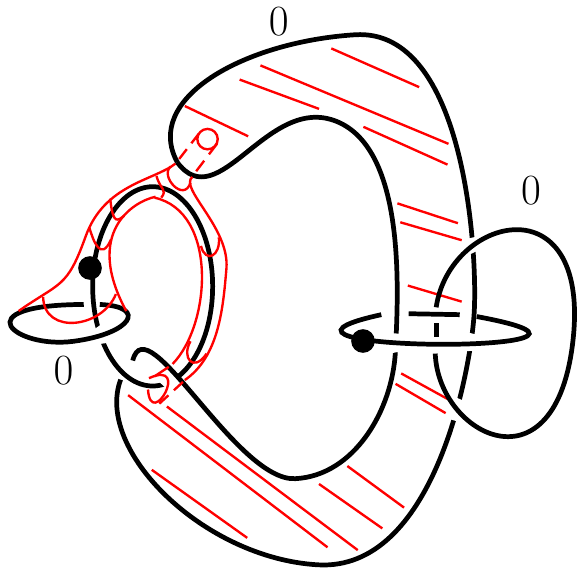}
         \caption{\label{fig: Boro S4}}
     \end{subfigure}
     \hfill
     \begin{subfigure}[b]{0.37\textwidth}
         \centering
         \includegraphics[width=\textwidth]{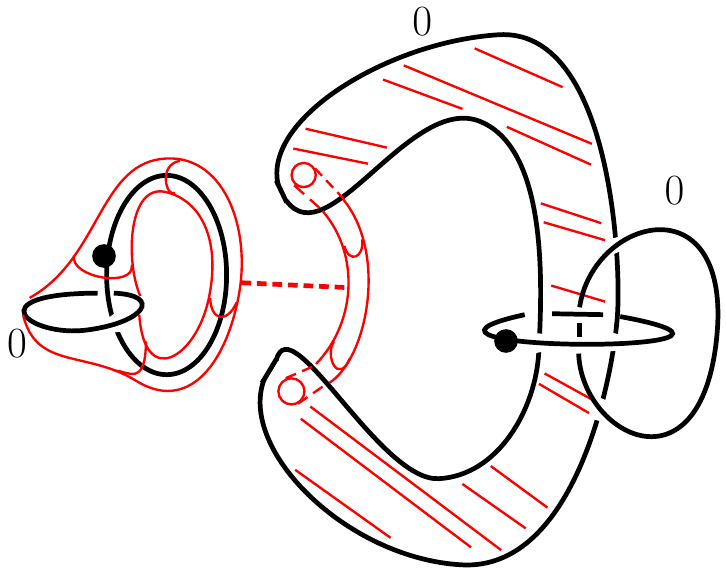}
         \caption{\label{fig: Boro S7}}
     \end{subfigure}
     \hfill
     \begin{subfigure}[b]{0.25\textwidth}         \centering        \includegraphics[width=\textwidth]{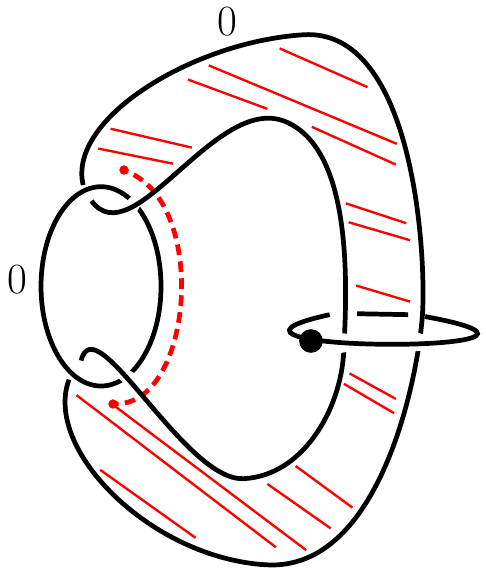}         \caption{\label{fig: Boro S10}}
     \end{subfigure}
     \caption{\label{fig: Boro S} Figure \ref{fig: Boro S0} is Kirby diagram of $T^2\times D^2\cong \nu(\widehat{x\times b})$ in $\widehat{N}$ and  the surface $\widehat{x\times b}$ is shaded in the picture (it is completed by a 2-disk in the 0-framed 2-handle).  Figure \ref{fig: Boro S1} depicts $(\nu(\widehat{x\times b}))_m$ a neighborhood of the surface $\widehat{x\times b}\subseteq \widehat{N}_m$. Figure \ref{fig: Boro S10} is Kirby diagram of $\nu(x\times b)^*\subseteq N^*$, shaded in this picture is the surface $\Gamma_1$ and the dotted arc is a chord that defines an internal stabilization for $\Gamma_1$.  The handle slide between Figures \ref{fig: Boro S2}-\ref{fig: Boro S3}  has  on the surface the effect of replacing a 2-disk, around the created intersection point between the preceding surface and the dotted 1-handle, with a tube connecting to the attaching circle of the 0-framed 2-handle plus a 2-disk in this 2-handle. An isotopy produces Figure \ref{fig: Boro S4}. Figure \ref{fig: Boro S7} is obtained repeating the previous steps for a new slide, symmetric to the preceding one, performed in the bottom left part of the diagram. Here the dotted arc denotes a tube, and it is showing how this surface can be viewed as a connected sum of the 2-sphere on the left and the 2-torus on the right.  The 2-sphere on the left is made up by a tube around the dotted 1-handle and two parallel 2-disk on the 0-framed 2-handle. This 2-sphere is smoothly unknotted (since it is a 2-sphere in a $S^3$, boundary of a $D^4$), the  1 and 2-handles cancel, leaving 
     Figure \ref{fig: Boro S10}.}
\end{figure}

\begin{lem} \label{lem: Identifing N hat}There exists a framed simple loop $\mu\subseteq S^1\times S^2$ and orientation preserving diffeomorphism of quadruples
    \begin{equation}
        \phi:(\widehat{N},x\times y,\widehat{x\times b},m)\to(S^1\times (S^1\times S^2), S^1\times(S^1\times\{1\}),S^1\times \mu, \{1\}\times \bar\mu),
    \end{equation}
    where  $\overline{\mu}\subseteq S^1\times S^2$ is a push off of $\mu$.
    The loop $\mu$ and its framing are described in Figure \ref{fig: mumu}. \\
    Moreover, the framings of $S^1\times(S^1\times\{1\}),S^1\times \mu$ and $\{1\}\times \bar\mu$, induced via $\phi$ from the corresponding framed submanifolds, are the product framings, i.e. they are the one induced by the product structure of $S^1\times (S^1\times S^2)$. If $b'$ is a push off of $b\subseteq Y$ and $x\times b'\subseteq \widehat N$, then $\phi(x\times b')$ is $S^1\times u$ where $u$ is meridian of $\mu$ (and $\bar\mu$) in $S^1\times S^2$. 
\end{lem}
   \begin{figure}[h]
     \centering
       \begin{subfigure}[b]{0.4\textwidth}
         \centering
         \includegraphics[width=0.55\textwidth]{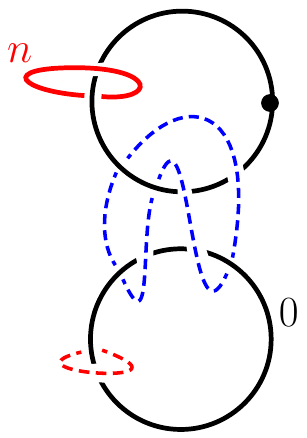}
         \caption{         \label{fig: intermidiat step}}  \end{subfigure}
         \hfill
     \begin{subfigure}[b]{0.4\textwidth}
         \centering
         \includegraphics[width=\textwidth]{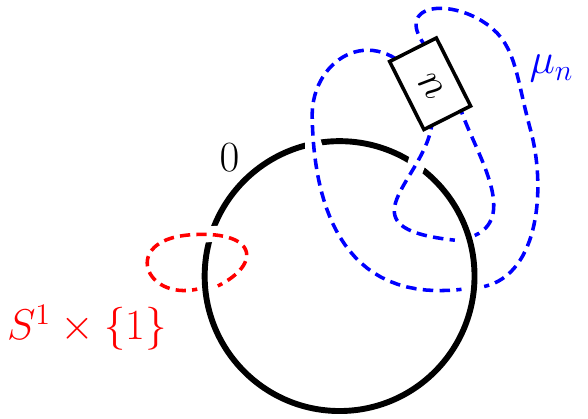}
         \caption{
         \label{fig: mu}}
     \end{subfigure}
     \caption{\label{fig: mumu}On the right we are identifying the manifold $S^1\times S^2$ as the boundary of a 4-manifold $D^2\times S^2$ depicted as a Kirby diagram with only one 0-handle and a 2-handle that is $0$-framed and unknotted. The curve $\mu_n$ is depicted in blue and framed with the $0$-framing. The box labelled $n$ indicates $n$ full twists. The curve $\mu$ is $\mu_1$ and it forms with the 2-handle a Whitehead link. The curve $\overline{\mu}$ is a longitude of the curve $\mu$  with $0$ linking number, and its framing is the $0$-framing. 
     The curve $\overline{\mu}$ is a longitude of the curve $\mu$  with $0$ linking number, and its framing is the $0$-framing.}
\end{figure}
\begin{proof}
    To prove this lemma is sufficient to prove that there exists a diffeomorphism of triples 
    \[\psi:(\widehat{Y},y,\alpha)\to(S^1\times S^2,S^1\times \{1\},\mu)\]
    for which the framings on $S^1\times \{1\}$ and $\mu$, induced by the framings of $y$ and $\alpha$, are the $0$-framings (with respect to Figure \ref{fig: mu}).
    Once the map $\psi$ is constructed, we can simply define $\phi=id_{S^1}\times \psi$ to obtain the desired diffeomorphism of quadruples. In particular, $\phi(m)$ is automatically sent to $\{1\}\times\overline{\mu}$, see Remark \ref{oss: framing 2-torus logtranform}. The loop $b'\subseteq \widehat{Y}$ is a meridian of the loop $\alpha$, therefore $\psi(b)$ is a meridian $u$ of the loop $\mu\subseteq S^1\times S^2$.\\
    The diffeomorphism $\psi$ is obtained via some  computations in Kirby calculus. The manifold $Y$ can be expressed as the boundary of a 4-manifold $W$, in Figure \ref{fig:KT}. We obtain a 4-manifold $\widehat{W}$ with boundary $\widehat{Y}$ by adding a 2-handle to $W$. The attaching circle of this new 2-handle is the loop $b$ with framing $0$ and $\alpha$ is its meridian equipped with the 0-framing. After sliding the loop $\alpha$ on dotted 1-handle on the right of Figure \ref{fig:KT}, we can cancel the same handle with the 2-handle which has $b$ as attaching circle. This produces Figure \ref{fig: intermidiat step}, in which there are depicted the 0-framed loops corresponding to the framed loops $y$ and $\alpha$. A final cancellation will produce Figure \ref{fig: mu}.
\end{proof}

\begin{lem}\label{lem: trivializing u}
Consider
\begin{equation}\label{eq: def torus tau}
    \mathcal{T}=\mathcal{T} (T_2):=(T_2)_{F_2,3_1}\subseteq \nu T_2\subseteq M_K,
\end{equation}
where $(T_2)_{F_2,3_1}$ denotes the 2-torus obtained by Construction \ref{const: nullhomologus tori} applied to the framed 2-torus $(T_2,F_2)\subseteq M_K$ and $3_1\subseteq S^3$ is the trefoil knot.
    There exists a diffeomorphism of triples
\begin{equation}\label{eq: diffe triples}
    (M_K\#_{T^2}\widehat{N},\widehat{x\times b},m)\cong(M_K\# S^4,\mathcal{T}, u),
\end{equation}
where the connected sum $M_K\# S^4$ is performed away from $\nu T_2$, and $u\subseteq S^4\bslash D^4$ is the unknot. The framing of $u$ induced by (\ref{eq: diffe triples}) is the spin framing. 
\end{lem}
\begin{proof}
First, we use Lemma \ref{lem: Identifing N hat} to obtain the diffeomorphism of triples
\begin{equation}\label{eq: Lemma triv u first step}
    (M_K\#_{T^2}\widehat{N},\widehat{x\times b},m)\cong\left((M_K\#_{T^2}\left(S^1\times S^1\times S^2\right), S^1\times\mu, \{1\}\times\overline{\mu} \right),
\end{equation} 
where the fiber sum identify $T_2\subset M_K$ with $S^1\times S^1\times\{1\}\subseteq S^1\times S^1\times S^2$. 
We can use the identification given by $F_2:\nu T_2\cong S^1\times S^1\times D^2=(S^1\times S^1\times S^2)\bslash\nu (S^1\times S^1\times\{1\})$ to simplify the right side of \ref{eq: Lemma triv u first step};
\begin{equation}
    (M_K\#_{T^2}\widehat{N},\widehat{x\times b},m)\cong\left(M_K, F_2^{-1}(S^1\times\mu), F_2^{-1}(\{1\}\times\overline{\mu}) \right).
\end{equation}
The loop $\mu\subseteq S^1\times D^2=(S^1\times S^2)\bslash\nu (S^1\times\{1\})$ bounds a punctured genus one surface. We can find this surface in Figure \ref{fig: mu}; $\mu$ bounds a 2-disk with two intersection points with the attaching circle of the 2-handle,  the required surface is obtained connecting those points with a tube. 
By looking at that surface, we can see that the framed loop $\overline{\mu}\subseteq (S^1\times D^2)\bslash\mu$ is isotopic to the band sum of a (trivially framed) unknot and two (opposite oriented) copies of $S^1\times\{p\}\subseteq\partial (S^1\times D^2)$. 
Therefore, the framed loop $F_2^{-1}(\{1\}\times \overline{\mu})\subseteq \nu T_2\bslash F_2^{-1}(S^1\times\mu)$ is the band sum of (trivially framed) unknot $u$ and two (opposite oriented) copies of $F_2^{-1}(\{p\}\times (S^1\times\{p\}))\subseteq\partial \nu T_2$, which are nullhomotopic in $M_K\bslash\nu T_2$. By Lemma \ref{lem: ConSumOfLoop and Disk} the framed loop $F_2^{-1}(\{1\}\times \overline{\mu})$ is isotopic to a spin framed unknot $u$ with an isotopy supported in the complement of $F_2^{-1}(S^1\times\mu)$, therefore
\begin{equation}
    (M_K\#_{T^2}\widehat{N},\widehat{x\times b},m)\cong\left(M_K\#S^4, F_2^{-1}(S^1\times\mu), u \right).
\end{equation}
Finally, the 2-torus $F_2^{-1}(S^1\times\mu)$ can be identified with $(T_2)_{F_2,3_1}$, since it is defined as in Construction \ref{const: nullhomologus tori} applied to the 2-torus $T_2\subseteq M_K$ with its framing $F_2$ and using the trefoil knot $3_1\subseteq S^3$, see Figure \ref{fig: knot Kn} with $n=1$.
\end{proof}

\begin{lem}
\label{lem: internal final step: remove K}
    There exists a diffeomorphism of couples
\[(M_K\#_{T^2}\widehat{N}_m,\widehat{x\times b})\cong(M_K\# (S^2\times S^2),\mathcal{T}),\]
where the connected sum $M_K\# (\Stab)$ is performed away from $\mathcal{T}\subseteq\nu T_2$ defined as in (\ref{eq: def torus tau}). 
\end{lem}
\begin{proof}
By Lemma \ref{lem: trivializing u} we have the following diffeomorphisms of pairs:
\[(M_K\#_{T^2}\widehat{N}_m,\widehat{x\times b})\cong(M_K\# (S^4)^*,\mathcal{T} )\cong(M_K\# (\Stab),\mathcal{T} ),\] where $(S^4)^*$ is the manifold obtained from $S^4$ by doing of loop surgery along $u$ with respect to its framing. The framing of $u$ gives that $(S^4)^*$ is diffeomorphic to $\Stab$, and this explains the second diffeomorphism above.
\end{proof}
\begin{thm}\label{thm: Int unknotted 1}
Let $S_K=\Gamma_K(M,1,T_1,T_2,\phi_K)\subseteq M\# (\Stab)$ be the 2-sphere defined as in (\ref{eq spheres}) and $\mathcal{T}=\mathcal{T}(T_2)\subseteq\nu T_2\subseteq M_K\#(\Stab)$ be the 2-torus defined as in (\ref{eq: def torus tau}). 
    After a trivial internal stabilization the 2-sphere $S_K\subseteq M\# (\Stab)$ becomes smoothly equivalent to the 2-torus $\mathcal{T}\subseteq \nu T_2\subseteq M_K\bslash D^4\subseteq M_K\# (\Stab)$. 
    Moreover: 
    \begin{itemize}
        \item The 2-sphere $S_K$ becomes smoothly unknotted after two trivial internal stabilizations. 
        \item If the framing of $T_2\subseteq M_K$ is $0/1$ spin-compatible or $M$ is a non spin manifold, then $S_K$ becomes smoothly unknotted after one trivial internal stabilizations. 
        \item  If $T_1\subseteq M\bslash\nu T_2$ has a 1/0 spin-compatible framing, then after a trivial internal stabilization the 2-sphere $S_K\subseteq M\# (\Stab)$ becomes smoothly equivalent to the 2-torus $\mathcal{T}\subseteq \nu T_2\subseteq M\bslash D^4\subseteq M\# (\Stab)$. 
    \end{itemize}
    
\end{thm}
\begin{proof}
     By Lemma \ref{lem: internal step 1} the 2-sphere $S_K$ in $M\# (\Stab)$ becomes smoothly equivalent to $\widehat{x\times b}$ in $M_K\#_{T^2} \widehat{N}_m$ after an internal stabilization. This stabilization is trivial since the fundamental group of the complement of $S_K$ is infinite cyclic, see \cite[Lemma 3]{InternalStabilizationBaykurSunukian}. Lemma \ref{lem: internal final step: remove K} gives the first conclusion. \\
     Since the 2-torus $\mathcal{T}$ becomes smoothly unknotted after a trivial internal stabilization \cite[Theorem 2 and Section 3.6]{InternalStabilizationBaykurSunukian}, then $S_K$ becomes smoothly unknotted with two.\\
     If the framing of $T_2\subseteq M_K$ is $0/1$ spin-compatible or if  $M_K$ is a non spin manifold, then  the 2-torus $\mathcal{T}\subseteq\nu T_2\subseteq M_K\bslash D^4 \subseteq M_K\#(\Stab)$ is smoothly unknotted by Theorem \ref{thm: ExtStab nullhomologus 2-tori}.\\
    If $T_1\subseteq M\bslash\nu T_2$ has a 1/0 spin-compatible framing, then there exists a diffeomorphism of couples
\[(M_K\# (\Stab),\mathcal{T} ){\cong}(M\# (\Stab),\mathcal{T} )\]
by Corollary \ref{cor: stabilization rel boundary} which restricts to the identity on $\nu T_2$.
\end{proof}
\subsection{External stabilization and brunnian 2-links}
In this section we will produce an infinite exotic family of $g$-component brunnian 2-links in a simply connected closed 4-manifold.
In particular, we will prove that the construction of \cite[Theorem B]{(Un)knotted}, under an additional hypothesis, produces a family of brunnian 2-links and not only a brunnianly exotic family \cite[Definition 2]{(Un)knotted}. So, we will partially answer positively at \cite[Question D]{(Un)knotted}.

The main theorem of the section is the following and its proof will be given in Section \ref{OneIsEnough}.
\begin{thm}\label{thm: (un)linked stabilization and bruniannity}
Let $\Gamma_K=\Gamma_K(M,g,T_1,T_2,\phi_K)\subseteq M\#g(\Stab)$ be defined as in Construction \ref{constr: brunnianl exotic 2-links}.
If the framing of $T_1\subseteq M$ is $1/0$ spin-compatible, $T_2\subseteq M_K$ is an ordinary surface and its framing is 1/0 spin-compatible, then $\Gamma_K\subseteq M\#g(\Stab)$ is a brunnian 2-link and it becomes smoothly unlinked in  $(M\#g(\Stab))\#(\Stab)$, where the connected sum is performed away from $\Gamma_K$.
\end{thm}
Notice that if $T_i\subseteq M$ is an ordinary surface, combining Remark \ref{oss: torus framings and stabilization sets} and Proposition \ref{prop: Propriety of S} item \ref{prop: SS item 2} we get that $T_i$ admits a $1/0$ spin-compatible framing. If $M$ is a spin manifold, then the 2-torus $T_i$ is an ordinary surface.   If $T_2$ is an ordinary surface in $M$, then it is an ordinary surface in $M_K$. 
\begin{cor}\label{cor: unliked family bruniannity}
     Consider the exotic family of 2-links $\{\Gamma_K\}_{K\in\mathcal{K}}$
    defined as in \cite[Theorem A]{(Un)knotted} with 2-link group $G=F_g$ the free group with $g$ generators.
    If the framing of $T_1\subseteq M$ is $1/0$ spin-compatible, $T_2\subseteq M$ is an ordinary surface and its framing is 1/0 spin-compatible, then for every $K\in\mathcal{K}$ the 2-link $\Gamma_K$ is brunnian and it becomes smoothly unlinked in  $(M\#g(\Stab))\#(\Stab)$.
\end{cor}
\begin{proof}
    For any knot $K\in\mathcal{K}$ the 2-link $\Gamma_K\subseteq M\# g(\Stab)$ is defined as
    \[\Gamma_K:=\Gamma_K(M,g, T_1,(T_2,F_K), \phi_K),\] where we use the notation introduced in Construction \ref{constr: brunnianl exotic 2-links} and where $F_K$ denotes the framing of $T_2\subseteq M$ which depends on the chosen knot $K$. This framing is obtained starting from an initial framing $F:\nu T_2\to S^1\times S^1\times D^2$ of $T_2\subseteq M$. In \cite{(Un)knotted} the authors prove that there exists a homeomorphism $h_K: M_K\to M$ which restrict to a diffeomorphism of couples $(\nu T_2,T_2)\to(\nu T_2,T_2)$, and the framing $F_K$ is $F\circ h_K|_{\nu T_2}$. It follows that if the framing $F$ for $T_2\subseteq M$ is $1/0$ spin-compatible, then the framing $F_K$ for $T_2\subseteq M_K$ is $1/0$ spin-compatible. Now we are under the assumption of Theorem \ref{thm: (un)linked stabilization and bruniannity} and we can conclude the proof.
\end{proof}
\subsubsection{Strategy\label{sec: strategy external unlinked}}
The main obstruction to the pairwise smooth equivalence of collection 2-links $\{\Gamma_K\}_{K\subseteq S^3}$ arises from the diffeomorphism type of the collection manifolds $\{Z_{K,g}\}_{K\subseteq S^3}$. In particular, since our goal is to show that all these 2-links become unlinked after one stabilization with $S^2\times S^2$, we want to solve this obstruction and prove that $Z_{K,g}$ and $M\# g(S^1\times S^3)\# g(S^2\times S^2)$ are 1-stably diffeomorphic for any knot $K$. We will achieve this objective, with some restrictions, in Theorem \ref{thm: stabilizaion of Z_kg}. Solving this obstruction immediately leads to the 2-links $\Gamma_K$ being smoothly unknotted, as shown in Section \ref{OneIsEnough}. In Section \ref{sec: tecnical lemmas} we will give the details of the proof of Theorem \ref{thm: stabilizaion of Z_kg}.
The following is a schematic strategy for the proof of Theorem \ref{thm: stabilizaion of Z_kg}.\\
Starting from a doubly toroidal 4-manifold $M$, we perform the following steps:
\begin{enumerate}
    \item Prove that there exist two framed loops $x$ and $y$ in $N_g$ such that the 4-manifold $M_K\# (N_g)_{x,y}$ obtained from $M_K\# N_g$ by doing surgery along those loops is diffeomorphic to the manifold $Z_{K,g}\#(\Stab)$. This is archived in  Corollary \ref{cor: Cobordism with s2xs2 comapible torus}.
    \item Show that $(N_g)_{x,y}$ is diffeomorphic to $g(S^1\times S^3)\# (g+1) (S^2\times S^2)$. See Proposition \ref{prop: Ngxy diffeo}.
    \item Finally, show that $M_K\# (N_g)_{x,y}$ is diffeomorphic to $ M\# g(S^1\times S^3)\# (g+1)(S^2\times S^2)$.
\end{enumerate}
\subsubsection{One external stabilization is enough \label{OneIsEnough}}
The goal of this section is to prove our main result, Theorem \ref{thm: (un)linked stabilization and bruniannity},  as concisely as possible by postponing the proofs of some technical lemmas to the next section. The following  theorem is the only ingredient that we need to state before going into the proof of Theorem \ref{thm: (un)linked stabilization and bruniannity}.


\begin{thm}\label{thm: stabilizaion of Z_kg}
If the framing of $T_1\subseteq M$ is $1/0$ spin-compatible, $T_2\subseteq M_K$ is an ordinary surface and its framing is 1/0 spin-compatible, then the manifold $Z_{K,g}\#(\Stab)$ is orientation preserving diffeomorphic to $M\#g(S^1\times S^3)\#(g+1)(\Stab)$. 
\end{thm}

We will postpone the proof of Theorem \ref{thm: stabilizaion of Z_kg} to next sections and  we now see how to use this key ingredient to get our main result.

\begin{proof}[Proof of Theorem \ref{thm: (un)linked stabilization and bruniannity}]
We would like to prove first that $\Gamma_K\subseteq M\# g(\Stab)$ becomes smoothly unlinked in $(M\# g(\Stab))\#(\Stab)$, i.e. there exists a diffeomorphism of pairs
\begin{equation}
    (M\# g(\Stab)\#(\Stab),\Gamma_K)\cong(M\# g(\Stab)\#(\Stab),U_g),
\end{equation} where $U_g$ is the g-component unlink. This is equivalent to have the diffeomorphism of couples
\begin{equation}\label{eq: thm brunnian links objective}
(Z_{K,g}\#(\Stab),\mathcal{L})\cong(M\#g(S^1\times S^3)\# (g+1)(\Stab),\mathcal{L}_U),
\end{equation} where $\mathcal{L}_U$ is the disjoint union of $g$ loops of the form $S^1\times\{p\}\subseteq S^1\times S^3$ one in each $S^1\times S^3$ factor. This is the case since 2-link $\Gamma_K$ is by definition smoothly equivalent to the belt 2-spheres defined in the result of the surgery along the 1-link $\mathcal{L}$, meanwhile the unlink $U_g$ is the belt 2-spheres defined in the result of the surgery along $\mathcal{L}_U$ (independently of the framing of the components). 
By Theorem \ref{thm: stabilizaion of Z_kg} we have a diffeomorphism
\[\phi:Z_{K,g}\#(\Stab)\cong M\#g(S^1\times S^3)\# (g+1)(\Stab).\]
The two sets of loops $\mathcal{L}_U$ and $\phi(\mathcal{L})$ define two set of generators of $\pi_1$ after connecting each loop to a base point. Any element of the automorphism group of $\pi_1(g(S^1\#S^3))$ can be realized by a diffeomorphism, this holds since any Nielsen transformations can be realized and they generate the automorphism group \cite{nielsen1924isomorphismengruppe}. So, there exists a diffeomorphism \[\psi: (M\#g(S^1\times S^3)\# (g+1)(\Stab),\phi(\mathcal{L}))\to (M\#g(S^1\times S^3)\# (g+1)(\Stab),\mathcal{L}_U)\]
since homotopic loops are isotopic in any 4-manifold.
The diffeomorphism of pairs (\ref{eq: thm brunnian links objective}) is realized by the map $\psi\circ\phi$, and this concludes the first part of the proof.

We now prove that $\Gamma_{K,g}$ is a brunnian 2-link. The key fact here is that the g-component 2-link $\Gamma_{K,g}\subseteq M\#g(\Stab)$ minus a component is smoothly equivalent to $\Gamma_{K,g-1}\subseteq (M\#(g-1)(\Stab))\#(\Stab)$ (see the proof of \cite[Proposition 19]{(Un)knotted}) and we had just finished the proof that $\Gamma_{K,g-1}\subseteq (M\#(g-1)(\Stab))\#(\Stab)$ is smoothly unlinked.
\end{proof}

\subsubsection{Technical lemmas\label{sec: tecnical lemmas}}
The main objective of this section is to prove Theorem \ref{thm: stabilizaion of Z_kg}, i.e. that the manifolds $Z_{K,g}$ and $M\#g(S^1\times S^3)\#g(S^2\times S^2)$ are 1-stably diffeomorphic.
\begin{lem}\label{KT loop surgery}
    Let $N_y$ be the 4-manifold which is the result of loop surgery around the loop $\{p\}\times y$ in the Kodaira Thurston manifold $N$ with the framing as in Remark \ref{ossframing}.
    We have the following diffeomorphism of pairs 
\begin{equation}
    (N_{y},x\times y')\cong (S^1\times(S^1\times S^2),S^1\times u)\# (S^2\times S^2, S^2\times\{p\}),
\end{equation}
    where $y'$ is a parallel push off of $y\subseteq Y$ and $u$ denotes an unknot in $S^1\times S^3$. 
    Moreover, the framed loop $x\times\{p\}\subseteq N_y$ is sent to the  loop $S^1\times\{p\}\subseteq (S^1\times (S^1\times S^2))\bslash B^4\subseteq (S^1\times (S^1\times S^2))\#(\Stab)$ equipped with the product framing.
\end{lem}
\begin{proof}
    We can use the automorphism of $Y$ described in Figure \ref{fig:KT} to exchange the framed loops  $y$ and $b$, so we get a diffeomorphism of pairs
    \[(N_y,x\times y')\cong(N^*,x\times b'),\] where $N^*$ is the manifold obtained from $N$ as result of loop surgery along $b$ and $b'$ is a push off of $b\subseteq Y$. We use \cite[Lemma 7]{(Un)knotted} to conclude the existence of a diffeomorphism between $N_y$ and $(S^1\times S^1\times S^2)\# (S^2\times S^2)$, but to see where the 2-torus $x\times y$ goes we have to take a closer look the diffeomorphism described in that proof.

    By the Moishezon trick \cite{moishezon1977complex} (Theorem \ref{thm: mois}), we have the diffeomorphism of couples 
    \[(N^*,x\times b')\cong(\widehat{N}_m,x\times b'),\] where  $\widehat{N}_m$ is the manifold obtained from $\widehat{N}$ as result of loop surgery along $m$ which are described in Section \ref{sec: internal Unknotted}. By Lemma \ref{lem: Identifing N hat} there is a diffeomorphism
    \begin{equation}\label{eq: proof triples 1}
        (\widehat{N}, x\times b', m)\cong(S^1\times (S^1\times S^2),S^1\times u,\{p\}\times \bar\mu),
    \end{equation} where $\bar\mu$ is described in Figure \ref{fig: mu} and $u$ is a meridian of $\bar\mu$. The loop $\mu$  can be isotoped to an unknot by a crossing change, so the loop $\{p\}\times\bar\mu\subseteq\widehat{N}$ is smoothly isotopic to an unknot $\{p\}\times\bar{u}\subseteq S^1\times(S^1\times S^2)$ where $u$ and $\bar {u}$ makes a Hopf link in $S^1\times S^2$. So, (\ref{eq: proof triples 1}) becomes a diffeomorphism of triples
    \begin{equation}\label{eq: proof triples 2}
        (\widehat{N}, x\times b', m)\cong(S^1\times (S^1\times S^2),S^1\times u,\{p\}\times \bar u).
    \end{equation}
    Noticing that $\{p\}\times \bar u$ is a meridian of $S^1\times u$, we get 
     \begin{equation}\label{eq: proof triples 3}
         (\widehat{N}, x\times b', m)\cong(S^1\times (S^1\times S^2),S^1\times u,\emptyset)\#(S^4,S^2,\mu_{S^2}),
     \end{equation} where $S^2\subseteq S^4 $ is the unknotted 2-sphere in 4-space and $\mu_{S^2}$ is its meridian (framed with the spin framing). It follows that
     \begin{equation}
       (\widehat{N}_m,x\times b')\cong ((S^1\times (S^1\times S^2),S^1\times u)\#(\Stab,S^2\times\{p\}),
     \end{equation}
     since $((S^4)^*,S^2)$, the result of loop surgery of the pair $(S^4, S^2)$ along $\mu_{S^2}$, is diffeomorphic to $(\Stab,S^2\times\{p\})$.

     The loop $x\times\{p\}$ is a push off of $x\times y'\subseteq N_y$, so it is sent to a push off a curve $S^1\times \{q\}\subseteq (S^1\times u)\# (S^2\times \{p\})$.
\end{proof}
   
\begin{prop}\label{prop: s1xs1xs2 to s1xs3 s2xs2}
Let $x'$ be the loop $S^1\times \{p\}\subseteq S^1\times (S^1\times S^2)$ with any framing and $u$ be an unknot in $S^1\times S^2$. If $(S^1\times (S^1\times S^2))_{x'}$ denotes the 4-manifold which is the result of loop surgery along the framed loop $x'$, then there exists a diffeomorphism of pairs
\begin{equation}
    (S^1\times (S^1\times S^2))_{x'}, S^1\times u)\cong \left((S^1\times S^3)\#(\Stab),\emptyset\right)\#(S^4, T^2),
\end{equation} where $T^2$ is an unknotted 2-torus in the $S^4$ factor. 
\end{prop}

\begin{proof} 
 The 2-torus $S^1\times u$ is bounding an embedded $S^1\times D^2$ in $(S^1\times(S^1\times S^2))_{x'}$, moreover the map $\pi_1(S^1\times D^2)\to \pi_1((S^1\times(S^1\times S^2))_{x'}$ is the zero map. It follows that  $S^1\times u$ is smoothly unknotted, and that there exists a diffeomorphism of pairs 
 \[(S^1\times (S^1\times S^2))_{x'}, S^1\times u)\cong \left((S^1\times (S^1\times S^2))_{x'},\emptyset\right)\#(S^4, T^2).\]
 It is an exercise of Kirby calculus to prove that $(S^1\times (S^1\times S^2))_{x'}$ is diffeomorphic to $(S^1\times S^3)\#(\Stab)$, see \cite[Figure 3.2]{akbulut20164} for a handle diagram of $T^2\times S^2$.
\end{proof}

\begin{constr}[Standard 2-torus $(S^2\times\{p\})_{\#T^2}\subseteq\Stab$]
Let us define \emph{the standard 2-torus} $(S^2\times\{p\})_{\#T^2}\subseteq\Stab$ as $S^2\times\{p\}$ after one trivial stabilization, i.e. $(S^2\times\{p\}){\#T^2}\subseteq (\Stab)\# S^4$ where $T^2\subseteq S^4$ is the unknotted 2-torus.\end{constr}

\begin{prop}\label{prop: Ngxy diffeo}
Let $x\times y, x\times y'\subseteq N_g$ be defined as in Construction \ref{constr: long KT}.
Let $(N_g)_{x,y}$ be the 4-manifold obtained from $N_g$ as the result of two loop surgery along the framed loops $\widetilde{x},\{p\}\times y\subseteq N_g$, where $\widetilde{x}$ is the loop $x\times\{p\}$ framed with some framing. Notice that the framed 2-torus $x\times y'\subseteq N_g$ still defines a framed 2-torus $x\times y'$ in $(N_g)_{x,y}$.
    There exists a diffeomorphism of couples
     \begin{equation}\label{eq: Ngxy diffeo of couple}
        ((N_g)_{x,y},x\times y')\cong(g(S^1\times S^3)\#g(\Stab),\emptyset)\#(\Stab,S^2\times\{p\})_{\#T^2})
    \end{equation}
    where $(S^2\times\{p\})_{\#T^2}$ is the standard 2-torus contained in one of the $\Stab$ factors. Moreover, the induced framing of $(S^2\times\{p\})_{\#T^2}$ is $1/0$ spin-compatible.
\end{prop}
\begin{proof}
Notice first that from the diffeomorphism  (\ref{eq: Ngxy diffeo of couple}) it follows that $S^2\times\{p\}_{\# T^2}$ has an induced framing which is $1/0$ spin-compatible, since the framed 2-torus $x\times y'\subseteq(N_g)_{x,y}$ is $1/0$ spin-compatible.\\
    We will prove the existence of   (\ref{eq: Ngxy diffeo of couple}) by induction on $g$. 
    In the base case $g=1$ the diffeomorphism (\ref{eq: Ngxy diffeo of couple})  is obtained combining Lemma \ref{KT loop surgery} and Proposition \ref{prop: s1xs1xs2 to s1xs3 s2xs2}.\\
    Assume that (\ref{eq: Ngxy diffeo of couple}) holds for $g\in\N^+$ we will now prove that it holds for $g+1$. We have the following two diffeomorphisms
    \begin{equation}
     ((N_{g+1})_{x,y},x\times y')\cong\left(N_g\#_{T^2} (N_1)_{x,y},x\times y')\cong ((N_g\#_{T^2} \left((S^1\times S^3)\#2(\Stab)\right),(S^2\times\{p\})_{\#T^2})\right),   
    \end{equation}
    where the first one is  the definition of $N_{g+1}$ and the second is obtained using the diffeomorphism of the case $g=1$. The fiber sum on the right identifies $(x\times y)\subseteq N_g$ with a parallel copy of $(S^2\times\{p\})_{\#T^2})\subseteq\Stab$.
    We can move away the factor $S^1\times S^3$ and the $\Stab$ factor that does not contain the 2-torus $(S^2\times\{p\})_{\#T^2})$. Moreover, we can identify $(S^2\times\{p\})_{\#T^2})$ with a parallel copy $x\times y'$ of $(x\times y)\subseteq N_g$, i.e.
   \begin{equation}
     ((N_{g+1})_{x,y},x\times y')\cong (\left(N_g\#_{T^2} (\Stab)\right)\#(\Stab)\#(S^1\times S^3),x\times y').  
    \end{equation}
    
    We apply Corollary \ref{cor: Cobordism with s2xs2 comapible torus} to replace the generalized fiber sum and the $\Stab$ factor on the right with a connected sum and two loop surgeries
    \begin{equation}
     ((N_{g+1})_{x,y},x\times y')\cong ((N_g)_{x,y}\# (\Stab)\#(S^1\times S^3),x\times y')  
     \end{equation}
     and we can conclude from the inductive hypotheses.
\end{proof}

\begin{proof}[Proof of Theorem \ref{thm: stabilizaion of Z_kg}]
    By definition $Z_{K,g}$ is the fiber sum $M_K\#_{T^2} N_g$ performed along the framed 2-tori $T_2\subseteq M_K$ and $(x\times y)\subseteq N_g$, so, by Corollary \ref{cor: Cobordism with s2xs2 comapible torus} the manifold $Z_{K,g}\#(\Stab)$ is diffeomorphic to $M_K\# (N_g)_{x,y}$. It is sufficient now to apply Corollary \ref{cor: stabilization rel boundary}.
\end{proof}

\section{Nullhomotopic 2-tori and 2-spheres\label{sec: nullhomologus surfaces}}
In this section we will examine the stabilization behavior of the 2-sphere $\Sigma\subseteq\overline{K}$ implicitly described in \cite{fintushel1994fake}, as pointed out by \cite[end of Section 2]{ray2017four}, the 2-tori $T'_n\subseteq M$ and the 2-spheres $S'_n\subseteq M\#(\Stab)$ defined in \cite[Theorem A]{torres2020Unknotted}. All these surfaces are nullhomologus and they have fundamental group of the complement which is infinite cyclic and generated by the meridian. These surfaces are topologically unknotted and smoothly knotted under opportune conditions \cite{fintushel1994fake,torres2020Unknotted}.\\
Most of the result in this section are obtained as an application or a variation of the techniques of Section \ref{sec: unlinked}.
\subsection{External stabilization of $\Sigma\subseteq\overline{K}$\label{sec: stabilization FintStern}}
From the proof of Theorem \ref{thm: (un)linked stabilization and bruniannity} we get the following proposition.
\begin{prop}\label{prop: extstabilization nullhom spheres}
Let $\Sigma\subseteq X$ be a smooth 2-sphere with trivial normal bundle in a simply connected 4-manifold $X$. Let $B$ be a simply connected 4-manifold. The surface $\Sigma$ is smoothly unknotted in $X\#B $ if and only if the manifold $X(\Sigma)\#B$ has a smooth $S^1\times S^3$ summand, where $X(\Sigma)$ is obtained from $X$ as the result of 2-sphere surgery along $\Sigma$.
\end{prop}

Now we will define and study   the 2-sphere $\Sigma\subseteq \overline{K}$  implicitly defined by \cite{fintushel1994fake}.
\begin{constr}\label{Constr fs: 2-sphere}
    Let $X$ be the 4-manifold described in \cite{fintushel1994fake}, it can be described as a \emph{self generalized fiber sum} of the $K3$ surface along two smoothly embedded and framed 2-tori $T_0,T_1\subseteq K3$ each one contained in a nucleus $N_2$, see \cite[Section 5]{fintushel1994fake}. 
That is
\begin{equation}\label{eq: fintushel stern fake}
    X\cong \left(K3\bslash(\nu T_0\cup\nu T_1)\right)\cup_g \left(T^2\times \partial D^2\times [0,1]\right),
\end{equation}
where the gluing map $g$ is defined by the framings of $T_0$ and $T_1$. This manifold has infinite cyclic fundamental group, the result of loop surgery along a generator of the fundamental group produces  the simply connected 4-manifold $\overline{K}$ described in \cite[Section 5]{fintushel1994fake}. 
Let $\Sigma\subseteq \overline{K}$ be the  belt 2-sphere of a loop surgery. This is a nullhomotopic 2-sphere with infinite cyclic fundamental group of the complement.
\end{constr}

The manifold $\overline{K}(\Sigma)$ of Proposition \ref{prop: extstabilization nullhom spheres} is the manifold $X$.
The cobordism argument presented in the proof of Proposition \ref{Baykur} can be applied to self generalized fiber sums, and implies that
\begin{equation}
    \overline{K}(\Sigma)\#(\Stab)\cong (K3\#(S^1\times S^3))_{\alpha,\beta},
\end{equation}
where $(K3\#(S^1\times S^3))_{\alpha,\beta}$ denotes the manifold obtained from $K3\#(S^1\times S^3)$ by performing loop surgery along some simple framed loops $\alpha,\beta$. Since the manifold $X(\Sigma)$ has infinite cyclic fundamental group \cite{fintushel1994fake}, then $\alpha,\beta$ are nullhomotopic and they can be isotoped to stay in $K3$. Therefore, 
\begin{equation}
    \overline{K}(\Sigma)\#(\Stab)\cong (K3_{\alpha,\beta})\#(S^1\times S^3),
\end{equation}
so the manifold $\overline{K}(\Sigma)\#(\Stab)$ has a smooth $S^1\times S^3$ summand.
We have proved the following theorem.
\begin{thm}\label{thm: External fintushel sphere}
    Let $\Sigma\subseteq \overline{K}$ be as in Construction \ref{Constr fs: 2-sphere}. The 2-sphere $\Sigma\subseteq \overline{K}\#(\Stab)$ is smoothly unknotted, where the connected sum is performed away from $\Sigma$.
\end{thm}
\begin{oss}\label{rem: sigma is smoothly knotted}
     Let $\Sigma\subseteq \overline{K}$ be as in Construction \ref{Constr fs: 2-sphere}.
    The 2-sphere $\Sigma\subseteq \overline{K}$ is topologically unknotted and it is not smoothly unknotted. 
\end{oss}
\begin{proof}
The 2-sphere $\Sigma\subseteq \overline{K}$ is topologically unknotted by \cite[Theorem 7.2]{SunukjianNathan}, since $\pi_1(\overline{K}\bslash \Sigma)$ is infinite cyclic and the inequality $b_2(\overline{K})\geq |\sigma(\overline{K})|+6$ holds.
The manifold $\overline{K}(\Sigma)$ does not contain a $S^1\times S^3$ summand \cite{fintushel1994fake}, therefore by Proposition \ref{prop: extstabilization nullhom spheres} the surface $\Sigma\subseteq\overline{K}$ is not smoothly unknotted. For the convenience of the reader we now present a different proof, due to Rafael Torres, to show that $\overline{K}(\Sigma)$ is irreducible. \\
The manifold $\overline{K}(\Sigma)$ is the \emph{self generalized fiber sum }of two 2-tori in $K3$. Those 2-tori are fibers of $K3$, and are Lagrangian (given a symplectic structure on $K3$). Since they are homologically essential, the symplectic structure can be perturbed so that both of them become symplettic submanifold (and symplettomorphic), see \cite{gompf1995new}. 
Gluing as in (\ref{eq: fintushel stern fake}) yields a spin 4-manifold with infinite cyclic fundamental group. It is symplectic by a result of Gompf \cite{gompf1995new} and \cite{mccarthy1994symplectic}.
It is minimal, since it is spin. It is irreducible by a result of Hamilton-Kotschick \cite{hamilton2006minimality}, since the infinite cyclic group is a residually finite group. 
This argument (and this 4-manifold) has already appeared in \cite{torres2011geography}.
\end{proof}

\subsection{Stabilization of the nullhomologus 2-tori $T_n$ and 2-spheres $S_n'$.\label{sec: Stab Torres unknotted}}
Since \cite{(Un)knotted} is based on Torres's work \cite{torres2020Unknotted}, the techniques and construction presented in Section \ref{sec: unlinked} can be adapted to cover the study of the stabilization problem of the 2-spheres $S_n$ and also the 2-tori $T_n'$.\\
The construction of the surfaces $S_n$ and $T_n'$ can be obtained from Section \ref{sec: setting unlinked} with the following modification:
\begin{itemize}
    \item The 4-manifold $M$ has an embedded 2-torus $T\subseteq M$ with trivial normal bundle, a chosen framing  and simply connected complement. It is not necessary that $M$ is doubly toroidal. In general, we will replace in any construction the 2-torus $T_2$ with $T$.
    \item No knot $K\subseteq S^3$ is involved in the construction, in particular, we do not use knot surgery.
    \item Here $g=1$, the Kodaira Thurston manifold $N=N_1$ is replaced with the infinite family $\{N(n)\}_{n\in\N}$. The manifold \begin{equation}
    N(n):=S^1\times Y(n),\end{equation}
where $Y(n)$ is the boundary of the 4-manifold described in Figure \ref{fig:KT}. The framed 2-torus $x\times y\subseteq N(n)$ can be defined as in Section \ref{sec: long KT} replacing $Y$ with $Y(n)$ and similarly for any other submanifolds.
\item The manifold $Z_{K,g}$ is replaced by $Z(n)$ which is  the generalized fiber sum between $M$ and the manifold $N(n)$ along the 2-tori $T\subset M$ and $x\times y \subset  N(n)$. 
\begin{equation}    Z(n):=(M\bslash\nu T)\cup (N(n)\bslash \nu(x\times y)).\end{equation}
\item The derived manifolds $\widehat{N(n)},N(n)^*,Z(n)^*$ are produced  with submanifolds $x\times y,\widehat{x\times b}\subseteq \widehat{N(n)}$ and $\Gamma(n)\subseteq N(n)^*\bslash\nu (x\times y)\subseteq Z(n)^*$.
\end{itemize}

\begin{constr}[Nullhomologus 2-spheres and 2-tori \cite{torres2020Unknotted}]\label{constr: nullhomo 2-sphere torres}\label{constr: nullhomo 2-tori torres}
Let $M$ be a  smooth 4-manifold and $T\subseteq M$ be a embedded 2-torus with trivial normal bundle, a chosen framing and simply connected complement.
The 2-sphere $S_n\subseteq M\#(\Stab)$ is defined as follows:
\begin{equation}
    S_n=S_n(M,T,n ,\phi_n):=\phi_n(\Gamma(n))\subseteq M\#(\Stab),
\end{equation}
    where $\phi_n: Z(n)^*\to M\#(\Stab)$ is a diffeomorphism. 
The 2-torus $T'_n\subseteq\nu T\subseteq M$ is defined as follows:
\begin{equation}
    T'_n=T_n'(M,T,n ,\phi_n):=\phi_n(\widehat{x\times b })\subseteq \nu T\subseteq M,
\end{equation}
    where $\phi_n: M\#_{T^2}\widehat{N(n)}\to M$ is a diffeomorphism which restrict to the identity on $M\bslash\nu T$ and $M\#_{T^2}\widehat{N(n)}$ denotes the 4-manifold obtained as generalized fiber sum along the framed 2-tori $T\subseteq M$ and $x\times y\subseteq\widehat{N(n)}$.
\end{constr}

\subsubsection{Stabilization of $Z(n)$ and the external stabilization of $S_n$}
It is possible to adapt the entirety of Section \ref{sec: tecnical lemmas} to the study of the manifolds $Z(n)$ instead of the manifolds $Z_{K,g}$. We should replace every appearance of $N$ with $N(n)$, $T_2$ with $T$, $M_K$ with $M$, $Z_{K,g}$ with $Z(n)$, $Y$ with $Y(n)$ and set $g=1$. By doing so we get the following theorem.
\begin{thm}\label{thm: External stab of Z(n)}
If $T\subseteq M$ is an ordinary surface and its framing is 1/0-spin-compatible, then the manifold $Z(n)\#(\Stab)$ is orientation preserving diffeomorphic to $M\#(S^1\times S^3)\#2(\Stab)$. 
\end{thm}
As a consequence, the Torres's 2-sphere $S_n$ is smoothly unknotted after one external stabilization with $\Stab$. 
\begin{thm}\label{thm: externa stab 2-sphers (un)knotted torres}
Let $S_n=S_n(M,T,n ,\phi_n)\subseteq M\#(\Stab)$ be defined as in Construction \ref{constr: nullhomo 2-sphere torres}.
If $T\subseteq M$ is an ordinary surface and its framing of  is 1/0 spin-compatible, then $S_n\subseteq M\#(\Stab)$ becomes smoothly unknotted in  $(M\#(\Stab))\#(\Stab)$, where the connected sum is performed away from $S_n$.
\end{thm}
\begin{proof}
It is sufficient to combine Proposition \ref{prop: extstabilization nullhom spheres} and Theorem \ref{thm: External stab of Z(n)}.
\end{proof}

\subsubsection{Identification of 2-tori $T'_n$ and stabilizations}
The following theorem shows how to identify the Torres's 2-tori $T_i'$ with the 2-tori of \cite{Hoffman_Sunukjian_2020}.
\begin{thm}\label{thm: identification Torres 2-tori}
    Let $T_n'=T_n'(M,T,n ,\phi_n)\subseteq \nu T\subseteq M$ be defined as in Construction \ref{constr: nullhomo 2-tori torres}.
    Let 
\begin{equation}
    \mathcal{T}_n=\mathcal{T}_n (T):=(T)_{F,K_n}\subseteq \nu T\subseteq M,
\end{equation}
where $(T)_{F,K_n}$ denotes the 2-torus obtained by Construction \ref{const: nullhomologus tori} applied to the framed 2-torus $(T,F)\subseteq M$ and knot $K_n\subseteq S^3$ depicted in Figure \ref{fig: knot Kn}. The 2-tori $T_n',\mathcal{T}_n\subseteq\nu T \subseteq M$ are smoothly equivalent relative to $M\bslash\nu T$.
In particular, the 2-torus $T_n'\subseteq M$ is smoothly unknotted after one trivial internal stabilization or one external stabilization with any manifold in the stabilization set $ \Ss(\mathcal{T}_n,\gamma)$, where $\gamma$ is the framed curve described in equation (\ref{eq: definition framing induced by the restriction}).
\end{thm}
\begin{figure}
    \centering
    \includegraphics[width=0.2\linewidth]{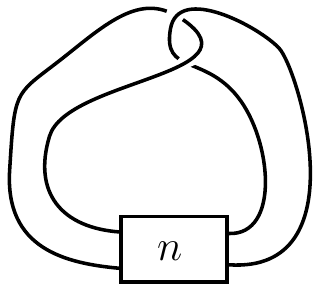}
    \caption{A knot diagram for the family of knots $K_n\subseteq S^3$, where the box labelled $n$ indicates $n$ full twist.}
    \label{fig: knot Kn}
\end{figure}
\begin{proof}[Sketch of proof] 
    Replacing every occurrence of $N$ with $N(n)$ in the proof of Lemma \ref{lem: Identifing N hat}, we obtain a diffeomorphism of triples 
    \begin{equation}
          \phi:(\widehat{N(n)},x\times y,\widehat{x\times b})\to (S^1\times(S^1\times S^2)), S^1\times (S^1\times \{1\}), S^1\times\mu_n),  \end{equation} 
          where $\mu_n$ is loop depicted in Figure \ref{fig: mumu} and the framing induced by $\phi$ on the 2-torus $S^1\times (S^1\times \{1\})$ is given by the product framing.
    This gives the diffeomorphism of couples
    \begin{equation}\label{eq: diffeo in the proof as in lemma}(M\#_{T^2}\widehat{N(n)},\widehat{x\times b})\cong(M\bslash\nu T)\cup_{F|_\partial} (S^1\times (S^1\times D^2)), S^1\times \mu_n)\cong (M,F^{-1}(S^1\times \mu_n)),\end{equation} which restricts to the identity on $M\bslash\nu T$. Lastly, notice that $F^{-1}(S^1\times \mu_n)$ is exactly $\mathcal{T}_n$, see construction \ref{const: nullhomologus tori}.
    The last part of the statement is implied by \cite[Theorem 2]{InternalStabilizationBaykurSunukian} and Theorem \ref{thm: ExtStab nullhomologus 2-tori}.
\end{proof}
\subsubsection{Internal stabilization of the 2-spheres $S_n$}

It is possible to adapt the entirety of Section \ref{subsec: follwing the surfaces} to the study of the 2-spheres $S_n\subseteq M \#(\Stab)$. We should replace every appearance of $N$ with $N(n)$, $T_2$ with $T$, $M_K$ with $M$, $\mu$ with $\mu_n$ and $\mathcal{T}$ with $\mathcal{T}_n$. 
By doing so we get the following theorem.

\begin{thm}\label{thm: intern unknotte 2spheres torres}
Let $S_n=S_n(M,T,n,\phi_n)\subseteq M\# (\Stab)$ be the 2-sphere defined as in Construction \ref{constr: nullhomo 2-sphere torres} and $\mathcal{T}_n=\mathcal{T}_n(T)\subseteq\nu T\subseteq M\#(\Stab)$ be the 2-torus defined in Theorem \ref{thm: identification Torres 2-tori}. 
    After a trivial internal stabilization the 2-sphere $S_n\subseteq M\# (\Stab)$ becomes smoothly equivalent to the 2-torus $\mathcal{T}_n\subseteq \nu T\subseteq M\bslash D^4\subseteq M\# (\Stab)$. 
    Moreover: 
    \begin{itemize}
        \item the 2-sphere $S_n\subseteq M\# (\Stab)$ becomes smoothly unknotted after two trivial internal stabilizations. 
        \item if the framing of $T\subseteq M$ is $0/1$ spin-compatible or $M$ is a non spin manifold, then $S_n\subseteq M\# (\Stab)$ becomes smoothly unknotted after one trivial internal stabilizations.
    \end{itemize}
    \end{thm}

\bibliographystyle{mynum2}
\bibliography{sample}
\end{document}